\documentclass[bibliography=totoc]{article}

\usepackage{graphics,graphicx,color,xcolor,caption}
\usepackage{amsmath,amsfonts,amssymb,latexsym,amsthm}
\usepackage{url}

\usepackage{wasysym}
\usepackage{mathrsfs}
\usepackage{amssymb }
\usepackage[shortlabels]{enumitem}

\usepackage[nottoc]{tocbibind}

\usepackage[colorlinks=false]{hyperref}

\newtheorem{lemma}{Lemma}[section]
\newtheorem{proposition}[lemma]{Proposition}
\newtheorem{theorem}{Theorem}

\newtheorem{corollary}[lemma]{Corollary}

\theoremstyle{definition}
\newtheorem{remark}[lemma]{Remark}
\newtheorem{example}[lemma]{Example}

\newcommand{\eps}{{\varepsilon}}
\newcommand{\C}{{\mathbb C}}

\newcommand{\R}{{\mathbb R}}
\newcommand{\Z}{{\mathbb Z}}
\newcommand{\RP}{{\mathbb {RP}}}
\newcommand{\CP}{{\mathbb {CP}}}

\newcommand{\g}{\gamma}
\newcommand{\G}{\Gamma}
\newcommand{\e}{\varepsilon}
\newcommand{\tr}{\mathrm{tr}}

\newcommand{\RR}{{\mathrm I}}

\newcommand{\mn}{\medskip\noindent}

\newcommand{\SLt}{\mathrm{SL}_2(\R)}
\newcommand{\GLt}{\mathrm{GL}_2(\R)}

\newcommand{\slt}{\mathfrak{sl}_2(\R)}

\newcommand{\sB}{self-B\"acklund\ }

\newcommand{\note}[1]{
 \marginpar
 {\scriptsize\color{blue}#1}}
 
 \newcommand{\gil}[1]{
 \marginpar
 {\scriptsize\color{red}#1 -GB}}

\renewcommand{\Im}{\mathrm{Im}}

\newcommand{\be}{\begin{equation}}
\newcommand{\ee}{\end{equation}}
\newcommand{\se}{\begin{enumerate}[wide = 0pt, leftmargin = 1em]}

\title{
Self-B\"acklund curves in centroaffine geometry and Lam\'e's equation}

\author{Misha Bialy\footnote{
School of Mathematical Sciences,
Tel Aviv University, Israel;
bialy@post.tau.ac.il
}
\and
Gil Bor\footnote{
CIMAT, A.P. 402, Guanajuato, Gto. 36000, Mexico; 
gil@cimat.mx
}
\and
Serge Tabachnikov\footnote{
Department of Mathematics,
Penn State University, USA;
tabachni@math.psu.edu}
}
\date{}

\begin{document}

\maketitle

\begin{abstract}
Twenty five years ago U. Pinkall discovered that the Korteweg-de Vries equation can be realized as an evolution of curves in centoraffine geometry. Since then, a number of authors interpreted various properties of KdV and its generalizations in terms of centoraffine geometry.  In particular, the B\"acklund transformation of the Korteweg-de Vries equation can be viewed as a relation between centroaffine curves.

Our paper concerns self-B\"acklund centroaffine curves. We describe general properties of these curves and  provide a detailed description of them in terms of elliptic functions. Our work is a centroaffine counterpart to the study done by F. Wegner of a similar problem in Euclidean geometry, related to Ulam's problem of describing the (2-dimensional) bodies that float in equilibrium in all positions and to the bicycle kinematics.

We also consider a discretization of the problem where curves are replaced by polygons. This is  related to discretization of KdV and the cross-ratio dynamics on ideal polygons.
\end{abstract}

\tableofcontents

\section{Introduction} \label{sect:intro}

%
%

The motivation for this work is the interpretation of the Korteweg-de Vries equation in terms of centroaffine geometry. This growing body of work started with U. Pinkall's paper \cite{Pin}, see \cite{CIM,FuKu1,FuKu2,TW} for a sampler.

 In \cite{Tab18}, the B\"acklund transformation of the KdV equation is interpreted as a relation between centroaffine curves. We start with a very brief description of this approach to KdV.

Let $\g(t)$ be a parametrized smooth curve in the affine plane with a fixed area form. The curve is {\em centroaffine} if the Wronski determinant is constant: $[\g(t),\g'(t)]=1$ for all $t\in\R$. The group $\SLt$ acts on centroaffine curves, and we shall also consider the moduli space of such curves. 

%
Unless specified otherwise, we assume that the curves are $\pi$-anti-periodic: $\g(t+\pi)=-\g(t)$ for all $t$.
That is, the curve is closed, centrally symmetric and  $2\pi$-periodic (the last condition, if not satisfied by a centrally symmetric centroaffine curve, can be arranged by an appropriate rescaling.)
%

%

The rationale for assuming that the curves are centrally symmetric is as follows. An orientation preserving diffeomorphism of $\RP^1$ admits a unique area preserving and homogeneous of degree 1 lifting to a diffeomorphism of the punctured plane. 
%
%
The image of the 
 unit circle under such a diffeomorphism is a centrally symmetric star-shaped curve, and projectively equivalent diffeomorphisms correspond to $\SLt$-equivalent curves. See \cite{OT97} for details.
 
 Our results can be extended to non-centrally symmetric curves, but we do not dwell on it in this paper.

%
Given a centroaffine curve, one has $\g''(t)=p(t) \g(t)$ where $p$ is a $\pi$-periodic potential function of the Hill operator $-d^2/dt^2 + p(t)$. In the language  of centroaffine geometry, $p$ is  
the {\it centroaffine curvature} of the curve $\g$ (alternatively, some authors call $-p$ the centroraffine curvature, but we shall  adopt the plus sign convention).

For example, $\g(t)=(\cos t,\sin t)$ has $p(t)=-1$. This unit circle, and its $\SLt$ images, are trivial examples of centroaffine curves. We refer to these curves as centroaffine conics.

A tangent vector  to a 
centroaffine curve $\g(t)$, in the space of $\pi$-anti-periodic centro affine curves, is given by  a vector field  along it of  the form $g(t)\gamma(t)+f(t)\gamma'(t)$, where $f,g$ are $\pi$-periodic. Taking the  derivative of  the centroaffine condition $[\gamma,\gamma']=1$ with respect to this vector field we obtain $f'+2g=0$.  Thus such a  vector field has the form 
\begin{equation} \label{eq:tang}
V_f:= -\frac{1}{2} f'(t) \g(t) + f(t) \g'(t),
\end{equation}
where $f$ is a $\pi$-periodic function. Pinkall observed in \cite{Pin} that the evolution of the curves $\g(t)$ with the potential function $p(t)$ under the vector field $V_p$ is a centroaffine version of the Korteweg-de Vries equation: the potential evolves according to the equation
%
%
$$
\dot p = -\frac{1}{2} p''' + 3p'p
$$
(where dot is the time derivative).

%
%

\begin{figure}[ht]
\centering
\includegraphics[width=.45\textwidth]{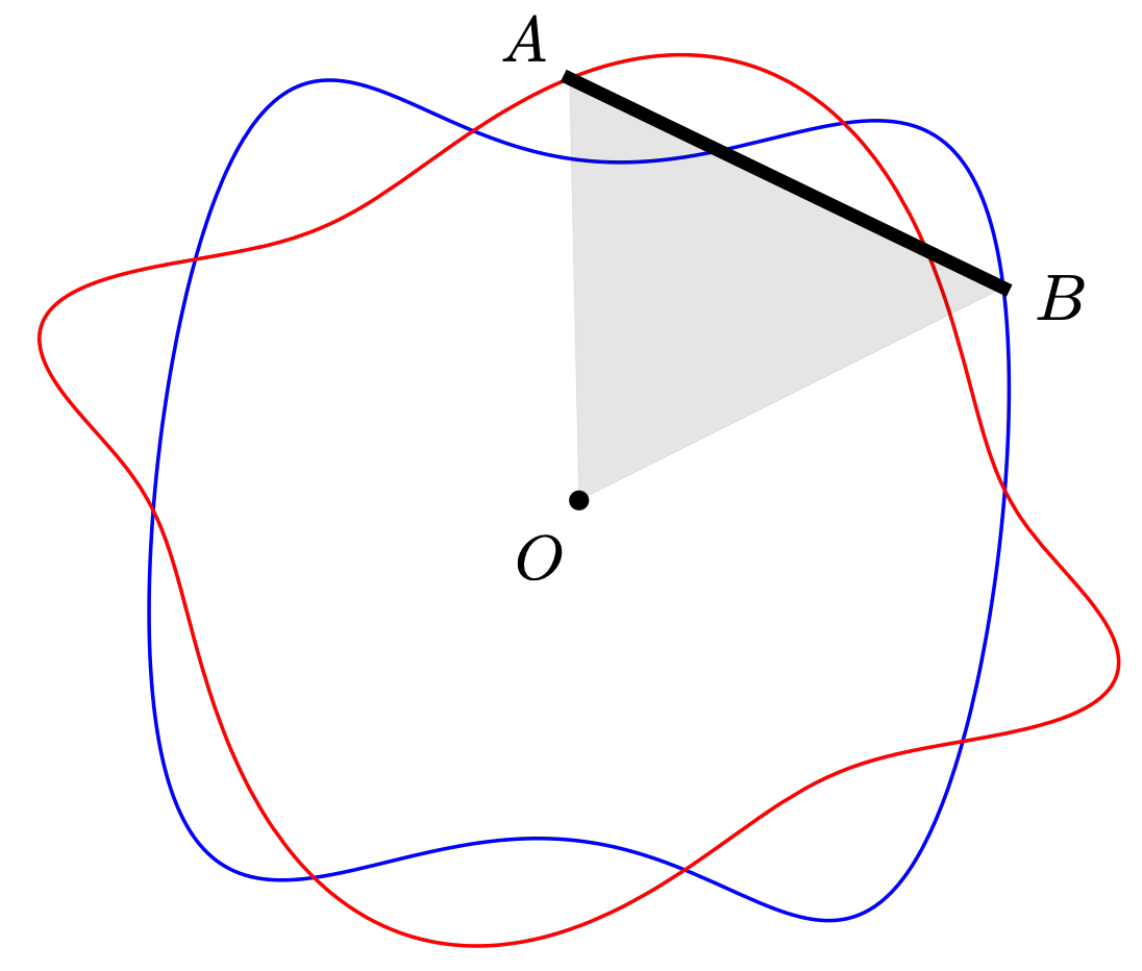}
\caption{B\"acklund transformation: as the end points of the line segment $AB$  trace the two curves, $OA$ and  $OB$ sweep area with the same rate and the area  of the shaded triangle $OAB$ remains constant. }
\label{fig:c}
\end{figure}

We say that two centroaffine curves, $\g(t)$ and $\delta(t)$, are {\it $c$-related} if $[\g(t),\delta(t)]=c$ for all $t$. See Figure \ref{fig:c}. It is shown in \cite{Tab18} that this relation is a geometric realization of the B\"acklund transformation for the KdV equation. 

In this paper we are mostly 
interested in {\it self-B\"acklund} centroaffine curves, the curves $\g(t)$ for which there exists $\alpha \in (0,\pi)$ and a constant $c$ such that 
\begin{equation}\label{eq:selfbacklund}
[\g(t),\g(t+\alpha)]=c \ \ {\rm for\  all}\  t.
\end{equation}
For example, the centroaffine conics are self-B\"acklund for every choice of $\alpha$ with $c=\sin\alpha$. To exclude trivial cases, we assume that $c\neq 0$. We call $\alpha$ in equation \eqref{eq:selfbacklund} the {\it rotation number} of a self-B\"acklund curve.
See Figure \ref{fig:weg} for examples of self-B\"acklund curves. 

\begin{figure}[ht]
\centering
\includegraphics[width=.8\textwidth]{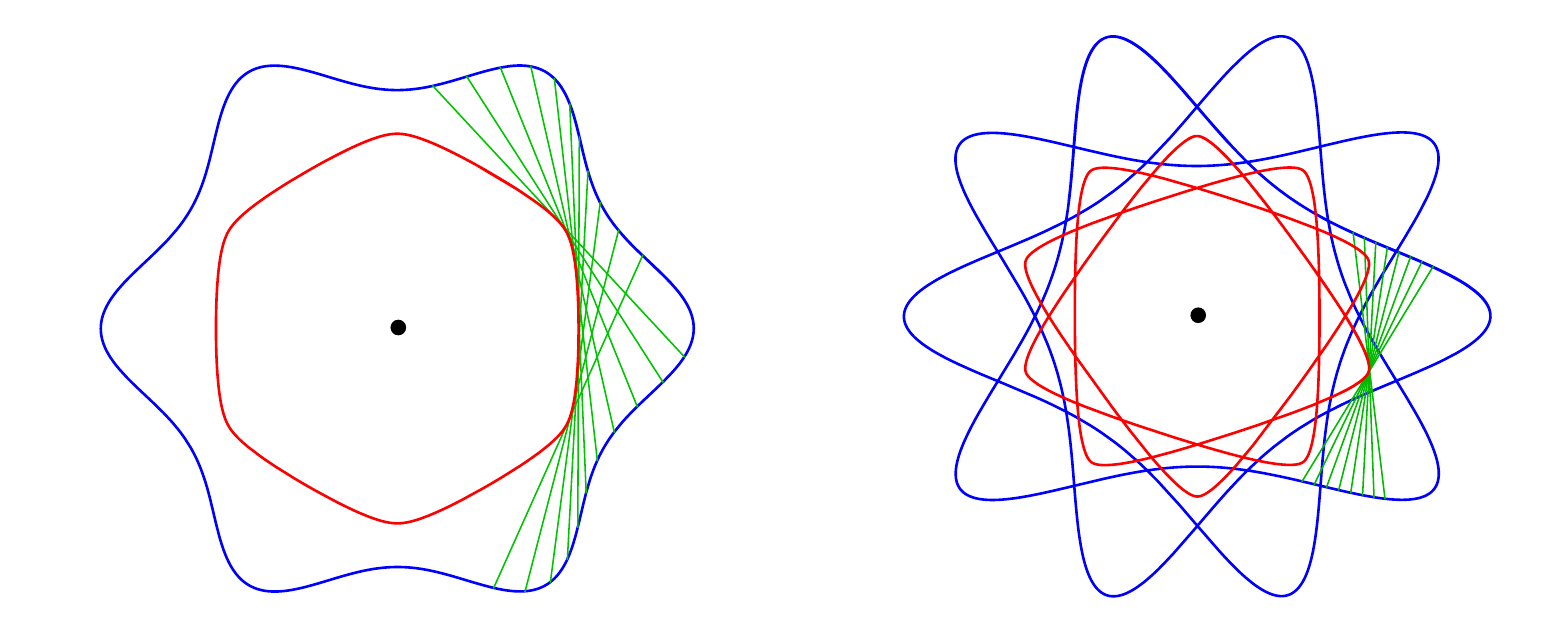}
\caption{Self-B\"acklund curves (blue), with winding numbers 1 (left) and 3 (right). A line segment (green) moves with its endpoints sliding along  the curve, forming a constant area triangle with the origin, while the midpoint of the line segment traces a curve (red), always tangent to the line segment at its midpoint. The two curves depicted here are members of an infinite family  of self-B\"acklund curves  described explicitly  in Section \ref{sect:Lame} in terms of  the Weierstrass $\wp$-function. For more images and animations, see \cite{Bo}.
	\label{fig:weg} }

\end{figure}

An analogous problem in Euclidean geometry was thoroughly studied relatively recently. The problem is to describe the closed smooth arc length parametrized curves $\g(t) \subset \R^2$ for which there exist constants $s$ and $\ell$ such that $|\g(t+s)-\g(t)|=\ell$ for all $t$. For example, a circle is a trivial solution to this problem.

Although the full solution of this problem is not available yet, there is a wealth of results, including many  non-trivial examples of such curves. See \cite{BLPT,Tab06,Tab17,We07,We19,We} for a sampler.

This problem originated in two seemingly unrelated theories. First, such curves are the boundaries of 2-dimensional bodies that float in equilibrium in all positions -- to describe such bodies (in all dimensions) is S.~Ulam's problem in flotation theory, see \cite{Mau}, problem 19.

Second, an interesting problem in the study of the bicycle kinematics is to describe the pairs of front and rear bicycle tracks for which one cannot determine the direction of the bicycle motion. The above mentioned curves appear in this problem as the front tracks in such ambiguous pairs; they are referred to as {\em bicycle curves}. See \cite{{FLT}} for a survey of this approach to the bicycle kinematics.

This geometric problem is intimately related to  another completely integrable equation of soliton type, the filament -- or binormal, or smoke ring, or local induction -- equation; more precisely, to the planar filament equation. 

Two arc length parametrized curves, $\g(t)$ and $\delta(t)$, are in bicycle correspondence if the length of the segment $\g(t) \delta(t)$ is constant and the velocity of its midpoint is aligned with the segment for all $t$. This correspondence is a geometric realization of the B\"acklund transformation of the planar filament equation, and in this sense, bicycle curves are self-B\"acklund. 

We must say more about the work of Franz Wegner, cited above. He discovered a large variety of bicycle curves (or solutions to the 2-dimensional Ulam's problem), explicitly described in terms of elliptic functions. Wegner made his discovery by assuming that the desired solutions have a certain geometrical property, resulting in a differential equation on their curvature that was solved in elliptic functions. Then he proved that indeed, for a proper choice of parameters, these curves solved the problem. 

It is shown in \cite{BLPT} that Wegner's curves are solutions to a variational problem: they are buckled rings (the relative extrema of the elastic -- or bending -- energy, subject to the length and area constraints), and they are solitons:  under the planar filament flow, they evolve by isometries.

Our main goal in this paper is to obtain centroraffine analogs of these results. 

In the spirit of discrete differential geometry, we also consider centroaffine polygons, a discretization of centroaffine curves. 
These are centrally symmetric  $2n$-gons $P_1,\ldots,P_{2n}$ such that $[P_i,P_{i+1}]=1$ and $P_{i+n}=-P_i$ for all $i$ (the index is understood cyclically). A centroaffine $2n$-gon is a {\it self-B\"acklund $(n,k)$-gon} if there exists a constant $c$ such that $[P_i,P_{i+k}]=c$ for all $i$.  A trivial example is an affine-regular $2n$-gon which is a self-B\"acklund $(n,k)$-gon for all $k$. The problem is to describe non-trivial self-B\"acklund $(n,k)$-gons.

These polygons are  centroaffine analogs of the discretization of the bicycle curves, the bicycle polygons, studied in \cite{Tab06,TT}. Some of our results on self-B\"acklund $(n,k)$-gons  were included in Section 7.3 of the original (but not the final) version  of \cite{AFIT}, and were motivated by the study of the cross-ratio dynamics on ideal polygons in the hyperbolic plane and hyperbolic space therein. 

Centroaffine polygons are closely related to linear second-order difference equations with periodic solutions and with Coxeter's frieze patterns, see \cite{Mor}. In particular, given a simple centroaffine $2n$-gon, the determinants $[P_i,P_j]$ with $|i-j|<n$ form the entries, all positive, of a frieze pattern of width $n-3$. In these terms, we are interested in frieze patterns that have a row consisting of the same numbers, but not every row being constant. 

A word about the terminology that we use. We call a closed smooth curve {\it star-shaped} if every ray emanating from the origin intersects the curve transversely and only once. A curve is {\it locally star-shaped} if the above property holds locally, near every point. 
Equivalently, $[\g(t),\g'(t)]\neq 0$ for all $t$. Star-shaped curves have winding number 1, but locally star-shaped curves can go around the origin several times.

\medskip

The contents of this paper are as follows.

Section \ref{sect:transf} concerns  B\"acklund transformations of centroaffine curves. We describe a centroaffine analog of the rear track curve (in the above mentioned bicycle setting). We also 
interpret the Miura transformation in terms of centroaffine geometry. 

Section \ref{subsect:range} is devoted to the following problem: given a centroaffine curve $\g$, for which $c$ do $c$-related curves exist? We provide a complete answer to this question. 
%
%
This result is a centroaffine analog of Menzin's conjecture -- now a theorem, originally formulated for hatchet planimeters, but it also applies to the bicycle model, see \cite{LT} or \cite{FLT}.

Section \ref{sect:SBfirst} comprises several results on self-B\"acklund curves. In Theorem \ref{thm:infdef} we prove that a 
non-trivial infinitesimal deformation of a central conic as a self-B\"acklund curve exists if and only if either $\alpha=\pi/2$ or
$\alpha$ satisfies the equation
$$
\tan ({k\alpha}) = k \tan\alpha
$$
for some integer $k\geq 4$. 
%
A similar result is known for bicycle curves, see \cite{Tab06}.

We show that if  $\alpha=\pi/3$ or $\alpha=\pi/4$ then only the central ellipses are self-B\"acklund. In contrast, if $\alpha=\pi/2$, one has a family  of self-B\"acklund centroaffine curves with functional parameters. Example \ref{example:weg} provides families of analytic curves with rotation number $\pi/2$  and, at the same time, examples of analytic Radon curves.
 
%
%

Section \ref{sect:Lame} is the core part of the paper. We start by developing a centroaffine analog of Wegner's ansatz, that is, we guess what geometric properties self-B\"acklund curves may possess. This leads to the assumption that these curves correspond to the traveling wave solutions of the KdV equation, that is, their centroraffine curvature is an elliptic function. 

Thus we assume that the coordinates of our self-B\"acklund curves satisfy the Lam\'e equation, the Hill equation whose potential is an elliptic function. In Section \ref{subsect:solLam} we construct these curves and describe the conditions on the parameters for which the curves are self-B\"acklund. This work is analogous to the one done by F. Wegner. In Section \ref{subsect:deform} we show that central conics indeed admit a deformation into self-B\"acklund centroaffine curves for each $\alpha$ appearing in Theorem \ref{thm:infdef}.

Section \ref{sect:selfBP} concerns self-B\"acklund centroaffine polygons. We start by showing that the $c$-relations on centroaffine curves satisfy the Bianchi permutability property (Theorem \ref{thm:Bianchi}). 

We describe a discrete version of B\"acklund transformation on centroaffine polygons (this transformation is studied in detail in \cite{AFT}). Theorem \ref{triv} presents some pairs $(n,k)$ for which non-trivial self-B\"acklund polygons do not exist, and some pairs for which they do. We also describe necessary and sufficient conditions for the existence of non-trivial infinitesimal deformations of regular centroaffine $n$-gons in the class of self-B\"acklund polygons. Similar results were known for bicycle polygons, see \cite{Tab06}.

In Appendix A we connect centroaffine geometry with another geometry associated with the group $\SLt$, two-dimensional hyperbolic geometry. We assign to a centroaffine curve a curve in the hyperbolic plane, it dual. The centroaffine curvature $p$ of a curve and the curvature $\kappa$  of its dual in $H^2$ are related by the equation $(1+p)(1+\kappa)=2$. 

Appendix B is a compendium of the formulas involving Weierstrass elliptic functions that we use in the body of the paper.

\bigskip

{\bf Acknowledgments}. We are grateful to L. Buhovsky, D. Fuchs, Fima Gluskin, A. Izosimov, M. Cuntz, A. Mironov, A. Sodin, and F. Wegner for fruitful discussions.
MB was supported by ISF grant 580/20, GB was supported by Conacyt grant $\#$A1-S-45886, ST was supported by NSF grants DMS-1510055 and DMS-2005444. We thank the referees for their constructive suggestions.

\section{B\"acklund transformations of centroaffine curves} \label{sect:transf}

\subsection{The middle curve} \label{subsect:middle}

Let $\g(t)$ be a centroaffine curve satisfying $\g''(t)=p(t)\g(t)$. Construct a new centroaffine curve $\delta(t)= f(t)\g(t)+g(t)\g'(t)$, where $f(t)$ and $g(t)$ are $\pi$-periodic functions. The next lemma repeats Lemma 1.2 of \cite{Tab18}.

\begin{lemma} \label{lm:fp}
The curves $\g$ and $\delta$ are $c$-related if and only if
$g(t)=c$ and
\begin{equation} \label{eq:fp}
cf'(t)-f^2(t)+c^2p(t)+1=0.
\end{equation}
\end{lemma}

\begin{proof}
One has 
$$
c=[\g(t),\delta(t))]=g(t)[\g(t),\g'(t)]=g(t),
$$
and therefore $g'(t)=0$.
Next,
$$
\delta'(t)=(f'(t)+p(t)g(t))\g(t) + (f(t)+g'(t))\g'(t),
$$
hence
$$
1=[\delta(t),\delta'(t)]=f^2(t)-c(f'(t)+cp(t)).
$$
This implies equation \eqref{eq:fp}.
\end{proof}

Note that equation  \eqref{eq:fp} is a Riccati equation on the unknown function $f(t)$.

\begin{lemma} \label{lm:mid}
Let $\g$ and $\delta$ be $c$-related and let $\G(t)$ be   the midpoint of the segment $\g(t) \delta(t)$.
Then the velocity of $\G$ is aligned with this segment:
$$
\G'(t) \sim \delta(t)-\g(t)
$$
for all $t$. In addition, $\G$ is locally star-shaped, that is, $[\G(t),\G'(t)]\ne 0$ for all $t$.
\end{lemma}

\begin{proof}
Since $[\g,\g']=[\delta,\delta']=1$ and $[\g,\delta]=c$, one has
$$
[\g'+\delta',\delta-\g]= [\g',\delta]-[\delta',\g]=[\g,\delta]'=0,
$$
as needed. 

For the second statement, if  $[\G(t),\G'(t)]= 0$ then the line connecting $\g(t)$ and $\delta(t)$ passes through the origin, and then $c=0$.
\end{proof}

\begin{remark}
{\rm 
The locus of midpoints in the previous lemma plays the role of the rear bicycle track in the analogous problem mentioned in  Introduction. This middle curve may have cusps.
}
\end{remark}




We describe a method of constructing pairs of $c$-related curves. Start with a locally star shaped curve $\G$, with a centroaffine parameter $s$ and curvature $p(s)$, so that
 $[\G, \G_s]=1,\ \G_{ss}=p\G.$     Let $\g_\pm:=\G\pm (c/2)\G_s.$ The condition $[\gamma_-, \gamma_+]=c$ is immediate;  however, in general, $s$ is not a centroaffine parameter for $\g_\pm$.

\begin{proposition}  If  $c^2p\neq 4$ along $\G$  (for example, if $\G$ is locally convex, that is, $p<0$), then  $\g_\pm$ can be simultaneously reparametrized by a centroaffine parameter $t$, so that  $[\g_\pm,(\g_\pm)_t]=1$. 
\end{proposition}

\begin{proof} We calculate that $[\g_\pm,(\g_\pm)_s]=1-(c^2/4)p$.  If this does not vanish, then the desired parameter $t$ is defined by 
$$
\frac{dt}{ds} = 1-\frac{c^2p(s)}{4}.
$$
 With this new parameter one has  $[\g_\pm,(\g_\pm)_t]=1,$ as needed.
\end{proof}

\begin{remark} \label{rmk:cusps}
{\rm As we mentioned, and as is seen from illustrations in this paper, the middle curve $\G$ may have cusps.
The above  construction of the curves $\g_\pm$ from $\G$ extends to the case when $\G$ has cusps and the curves $\g_\pm$  remain smooth. Without going into details, we illustrate this with an example.

Let $\G(x)=(x^2,x^3+1)$ be a cusp, and let $s$ be a centroaffine parameter. Then $\G_x=(2x,3x^2)$ and 
$$
\frac{ds}{dx}=[\G,\G_x]= x^4-2x.
$$
It follows that
$$
\g_\pm=\G\pm\frac{c}{2}\G_s
=\left(\mp{c\over 2},1\right)+\left(0,\mp  \frac{3 c }{4}\right)x+O(x^2),
$$
which, for $c\neq 0$ and $x$ close to zero, are smooth curves.
}
\end{remark}

\begin{remark} \label{rmk:outer}
{\rm
Consider an oriented smooth closed strictly convex plane curve $\G$.  The outer billiard transformation $T$ is a map of its exterior, defined as follows: given a point $x$, draw the oriented tangent line from $x$ to $\G$, and reflect $x$ in the tangency point to obtain the point $T(x)$. See  \cite{DT} for a survey. 

The relation of our topic to outer billiards is as follows:
if $\g$ is a self-B\"acklund curve and the respective middle curve $\G$ is convex, then $\g$ is an invariant curve of the outer billiard map about $\G$. 
}
\end{remark}



\subsection{Curves $c$-related to centroaffine conics} \label{subsect:infty}

%
In this section we consider the curves that are $c$-related to centroaffine conics and 
identify self-B\"acklund curves among them. These curves will have points at infinity.

Let $\g(t)=(\cos t, \sin t)$, and let us construct a $c$-related curve as in Lemma \ref{lm:fp}: $\delta(t)=f(t)\g(t)+c\g'(t)$. The respective Riccati equation for the function $f$ is
\begin{equation} \label{eq:Ric}
cf'(t) = f^2(t) + c^2-1.
\end{equation}
Assume that $c>1$. This differential equation is easily solved:
\begin{equation} \label{eq:solnR}
f(t) = a \tan \left(\frac{at}{c}\right), \quad  \mbox{where }\ a=\sqrt{c^2-1}
\end{equation}
and a choice of the  constant of integration has been  made so that $f(0)=0$ (any other solution is obtained by a parameter shift). 

The function $f$ has poles (the same is true for the solutions with $c<1$ and $c=1$), and the respective centroaffine curve goes to infinity, having there an inflection point. 

For example, let $c=5/3, a=4/3$, see Figure \ref{curveinf}. This curve is periodic with period $10\pi$.

\begin{figure}[ht]
\centering
\includegraphics[width=.5\textwidth]{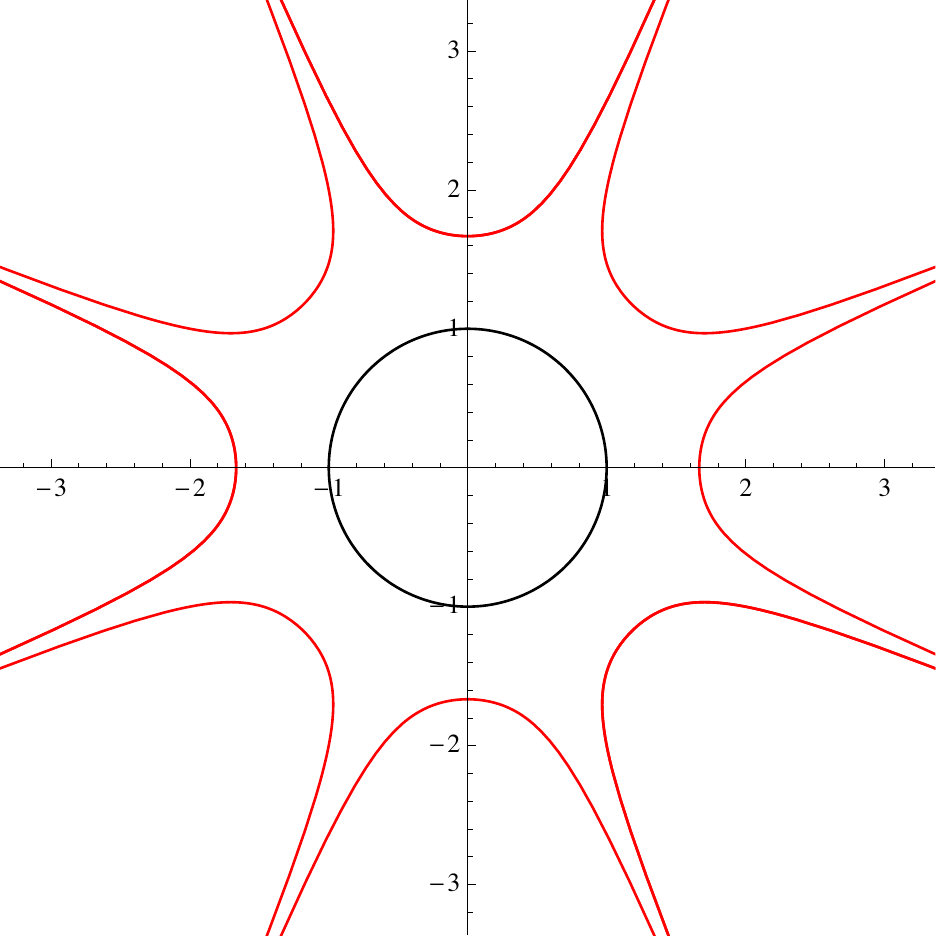}
\caption{The curve 
$\delta(t) = \left(\frac{4}{3} \tan \left(\frac{4t}{5}\right) \cos t - \frac{5}{3}\sin t, \frac{4}{3} \tan \left(\frac{4t}{5}\right) \sin t + \frac{5}{3}\cos t  \right).$
}
\label{curveinf}
\end{figure}
%
Let us look for self-B\"acklund curves among the above curves $\delta$. 
\begin{lemma}\label{circle}
Let $\delta$ be the centroaffine curve $c$-related to the  unit circle  $\gamma(t)=(\cos t, \sin t)$,  where $c>1$. Then $\delta$ is self-B\"acklund with rotation number $\alpha$,  that is, $[\delta(t), \delta(t+\alpha)]=$const, if and only if   $\alpha$ satisfies 
\begin{equation} \label{eq:cond}
\tan\left(u \alpha\right) = u \tan \alpha,\quad \mbox{where }\ u={\sqrt{c^2-1}\over c}.
\end{equation}
Furthermore, given such an $\alpha$, one has $[\delta(t), \delta(t+\alpha)]=\sin\alpha.$

\end{lemma}

\begin{proof}The statement is invariant under parameter shift so it is enough to consider    $\delta=f\g+c\g'$, where $f$ is given by formula \eqref{eq:solnR}. 
%
%
Next, by a straightforward calculation, 
the derivative of  $[\delta(t),\delta(t+\alpha)]$ with respect to $t$  is  some non-zero function   times  
$\tan\left(u \alpha\right) - u \tan \alpha$. It follows that  $[\delta(t),\delta(t+\alpha)]$ is constant if and only if  $\tan\left(u \alpha\right) = u \tan \alpha$. Using this equation for $\alpha$,  we  calculate that  $[\delta(t), \delta(t+\alpha)]=\sin\alpha$.
\end{proof} 

In general, for a fixed $u\in(0,1)$, equation \eqref{eq:cond} has infinitely many solutions. See Figure \ref{fig:eqn}. If $u$ is rational  then $\delta$ is  periodic and there are  finitely many solutions $\alpha$ within a period.  

\begin{figure}[ht]
\centering
\includegraphics[width=\textwidth]{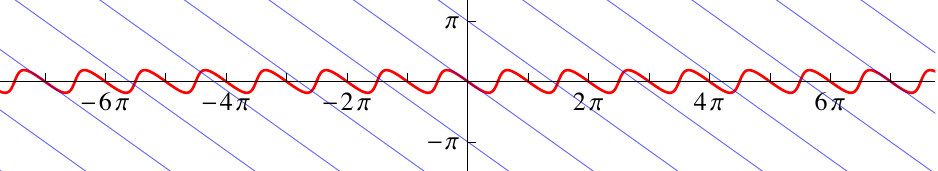}
\caption{Solutions to equation \eqref{eq:cond}, $u\tan\alpha=\tan(u\alpha),$ $u\in(0,1)$, are given by the intersection points of the (red) graph of the $\pi$-periodic function $y=\tan ^{-1}(u \tan\alpha)-\alpha+\pi  n,$ $\pi  n-\frac{\pi }{2}\leq \alpha \leq \pi  n+\frac{\pi }{2}$, $n\in\Z$, and any of the  (blue) lines $y=(u-1)\alpha +n\pi,$ $n\in \Z$. If $u$ is rational then $f=a\tan(ut)$ is periodic and $\delta$ is closed, self-B\"acklund with rotation numbers $\alpha$ given by the intersection points within a period of $f$. In the figure above,  $u=2/7$, $f$ is $7\pi$-periodic, $\delta$ is $14\pi$-periodic, and there are 8 solutions $\alpha\in (0,14\pi)$ with $\sin\alpha\neq 0.$
}
\label{fig:eqn}
\end{figure}

%


\begin{figure}[ht]
\centering
\includegraphics[width=.35\textwidth]{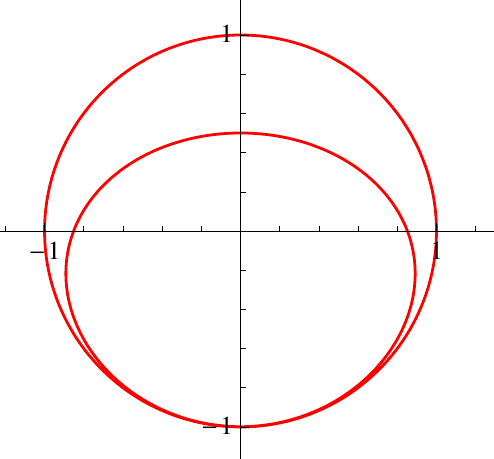}
\caption{The curve 
$\delta(t) = \left(-\frac{4}{5} \tanh \left(\frac{4t}{3}\right) \cos t - \frac{3}{5}\sin t, -\frac{4}{5} \tanh \left(\frac{4t}{3}\right) \sin t + \frac{3}{5}\cos t  \right).$
}
\label{circles}
\end{figure}

A solution of equation \eqref{eq:Ric} for $c < 1$ is similar:
$$
f(t)=-a\tanh\left(\frac{at}{c}\right),
$$
where $a^2=1-c^2$. The associated $c$-related curve $\delta=f\g+c\g'$ is  non-periodic and  stays bounded;  it 
is self-B\"acklund with a parameter shift $\alpha$ satisfying 
$$
\tanh \left(u\alpha\right) = u \tan \alpha,\quad \mbox{where }\ u={\sqrt{1-c^2}\over c},
$$
and the constant determinant is $\sin \alpha$. This equation admits infinitely many solutions $\pm \alpha_1, \pm\alpha_2,\ldots, $ with $\alpha_n\in (n\pi, n\pi+\pi/2).$ For $t\to\pm\infty$, the curve approaches the unit circle, see Figure \ref{circles}.

Another solution of \eqref{eq:Ric} for $c < 1$ is
$$
f(t)=-a\coth\left(\frac{at}{c}\right),
$$
with the respective value of $\alpha$ given by
$$
\coth \left(u\alpha\right) = u \tan\alpha, \quad \mbox{where }\ u={\sqrt{1-c^2}\over c}
$$
and the constant determinant is $\sin\alpha$. There are infinitely many solutions here as well, $\pm \alpha_0, \pm\alpha_1,\ldots, $ with $\alpha_n\in (n\pi, n\pi+\pi/2).$ This curve approaches the unit circle as $t\to\pm\infty$ and goes to infinity as $t\to 0$. See Figure \ref{long}.

\begin{figure}[ht]
\centering
\includegraphics[width=.8\textwidth]{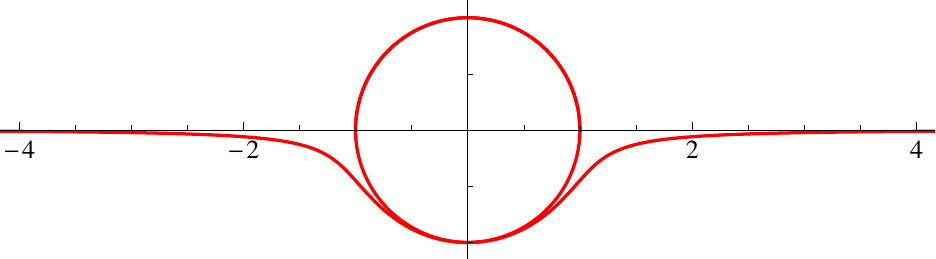}
\caption{The curve 
$\delta(t) = \left(-\frac{4}{5} \coth \left(\frac{4t}{3}\right) \cos t - \frac{3}{5}\sin t, -\frac{4}{5} \coth \left(\frac{4t}{3}\right) \sin t + \frac{3}{5}\cos t  \right).$
}
\label{long}
\end{figure}

If $c=1$, a solution of equation  \eqref{eq:Ric} is $f(t)=-1/t$. This curve is  self-B\"acklund with a parameter shift $\alpha$ satisfying
$\tan \alpha = \alpha$ and the constant determinant is $\sin\alpha$. There are  infinitely many solutions  $\pm \alpha_1, \pm\alpha_2,\ldots, $ with $\alpha_n\in (n\pi, n\pi+\pi/2).$  Its asymptotic behavior is the same as in the previous example, see Figure \ref{asympt}.

\begin{figure}[ht]
\centering
\includegraphics[width=.8\textwidth]{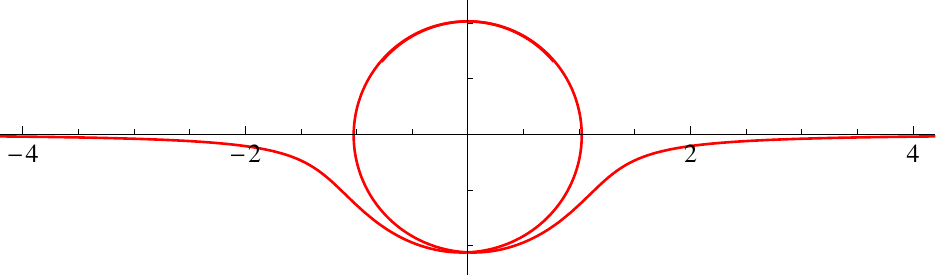}
\caption{The curve 
$\delta(t) = \left(-\frac{1}{t} \cos t - \sin t, -\frac{1}{t} \sin t + \cos t  \right).$
}
\label{asympt}
\end{figure} 

For completeness, consider the case of a straight line $\g(t)=(t,-1)$. This centroaffine curve is self-B\"acklund for an arbitrary parameter shift. A $c$-related curve $f\g+c\g'$ has $f(t)=-\tanh(t/c)$, see Figure \ref{clip}. This curve is not self-B\"acklund: the respective equation on  the parameter shifts $b$ is
$$
\tanh \left(\frac{b}{c}  \right) = \frac{b}{c},
$$
and the only solution is $b=0$.

\begin{figure}[ht]
\centering
\includegraphics[width=.8\textwidth]{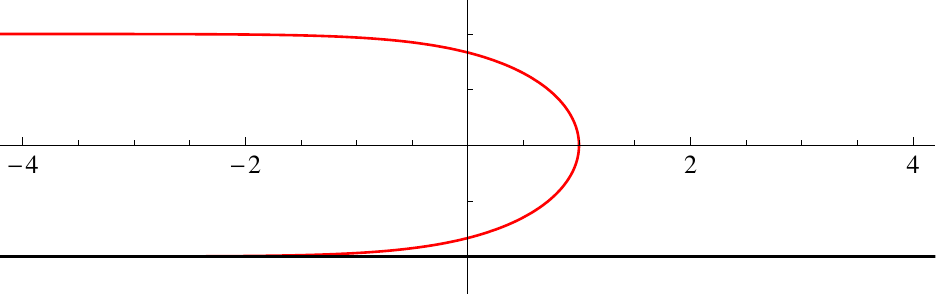}
\caption{The curve 
$\delta(t) = (1-t\tanh t, \tanh t)$ (red),  a B\"acklund transform of the line $y=-1$ (black). 
}
\label{clip}
\end{figure}

\subsection{$c$-related curves and  Miura transformation} \label{subsec:Miura}
The Miura transformation connects the Korteweg-de Vries equation $\dot u =u'''+6uu'$ and the modified Korteweg-de Vries equation $\dot v = v'''-6v^2v'$: if $v$ satisfies mKdV then $u=-v'-v^2$ satisfies KdV. More generally, if
\begin{equation} \label{eq:Miura}
u=-v'-v^2+\lambda,
\end{equation}
and $v$ satisfies
\begin{equation} \label{eq:MKdV}
\dot v = v'''-6v^2v'-6\lambda v',
\end{equation}
then $u$ satisfies KdV. See  \cite{For}.

Given $u$,  equation  \eqref{eq:Miura} is a Riccati equation on $v$, just like equation \eqref{eq:fp} on the function $f(t)$ that describes the curves, $c$-related to a centroaffine curve with curvature $p(t)$. This provides a geometrical interpretation of the Miura transformation in centroaffine geometry. 

The details are described by the next theorem.

\begin{theorem} \label{thm:evolf}
Let $\g$ be a centroaffine curve, and  $\delta=f\g+c\g'$ be a $c$-related curve. Let the curves $\g$ and $\delta$ evolve by the KdV flow. Then they remain $c$-related, and the function $f$ evolves according to a version of mKdV:
$$
\dot f = -\frac{1}{2}f'''+\frac{3}{c^2}(f^2 -1)f'.
$$
\end{theorem}

\begin{proof}
Let $q$ be the centroaffine curvature of $\delta$, that is, $\delta''(t)=q(t) \delta(t)$. Then
$\dot\g=V_p,\ \dot\delta=V_q,$
where we use the notation as in equation  \eqref{eq:tang}. 

We start with the observation that $\g=f\delta-c\delta'$, and then we express the curvatures $p$ and $q$ from   equations \eqref{eq:fp} as follows
\begin{equation} \label{pqviaf}
p=\frac{1}{c^2} (f^2-1 - cf'),\ q=\frac{1}{c^2} (f^2-1 + cf')
\end{equation}
(compare with Lemma 3.1 in \cite{Tab18}). It follows that
\begin{equation} \label{pqviaf1}
q-p=\frac{2}{c}f',\ p'+q'=\frac{4}{c^2}ff'.
\end{equation}

That $\g$ and $\delta$ remain $c$-related under the KdV flow follows form the fact the $c$-relation commutes with the KdV flow, see \cite{Tab18}. Here is an independent verification.

We have: $\delta'=(f'+cp)\g+f\g'$, and
\begin{equation*}
\begin{split}
[\g,\delta]^{\cdot} = [V_p,\delta]+[\g,V_q] = &[-0.5p'\g+p\g',\delta]+[\g,-0.5q'\delta+q\delta']=\\
  &-0.5c(p'+q')+f(q-f) = 0,
\end{split}
\end{equation*}
the last equality due to equation  \eqref{pqviaf1}. 

To calculate $\dot f$, note that $f=[\delta,\g']$. Then
$$
\dot f = [\dot\delta,\g']+[\delta,\dot\g'] = [V_q,\g']+[\delta,V_p'] = 
[-0.5 q'\delta+q\delta',\g'] + [\delta,(-0.5p'\g+p\g')'].
$$
After substituting the values of $p$ and $q$ and their derivatives in terms of $f$ from equation  \eqref{pqviaf} and collecting terms we obtain the stated equality.
\end{proof}

One can expand a periodic solution of equation  \eqref{eq:fp} in a power series in $c$:
\begin{equation*}
\begin{split}
f=1+\frac{c^2}{2}p+\frac{c^3}{4}p'+&\frac{c^4}{8}(p''-p^2)+
\frac{c^5}{16}(p'''-8pp')\\
&+\frac{c^6}{32}[p^{''''}-10pp''-9(p')^2+2p^3] + \ldots
\end{split}
\end{equation*}
Given the relation of $f$ with the Miura transformation, one has the next statement; see Section 1.1 of \cite{For}.

\begin{corollary}
The integrals of the odd terms of this series vanish, and the integrals of the even terms are integrals of the KdV equation:
$$
\int_0^\pi p\ dt,\ \int_0^\pi p^2\ dt,\ \int_0^\pi \left(p^3+\frac{1}{2}(p')^2\right)\ dt, \ldots
$$ 
\end{corollary}

See \cite{BLPT}, Section 3.3 for a similar statement about the bicycle transformation and the filament equation.

\subsection{Range of the  parameter $c$} \label{subsect:range}

%
The aim of this section is to describe, for a given centroaffine closed 
$\pi$-anti-periodic curve $\g(t)$,  the  range of the parameter $c$ for which $\g$ admits  closed centroaffine  $c$-related curves.  The main result is Theorem \ref{th:c} below, describing this range (a closed interval) in terms of the lowest eigenvalue of a  Hill equation associated with $\g$.  For  a convex $\g$ we obtain  as a corollary  an upper bound on $c$   in terms of the area  enclosed by its dual curve $\gamma^*$. This result can be viewed  as a centroaffine analog of Menzin's conjecture for hatchet planimeters (equivalently, bicycle monodromy), discussed and proved in \cite{LT}.

As we saw in Lemma \ref{lm:fp}, finding a  centroaffine curve  $c$-related  to a given curve $\gamma$ amounts to finding a  solution $f(t)$ to the  Riccati equation
\begin{equation}
cf'-f^2+c^2p(t)+1=0, \label{eq:ric}
\end{equation}
where $p=[\g'',\g']$ (the centroaffine curvature of $\gamma$). The corresponding $c$-related   centroaffine curve is  $\delta=f\g+c\g'$.
If $\gamma$ is $\pi$-anti-periodic then $p$ in equation \eqref{eq:ric}
 is $\pi$-periodic and we are looking for the values of the parameter $c$ for which the equation admits a $\pi$-periodic solution, so that $\delta$ is $\pi$-anti-periodic as well. Note that for $c=0$ the equation admits the   trivial solution $f\equiv 1$.  

Our study of  the Riccati equation \eqref{eq:ric} is based on  its relation with the Hill equation
\begin{equation} \label{eq:Hill} y''+(\lambda -p(t))y=0. \end{equation} 

To state this relation we recall first that a solution $y(t)$ of \eqref{eq:Hill} is
called {\em $\pi$-quasi-periodic} if   $y(t+\pi)=\mu\, y(t)$ for all $t$ and some  $\mu\in\R$,  $\mu\neq 0$, called the {\em Floquet multiplier} of $y(t)$. If $\mu=1$ then the solution is {\em $\pi$-periodic} and if $\mu=-1$ it is {\em $\pi$-anti-periodic.} 




\begin{proposition}\label{lemma:easy}
The Riccati equation  \eqref{eq:ric} with a $\pi$-periodic  $p(t)$ admits  a $\pi$-periodic solution  $f(t)$ for   a  parameter value $c\neq 0 $  if and only if the  Hill equation \eqref{eq:Hill}
admits a positive $\pi$-quasi-periodic solution $y(t)$ for  $\lambda=-1/c^2$.

\end{proposition}

	\begin{proof} 
	Indeed, 
	if there exists such $y(t)$, 
	then  
	$f:=-cy'/y
	$
	is a periodic solution of equation  \eqref{eq:ric}. In the opposite direction: if $f$ is a periodic solution of equation  \eqref{eq:ric} and  $F$ is a primitive of $f$ then  
	$y:=e^{-F/c}$ is the required  solution of  equation  \eqref{eq:Hill}.
	 \end{proof}

We now borrow a well-known result from the general theory of the Hill equation, due to  Lyapunov and Haupt (ca.~1910, see Theorem 2.1 on page 11 of \cite{Mag}).


\mn{\bf Theorem} (Spectrum of the Hill operator). {\em  Consider equation
\eqref{eq:Hill}, 
$$y''+(\lambda -p(t))y=0,
$$
where  $y(t)$  is an unknown real function,  $p(t)$ is a real $\pi$-periodic function and $\lambda$ a real parameter. Then there exist  two unbounded sequences of real numbers \begin{align*}
&\lambda_0<\lambda_1\leq \lambda_2<\lambda_3\leq \lambda_4 <\ldots, \\
&\mu_0\leq \mu_1< \mu_2\leq \mu_3< \mu_4 \leq\ldots,
\end{align*}
satisfying 
\begin{align}\label{spectrum}
&\lambda_0 < \mu_0\le \mu_1 < \lambda_1\le \lambda_2<\mu_2\le\mu_3<\lambda_3\leq \lambda_4 <\ldots,
\end{align}
such that equation  \eqref{eq:Hill} has a non-trivial $\pi$-periodic solution if and only if $\lambda=\lambda_k,$ and a $\pi$-anti-periodic non-trivial solution if and only if $\lambda=\mu_k,\ k=0,1,\ldots.$ The number of zeros on $[0,\pi)$ of a solution corresponding to $\lambda_{2k-1}$ or $\lambda_{2k}$  is $2k$. In particular, if a $\pi$-periodic solution has no zeros, then $\lambda=\lambda_0$. Similarly, the number of zeros on $[0,\pi)$ of a non-trivial  solution corresponding to $\mu_{2k}$ or  $\mu _{2k+1}$  is $2k+1$. Moreover, a solution to equation  \eqref{eq:Hill} is unstable (that is, unbounded) if and only if $\lambda$ belongs to one of the intervals $(-\infty,\lambda_0), \ (\mu_0, \mu_1), \ (\lambda_1, \lambda_2), \ldots $ (called instability intervals, or `gaps').
See Figure \ref{fig:hill}.
}

\begin{figure}[ht]
\centering
\includegraphics[width=.8\textwidth]{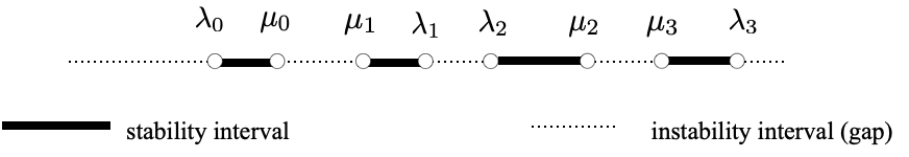}
\caption{The spectrum of Hill's equation \eqref{eq:Hill}, stability and instability intervals.}
\label{fig:hill}
\end{figure}

Concerning the lowest eigenvalue $\lambda_0$, we have the following.

\begin{lemma} \label{lm:Pl}Let $\lambda_0$ be the first eigenvalue of the  spectrum \eqref{spectrum} of the   Hill equation \eqref{eq:Hill} associated with a $\pi$-anti-periodic centroaffine curve $\gamma$.   Then
$$ \lambda_0<0,\quad \lambda_0\leq-P,$$
where 
\begin{equation}\label{eq:P}
P:=-\frac{1}{\pi}\int_0^\pi p(t)\ dt.
\end{equation}

\end{lemma}
\begin{proof}

Each of the two coordinate components 
of $\g$ is a non-trivial  $\pi$-anti-periodic solution of equation  \eqref{eq:Hill} for  $\lambda=0$. This implies that $\mu_k=0$  for some  $k\geq 1$,  hence  $\lambda_0<0$. 

The inequality $ \lambda_0\leq-P$ is due to  Borg (see Theorem 3.3.1 of \cite{East}).
The following argument is due to Ungar:
Take a positive periodic solution $y(t)$  of equation \eqref{eq:Hill} corresponding to $\lambda_0$. Then $h(t)=y'(t)/y(t)$ is a periodic solution of  the Riccati equation $h'+h^2+(\lambda_0-p(t))=0.$ Integrating this equation over the period gives:
$$
\int_{0}^{\pi}(\lambda_0-p(t)) dt\leq 0.
$$
This yields the result. \end{proof}
\begin{remark}If $\gamma$ is locally convex, so that $p(t) $ is strictly  negative, then $P>0$ and we have 
	$\lambda_0\leq-P<0$.
	The geometric meaning of $P$  is the area bounded by the dual curve $\g^*$ (we refer to \cite{Gugg} and \cite{Tab11} for this and related facts). 
\end{remark}


\begin{theorem}\label {th:c}Let $\g$ be a centroaffine $\pi$-anti-periodic curve and $\lambda_0<0$ the lowest  $\pi$-periodic eigenvalue of the associated Hill equation \eqref{eq:Hill}. Then  $\gamma$ admits a $c$-related closed curve 
 if and only if $|c|\leq 1/\sqrt{-\lambda_0}.
	$
\end{theorem}

An immediate consequence of this theorem and Lemma \ref{lm:Pl} is the following. 

\begin{corollary}\label{cor:main}
	Suppose $P>0$ (for example $\g$  is locally convex) and $\gamma$ admits a $c$-related $\pi$-anti-periodic closed curve. Then  
	$|c|\leq 1/\sqrt{P}$.
\end{corollary}

\mn{\em  Proof of Theorem \ref {th:c}.} By Proposition \ref{lemma:easy}, we need to  show that equation  \eqref{eq:Hill} admits a $\pi$-quasi-periodic positive solution if and only if $\lambda\leq \lambda_0.$

	Consider first  the ``if" part. If 
	$\lambda=\lambda_0$ then  equation  \eqref{eq:Hill} has a positive periodic solution, hence quasi-periodic.  So we shall assume now that $\lambda<\lambda_0$. In this case 
	equation  \eqref{eq:Hill} has no conjugate points, that is,  a  non-trivial solutions vanishing at two distinct points $t_1,t_2$ because,  by the Sturm Comparison Theorem, any solution  for every larger $\lambda$
	must have a zero between $t_1,t_2$. However for $\lambda_0$ there is a positive periodic solution.
	To  complete the proof of the ``if" part we make use of the  following lemma. 
	\begin{lemma}\label{lm:hopf}
		The equation  $
		y''+q(t)y=0,$ where $q(t+\pi)=q(t)
		$, 
		has no conjugate points 
		if and only if it admits a positive $\pi$-quasi-periodic solution.
		
	\end{lemma}
	As far as we know, this lemma   is due  to E.~Hopf \cite{Hopf}. For completeness, we give its proof  below.
	
	Now we  prove  Theorem \ref{th:c} in the opposite direction. We  need to show that equation  \eqref{eq:Hill} admits no positive $\pi$-quasi-periodic solution for  $
	\lambda>\lambda_0
	$. Assume $y(t)$ is such a solution,  $y(t+\pi)=\mu\, y(t),$ where $\mu>0.$ 
	There are two cases:
	\begin{itemize}
	\item
	If $\mu=1$ then $y(t)$ is a positive periodic solution. But this is possible only for $\lambda=\lambda_0$, a contradiction.
	
	\item If  $\mu\neq 1$  then the solution $y(t)$ is unbounded, and hence
	$\lambda$ belongs to one of the instability zones.
	In particular, 
	$
	\lambda>\mu_0.
	$
	But then, by the Sturm Comparison Theorem, $y(t)$ cannot be positive since solutions for $\mu_0$ have zeroes.
	\end{itemize}
	This completes the proof of Theorem \ref{th:c}.
\qed

\mn {\em Proof of Lemma \ref{lm:hopf}} (after E.~Hopf).
	If a  Hill equation $
		y''+q(t)y=0$ has no  conjugate points then for every two distinct $a, b\in\R$  there exist a unique
	solution $y(t;a, b)$ satisfying 
	$$y(a;a,b)=1,\ y(b;a,b)=0.
	$$
	By uniqueness, one has the relation for distinct $a,a'$: 
	\begin{equation}\label{a'}
	y(t;a,b)=y(a';a,b)y(t;a',b).
	\end{equation} 
	Using disconjugacy, one can show that a limiting solution 
	exists and is positive everywhere:
	$$
	y(t;a):=\lim_{b\rightarrow +\infty} y(t;a,b).
	$$
	
	These positive solutions are $\pi$-quasi-periodic. Indeed, setting  $a'\mapsto a+\pi,\ t\mapsto t+\pi$ in equation  \eqref{a'}) and passing to the limit $b\rightarrow+\infty$, we get
	\begin{equation}\label{floquet}
	y(t+\pi;a)=y(a+\pi;a)y(t+\pi;a+\pi)=y(a+\pi;a)y(t;a),
	\end{equation}
	where the last equality is due to the  $\pi$-periodicity of $q(t).$ 
	Thus,  $y(t;a)$ is $\pi$-quasi-periodic with multiplier $\mu=y(a+\pi;a),$ as needed. 
	
	In the opposite direction the claim is obvious: if $y''+q(t)y=0$ admits a positive solution then,  by the Sturm Oscillation Theorem, any non-trivial solution has no conjugate points. 
\qed

\section{Self-B\"acklund curves: first study} \label{sect:SBfirst}

\subsection{Infinitesimal deformations of centroaffine conics} \label{subsect:infinitesimal}

%
%

In this section we study infinitesimal deformations of centroaffine conics in the class of self-B\"acklund centroaffine curves. (This includes, as we recall from  the Introduction, the requirement for $\pi$-anti-periodicity). We describe the values of the parameter $\alpha$ for which centroaffine conics admit non-trivial infinitesimal deformations. 
Later, in Section \ref{subsect:deform}, we shall show that these values of $\alpha$ are realized by actual deformations, see Corollary \ref{cor:deform}.

Here is a brief reminder about   deformations.  Let $ \g(t)$ be a self-B\"acklund centroaffine curve,   satisfying
\be\label{eq:deform}[\g,\g']=1, \ [\g(t), \g(t+\alpha)]=c,
\ee  
for some constants $\alpha,c.$ A {\em deformation} of such a curve, within the class of self-B\"acklund centroaffine curves, 
is a function $\tilde\g( t,\e)$ defined on  $\R\times(-\e_0, \e_0)$ for some $\e_0>0$, and functions $\tilde\alpha(\e), \tilde c(\e)$  defined on $(-\e_0, \e_0),$ satisfying equation  \eqref{eq:deform} for each fixed $\e$,  namely 
\be\label{eq:inf_deform}
\left[\tilde\g,{\partial\over\partial t}\tilde\g\right]=1, \ [\tilde\g(t,\e), \tilde\g(t+\tilde\alpha(\e), \e)]=\tilde c(\e),
\ee
and such that   $\g=\tilde\g(\cdot,0),$ $\alpha=\tilde\alpha(0)$ and $c=\tilde c(0)$.

An {\em infinitesimal deformation} of $\g$ is a  formal expression $\tilde\g=\g(t)+\e\g_1(t)$, satisfying equation  \eqref{eq:inf_deform} for each $\e$, modulo $\e^2$, for some $\tilde\alpha=\alpha+\e\alpha_1$, $\tilde c=c+\e c_1.$  Clearly, if $\tilde\g$ is a deformation of $\g$, then its first jet, $\g+\e\left.{\partial\over \partial\e}\right|_{\e=0}\tilde\g$, is an infinitesimal deformation of $\g$. However, the converse is  not necessarily true, that is, given an infinitesimal deformation $\g+\e\g_1$, it is not clear  {\em a priori} that  there exists an `actual' deformation $\tilde\g$ of $\g$ such that $\g_1=\left.{\partial\over \partial\e}\right|_{\e=0}\tilde\g$. 

An infinitesimal deformation is {\em trivial} if it is induced by a shift of the argument, $\tilde\g(t,\e)=\g(t+a\e),$  or by the action of $\SLt$, $\tilde\g(t,\e)=e^{\e A}\g(t),$  $A\in\mathfrak{sl}_2(\R)$.



\begin{theorem} \label{thm:infdef} Let $\g(t)=(\cos t,\sin t)$. Then 

\begin{enumerate}
\item  A non-trivial  infinitesimal deformation of $\g$ within the class of self-B\"acklund $\pi$-anti-periodic centroaffine curves exists if and only if $\tilde\alpha=\alpha+\e\alpha_1$ 
where 
$\alpha=\pi/2$, or
$\alpha\neq \pi/2$ and $\alpha$ satisfies the equation
\begin{equation} \label{eq:ubiq}
\tan ({k\alpha}) = k \tan\alpha
\end{equation}
for some integer $k\geq 4$. 
%
\item For $k\geq 2$, there are exactly $k-2$ solutions of equation \eqref{eq:ubiq}  in the interval $(0,\pi)$, counting also $\alpha=\pi/2$ as a solution for $k$ odd.

\end{enumerate}
\end{theorem}

\begin{proof} 1.  
We make calculations  mod $\e^2$. 
%
The first equation of \eqref{eq:inf_deform} means that $\g_1$ is a vector field along $\g$, hence 
$\g_1 = -(1/2) f' \g + f\g'$ for a $\pi$-periodic function $f(t)$, see equation  \eqref{eq:tang}.
The second equation of \eqref{eq:inf_deform} implies
\be \label{eq:gdef}
[\g_1(t),\g(t+\alpha)]+[\g(t),\g_1(t+\alpha)]+\alpha_1[\g(t), \g'(t+\alpha)]=c_1.
\ee
For $\g(t)=(\cos t, \sin t)$ we have 
\be\label{eq:prep}
\begin{split}
&[\g(t),\g(t+\alpha)]=\sin\alpha,\quad  [\g'(t),\g(t+\alpha)]=-\cos\alpha,\\
& [\g(t),\g'(t+\alpha)]=\cos\alpha,
\end{split}
\ee
hence \eqref{eq:gdef} becomes
$$
[\g_1(t),\g(t+\alpha)]+[\g(t),\g_1(t+\alpha)]=c_1+\alpha_1\cos\alpha={\rm const}.
$$
It follows that 
\begin{equation*}
\begin{split}
&[-{1\over 2}f'(t) \g(t) + f(t)\g'(t),\g(t+\alpha)]+\\ 
&\quad+[\g(t),-{1\over 2}f'(t+\alpha) \g(t+\alpha) + f(t+\alpha)\g'(t+\alpha)]={\rm const}.
\end{split}
\end{equation*}
In view of equation \eqref{eq:prep}, this implies 
\begin{equation} \label{eq:harm}
{1\over 2} \left[f'(t) + f'(t+\alpha)\right]\sin\alpha - \left[f(t+\alpha)-f(t)\right]\cos\alpha={\rm const}.
\end{equation}
Since the integral of the left hand side over the period is zero, the constant on the right hand side is also zero.

Recall that $f$ is a $\pi$-periodic function and let
$$
f(t) = \sum_{k=-\infty}^{\infty} a_k e^{2ikt}
$$
 be its Fourier expansion, with $a_{-k}=\bar a_k$. Then
\begin{equation*}
\begin{split}
&f'(t) = 2i\sum_{k=-\infty}^{\infty} k a_k e^{2ikt},\quad  
f(t+\alpha) = \sum_{k=-\infty}^{\infty} a_k  e^{2ik\alpha} e^{2ikt},\\
 &f'(t+\alpha) = 2i \sum_{k=-\infty}^{\infty} k a_k e^{2ik\alpha}e^{2ikt}.
\end{split}
\end{equation*}
Substitute this in equation  \eqref{eq:harm} to conclude that 
$$
 a_k \left[ik\left(1+e^{2ik\alpha}\right)\sin\alpha - \left(e^{2ik\alpha}-1\right)\cos\alpha\right] =0
$$
for each $k$.
Hence $a_k=0$, unless
$$
ik(1+e^{2ik\alpha})\sin\alpha = (e^{2ik\alpha}-1)\cos\alpha,
$$ 
or
$$
k \frac{e^{ik\alpha}+e^{-ik\alpha}}{2} \sin\alpha = \frac{e^{ik\alpha}-e^{-ik\alpha}}{2i} \cos\alpha,
$$
that is, 
$k\tan\alpha = \tan(k\alpha).$

Conversely, if equation \eqref{eq:ubiq} holds, then one can choose $f(t)$ to be a pure harmonic of order $2k$, 
%
and then equation \eqref{eq:inf_deform} holds modulo $\e^2$. Likewise, if $\alpha=\pi/2$, one can choose $g(t)$ to be a pure harmonic of order $2k$ with odd $k\geq 3$ or a linear combination of such harmonics. 

Note that equation  \eqref{eq:ubiq} holds trivially  for $k=0$ and $k=1$. The former case corresponds to $f(t)$ being constant, a shift of the argument of $\g(t)$. The latter case corresponds to the action of $\mathfrak{sl}(2,\R)$, a stretching of the unit circle to an ellipse bounding area $\pi$. 

For $k=2$ there are no solutions $\alpha\in(0,\pi)$ to equation \eqref{eq:ubiq} and for $k=3$ the only solution is $\alpha=\pi/2$ (see next item). 

\mn 2. See Proposition 2 of \cite{Gut}, or Lemma 4.8 of \cite{BLPT}.  
\end{proof}


\begin{remark}
{\rm 
Equation \eqref{eq:ubiq} appeared in the context of bicycle kinematics in \cite{Tab06,BLPT} and in the papers by Wegner, summarized in \cite{We07}. It also appeared in \cite{Gut} in the context of billiards and flotation problems,  
{and in \cite {BMS1}, \cite{BMS2}, \cite{BMT} for magnetic, outer and wire billiards.}
This ubiquitous equation has a countable number of solutions but, except for $\pi/2$, there are no $\pi$-rational solutions  \cite{Cyr}. 
}
\end{remark}

\subsection{Periods 3 and 4} \label{subsect:three}


\begin{theorem} \label{thm:three}
Let $\g(t)$ be a $\pi$-anti-periodic self-B\"acklund centroaffine curve, that is, $[\g(t),\g(t+\alpha)]=c\neq 0$. If $\alpha=\pi/3$ or $\alpha=\pi/4$ then $\g$  is a centroaffine ellipse.
\end{theorem}

\begin{proof} Consider the case of $\alpha=\pi/3$.
Let us use the shorthand notation
$$
\g(t)=\g_0,\quad \g\left(t+\frac{\pi}{3}\right)=\g_1, \quad \g\left(t+\frac{2\pi}{3}\right)=\g_2.
$$
Then
$$
[\g_0,\g_1]=[\g_1,\g_2]=[\g_2,-\g_0]=c,
$$
hence $[\g_0,\g_2]=[\g_0,\g_1]$, and the vector $\g_1-\g_2$ is collinear with $\g_0$. Likewise, $\g_2+\g_0$ is collinear with $\g_1$, and $\g_1-\g_0$ with $\g_2$. 
We write
$$
\g_1-\g_2 =\varphi_0 \g_0, \g_2+\g_0 = \varphi_1 \g_1, \g_1-\g_0 = \varphi_2 \g_2.
$$
%

%
Since $[\g_0,\g_1]\neq 0$, the  linear map $\R^3\to \R^2$, $(x_0,x_1,x_2)\mapsto \sum x_i\g_i$, has rank $2$, hence nullity $1$. It follows that  the matrix 
$$
\begin{bmatrix}
	-\varphi_0&1&-1\\
	1&-\varphi_1&1\\
	-1&1&-\varphi_2
	\end{bmatrix}
$$
has rank 1, hence $\varphi_0=\varphi_1=\varphi_2=1$. Thus $\g_2=\g_1-\g_0$.

It follows that $\g'_2=\g'_1-\g'_0$, and hence
$$
1=[\g_2,\g'_2]=[\g_1-\g_0,\g'_1-\g'_0]=2-[\g_0,\g'_1]+[\g'_0,\g_1].
$$
Since $[\g_0,\g_1]=c$, one has $[\g'_0,\g_1]+[\g_0,\g'_1]=0$. This implies that
$$
[\g_0,\g'_1]=\frac{1}{2},\ [\g'_0,\g_1]=-\frac{1}{2},
$$
and hence $\g_1 = (1/2)\g_0+c\g'_0$. 

It follows that in equation  \eqref{eq:fp} one has $f=1/2$, and hence, by Lemma \ref{lm:fp}, $c^2p=-3/4$. That is, $p$ is constant, which implies $p=-1$ and $c=\sqrt{3}/2$, and therefore the curve is a centroaffine conic.

The case $\alpha=\pi/4$ is similar. In analogous notations, one has
%
%
$$
[\g_0,\g_1]=[\g_1,\g_2]=[\g_2,\g_3]=[\g_3,-\g_0]=c,
$$
hence
$$
\g_0 \sim \g_1-\g_3, \g_1 \sim \g_0 + \g_2, \g_2 \sim \g_1 + \g_3, \g_3 \sim -\g_0 + \g_2.
$$
This implies 
\begin{equation} \label{g13}
\g_1 = g(\g_0+\g_2),\ \g_3=g(\g_2-\g_0)
\end{equation}
for some function $g(t)$.

Since $[\g_1,\g'_1]=[\g_3,\g'_3]=1$, equation \eqref{g13} implies
$$
2g^2=1,\ [\g_0,\g'_2]+[\g_2,\g'_0]=0.
$$
But $[\g_0,\g_2]=c$, hence $[\g'_0,\g_2]+[\g_0,\g'_2]=0$, and therefore $[\g'_0,\g_2]=[\g_0,\g'_2]=0$. In particular, $\g_2 \sim \g'_0$. 

It follows that $\g_1= (1/\sqrt{2}) \g_0 + c \g'_0$. Then, in equation  \eqref{eq:fp}, one has $f=1/\sqrt{2}$, and hence, by Lemma \ref{lm:fp}, $c^2p=-1/2$. Thus $p=-1, c=1/\sqrt{2}$, and the curve   is a centroaffine conic.
\end{proof}

\begin{remark}
{\rm An analogous result, rigidity for periods 3 and 4, holds for bicycle curves, see \cite{BMO1,BMO2,Tab06}.
}
\end{remark}

\subsection{Period two: flexibility and Radon curves} \label{subsect:period2}

In this section we show that  self-B\"acklund curves of period two, that is, $\alpha=\pi/2$, exhibit a substantial flexibility. A similar result, for the value of the density 1/2, was known for a long time for Ulam's flotation in equilibrium problem \cite{Au,Zi}. 

%
%
Let us construct a self-B\"acklund curve of period two as a closed trajectory of a vector field $V$ on the space of origin-centered parallelograms.  Let the vertices be $P_1,P_2,-P_1,-P_2$, and let the vector field have the values $V_1, V_2, -V_1, -V_2$ at these vertices, respectively.  

We want the trajectories of the  points $P_1,P_2,-P_1,-P_2$ to coincide and to form a self-B\"acklund curve with $\alpha=\pi/2$. 
Let $(P_1(t), P_2(t))$ be an integral curve of such a vector field. Then $P_2(t)=P_1(t+\pi/2)$. The centroaffine conditions $[P_i,P'_i]=1$ and the $c$-relation $[P_1,P_2]=c$ amount to
\begin{equation} \label{eq:onUV}
 [P_1,V_1]=[P_2,V_2]=1,\quad [V_1,P_2]+ [P_1,V_2]=0.
\end{equation}
Note that the area of the parallelogram $(P_1,P_2,-P_1,-P_2)$ remains constant.

\begin{lemma} \label{lm:lincomb}
Equations \eqref{eq:onUV} are satisfied if and only if 
$$
V_1=fP_1+ {1\over c}P_2,\  V_2=-{1\over c}P_1-fP_2,
$$
where $f(P_1,P_2)$ is an odd function, in the sense that $f(P_2,-P_1)=-f(P_1,P_2)$.
\end{lemma}

\begin{proof}
Write 
$V_1=fP_1+gP_2, V_2=\bar fP_1+\bar gP_2$ and substitute into equations \eqref{eq:onUV}, using $[P_1,P_2]=c$, to obtain $f+\bar g=0, g=-\bar f=1/c$. That $f$ is odd follows from the central symmetry of the parallelogram.
\end{proof}
Thus one has a functional parameter $f$ to play with. The boundary conditions 
\be\label{eq:boundary}
P_1(0)=(1,0), \  P_1\left(\frac{\pi}{2}\right)=P_2(0)=(0,c), \ P_2\left(\frac{\pi}{2}\right)=-P_1(0)=(-1,0)
\ee
%
%
impose a finite-dimensional restriction on the function $f$. As a result,  
we obtain a functional space of self-B\"acklund curves of period two.

For example, if $f$ is identically zero and $c=1$, then $P_1''= P_2'=- P_1$, and the curve is a centroaffine  ellipse. See Figure \ref{fig:period2} for a non-trivial example. In Example \ref{example:weg} below (Figure \ref{fig:deform}) we construct explicitly many analytic curves. 

\begin{figure}[ht]
\centering
\includegraphics[width=.35\textwidth]{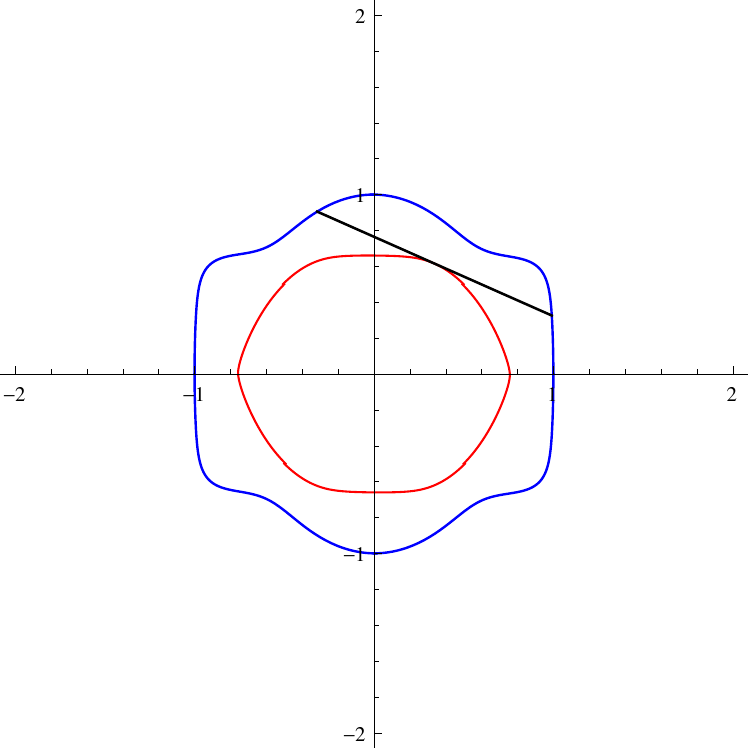}
\caption{A  self-B\"acklund curve with rotation angle $\alpha=\pi/2$ and $c=1$, using Lemma  \ref{lm:lincomb} and  equation \eqref{eq:boundary}, where   $f(P_1,P_2)=u(P_1)u(P_2)$ and $u(x,y)=
1.2x-4x^3-4x^5$  (approximately). 
}
\label{fig:period2}
\end{figure}

\begin{remark} \label{rmk:cont}
{\rm
The space of origin-centered parallelograms of a fixed area is identified with $\SLt$. If $P=(p_1,p_2), Q=(q_1,q_2)$, then the first equation \eqref{eq:onUV}, $[P,U]=[Q,V]$, means that the curve under consideration is tangent to the kernel of the 1-form $p_1dp_2-p_2dp_1+q_2dq_1-q_1dq_2$. This form defines a contact structure on $\SLt$, and the curve is Legendrian.
}
\end{remark}

Let $\G$ be a smooth closed convex curve, symmetric with respect to the origin. Let $x,y\in\G$. One says that $y$ is Birkhoff orthogonal to $x$ if $y$ is parallel to the tangent line to $\G$ at $x$. This relation is not necessarily symmetric; if it is symmetric, then $\G$ is called a {\it Radon curve}. Radon curves comprise a functional space, with ellipses providing a trivial example.

Radon curves have been thoroughly studied since their introduction more than 100 years ago; see \cite{MS} for a modern treatment.

Let $\G$ be a Radon curve, $x\in \G$ be a point, and $y\in\G$ be its Birkhoff orthogonal. Then the tangent lines at points $x,y,-x,-y$ form a parallelogram circumscribed about $\G$. As $x$ traverses $\G$, the vertices of the parallelogram describe a curve $\g$. The latter curve is an invariant curve of the outer billiard transformation about $\G$, see Remark \ref{rmk:outer}. 

The relation of self-B\"acklund curves with Radon curves is as follows. 
Let $\g$ be a self-B\"acklund curve with rotation number $\pi/2$, then the points $\g(t), \g(t+\pi/2), \g(t+\pi), \g(t+3\pi/2)$ form a parallelogram. Therefore the middle curve $\G$ is a Radon curve. Example \ref{example:weg} below provides analytic families of Radon curves.

\section{Self-B\"acklund curves and Lam\'e equation} \label{sect:Lame}
 
\subsection{Traveling wave solutions of KdV and Wegner's ansatz} \label{subsect:anzatz}

The first two in the hierarchy of integrals of the  Korteweg-de Vries equation are the functionals
\begin{equation} \label{eq:funct}
\int p(t) \ dt, \int p^2(t) \ dt
\end{equation}
on centroaffine curves. In particular,  KdV is the Hamiltonian flow of the former functional with respect to the symplectic form $\int [V_f,V_g]\ dt$, where we use  formula \eqref{eq:tang} for tangent vector fields \cite{Pin}.  

Consider a centroaffine curve that is a relative extremum of the second functional \eqref{eq:funct} subject to the constraint given by the first one. The next lemma is well known and we do not present its proof, see \cite{DrJ}. 

\begin{lemma} \label{lm:crit}
These relative extrema are characterized by the differential equation on the centroaffine curvature
%
%
\begin{equation} \label{eq:crit}
p'''= 6pp' + a p',
\end{equation}
where $a$ is a Lagrange multiplier.
\end{lemma}



Equation \eqref{eq:crit} describes traveling wave solutions of KdV, see  \cite{DrJ}. 
For the centroaffine curves satisfying equation \eqref{eq:crit}, the KdV evolution is described by the equation $\dot p = ap'$, that is, by a parameter shift of the curvature $p(t)$. Two centroaffine curves with the same curvature function differ by an element of $\SLt$. Therefore these curves evolve in time by special linear transformations. 

Equation \eqref{eq:crit} can be integrated to
\begin{equation} \label{eq:wave}
(p')^2 = 2p^3 + a p^2 + 2bp + c,
\end{equation}
where $a,b,c$ are constants.

\begin{lemma} \label{lm:manysol}
The curves described in Section \ref{subsect:infty} satisfy equation \eqref{eq:crit}.
\end{lemma}

\begin{proof}
Let $q(t)$ be the centroaffine curvature of the curve $f \g + c \g'$ where $\g$ is a unit circle and $f$ satisfies equation \eqref{eq:Ric}. Then
$$
q=\frac{2}{c^2} (f^2-1) -1,
$$
see Lemma 3.1 in \cite{Tab18} for this calculation. Hence 
$$
q' = \frac{4ff'}{c^2}= \frac{4f}{c^2} \left(\frac{f^2}{c} +c -\frac{1}{c}\right).
$$
One needs to check that $(q')^2 = 2q^3 +aq^2+2bq+c$ for some constants $a,b,c$. One has
$$
(q')^2 = \frac{16f^2}{c^4} \left(\frac{f^2}{c} +c -\frac{1}{c}\right)^2
$$
a cubic polynomial in $f^2$ with the leading coefficient $16/c^6$. The same holds for $2q^3 +aq^2+2bq+c$, so one can choose the coefficients $a,b,c$ as needed.
\end{proof}

Now we develop a centroaffine analog of F. Wegner's approach to 2-dimensional bodies that float in equilibrium in all positions (or bicycle curves) \cite{We07,We19,We}.

Consider a centroaffine curve $\g(t)=(r(t)\cos\alpha(t),r(t)\sin\alpha(t)).$
 The centoraffine condition $[\g,\g']=1$ is satisfied if $\alpha'=r^{-2}$. We use prime to denote the derivative with respect to $t$; the derivative with respect to $\alpha$ is denoted as $r_\alpha$.

Emulating Wegner's approach and using material of Section \ref{subsect:middle}, fix a small $\e$ and consider the curves $\G_{\pm} = \g \pm \e\g'$. These curves are $2\e$-related. We want them to be obtained from the same curve, $\G$, 
by rotating it through small angles $\pm\delta$. The assumption is that $\delta$ is of order $\e^3$; all the calculations below are mod $\e^4$.
We use the notations in Figure \ref{angles}.

\begin{figure}[ht]
\centering
\includegraphics[width=.6\textwidth]{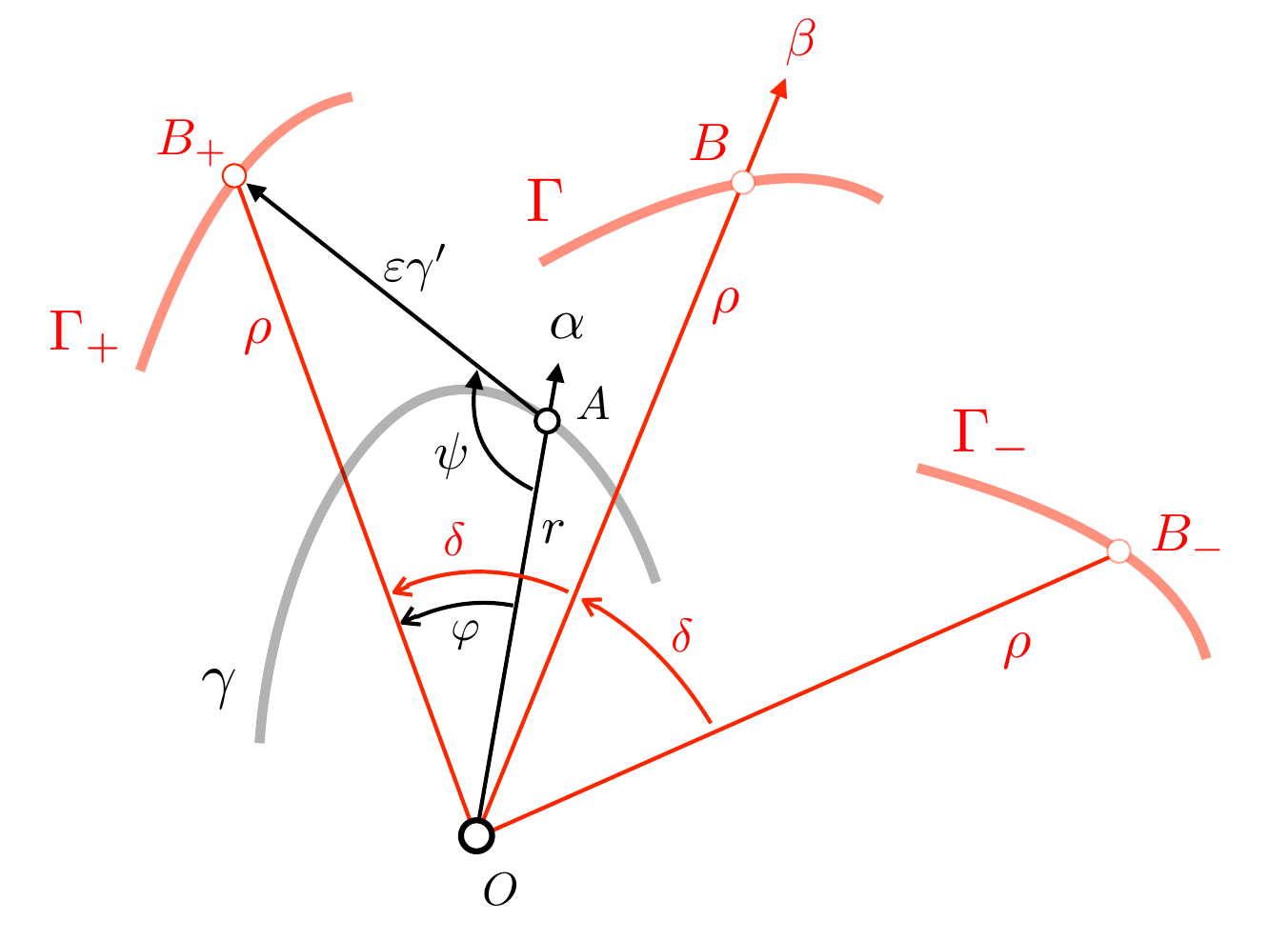}
\caption{Notation for Lemma \ref{lm:angles}: $r=|OA|, \ \rho=|OB_-|=|OB|=|OB_+|,\   \varphi=\angle AOB_+,  \  \psi=\angle OAB_+ ,\ \delta=\angle BOB_+=\angle B_-OB$. $\g$ and $\G$ are given in polar coordinates by $r(\alpha)$ and $\rho(\beta)$ (respectively). }
\label{angles}
\end{figure}

\begin{lemma} \label{lm:angles}
One has:
$$
\varphi=\tan^{-1} \left(\frac{\e}{r^2+\e rr'}\right),\ \rho=\sqrt{r^2+2\e rr'+\e^2(r^{-2}+r'^2)}.
$$
\end{lemma}

\begin{proof}
One has $|\g'|=r^{-1}\sqrt{1+r^2r'^2}$, hence $|AB_+|=\e r^{-1}\sqrt{1+r^2r'^2}$. Next,
$1=[\g,\g']=|\g||\g'|\sin\psi$, hence 
$$
\sin\psi=\frac{1}{\sqrt{1+r^2r'^2}},\ \cos\psi=-\frac{rr'}{\sqrt{1+r^2r'^2}}.
$$
Then
$$
\tan\varphi = \frac{|AB_+|\sin\psi}{|OA|-|AB_+|\cos\psi} = \frac{\e}{r^2+\e rr'}.
$$
Finally, by the cosine rule,
$$
|OB_+|^2=|OA|^2+|AB_+|^2-2|OA||AB_+|\cos\psi =r^2+2\e rr'+ \e^2(r^{-2}+r'^2),
$$
as claimed.
\end{proof}

Thus we have an equation for $\G$ in polar coordinates:
\begin{equation} \label{eq:main}
\rho(\beta)=\rho(\alpha+\varphi-\delta)=\sqrt{r^2+2\e rr'+\e^2(r^{-2}+r'^2)},
\end{equation}
where $\varphi$ is given in Lemma \ref{lm:angles}, and where $\delta=c\e^3$ with $c$ being a constant.

To solve equation \eqref{eq:main}, consider the cubic Taylor polynomials of both sides  and equate the even and odd parts separately (since the equation holds for $\pm\e$). 
One has
\begin{equation*}
\begin{split}
&\varphi=\e r^{-2}-\e^2 r^{-3}r'+\e^3 \left(r^{-4}r'^2-\frac{1}{3} r^{-6}\right), \\
&\varphi^2=\e^2 r^{-4}-2\e^3 r^{-5}r', \varphi^3=\e^3 r^{-6},\\
&\sqrt{r^2+2\e rr'+\e^2(r^{-2}+r'^2)}=r+\e r'+\frac{\e^2}{2}r^{-3}-\frac{\e^3}{2}r^{-4}r'.
\end{split}
\end{equation*}

To expand the left hand side of equation  \eqref{eq:main}, we calculate $\rho_\alpha, \rho_{\alpha\alpha}$ and $\rho_{\alpha\alpha\alpha}$, using $\alpha'=r^{-2}$:
$$
\rho_\alpha=r^2\rho', \rho_{\alpha\alpha}=2r^3r'\rho'+r^4\rho'',
\rho_{\alpha\alpha\alpha}=6r^4r'^2\rho'+2r^5r''\rho'+6r^5r'\rho''+r^6\rho'''.
$$
Now we have for the left hand side of equation  \eqref{eq:main}
\begin{equation*}
\begin{split}
\rho(\alpha+\varphi-\delta)& = \rho+\varphi \rho_\alpha+\frac{1}{2}\varphi^2 \rho_{\alpha\alpha}+\frac{1}{6}\varphi^3 \rho_{\alpha\alpha\alpha} - \delta \rho_\alpha=\\
&=\rho+\e\rho'+\frac{1}{2}\e^2\rho''+\frac{1}{6}\e^3 (r^{-2}\rho'''+2r^{-1}r''\rho'-2r^{-4}\rho') - c\e^3r^2\rho'.
\end{split}
\end{equation*}
Thus
\begin{equation*}
\begin{split}
&\rho+\frac{1}{2}\e^2 \rho''=r +\frac{1}{2}\e^2 r^{-3},\\
&\rho'+\frac{1}{6}\e^2 (\rho'''+2r^{-1}r''\rho'-2r^{-4}\rho'-6cr^2\rho')=r'-\frac{1}{2}\e^2 r^{-4}r'.
\end{split}
\end{equation*} 
Differentiate the first equation and subtract from the second one, setting, following Wegner, $\rho=r$ (since $\e$ is infinitesimal),
to obtain
$$
r'''-r^{-1}r'r''+4r^{-4}r'+3cr^2r'=0.
$$
Multiply this by $r^{-1}$ and write it as
$$
\left(r^{-1}r''-r^{-4}-\frac{3}{2}cr^{2}\right)'=0,
$$ 
or 
$$
r''-r^{-3}+\frac{3}{2}cr^3-br=0,
$$ 
where $b$ is a constant. Multiply this by $2r'$ and write it as
$$
\left(r'^2+r^{-2}+\frac{3}{4}cr^4-br^2\right)'=0.
$$
 Hence 
 $$
 r'^2=-r^{-2}-\frac{3}{4}cr^4+br^2+a,
 $$ 
 where $a$ is another constant. Multiply by $4r^2$ to obtain
$$
4r^2r'^2=-4-3cr^6+4br^4+ar^2.
$$ 
Finally, setting $R=r^2$ and renaming the constants, we obtain the differential equation
\begin{equation} \label{eq:W}
R'^2=aR^3+bR^2+cR-4.
\end{equation}
Thus $R(t)$ is an elliptic function. The curve is given by a parametric equation
\begin{equation} \label{eq:param}
\G(t)=(R(t)^{1/2}\cos\alpha(t), R(t)^{1/2}\sin\alpha(t))
\end{equation}
with  $R$ as in equation \eqref{eq:W} and $\alpha'=R^{-1}$. 

\begin{remark}
{\rm If the curve is a centroaffine ellipse, one has $a=0$ in  equation  \eqref{eq:W}.
}
\end{remark}

Concerning the centroaffine curvature of this curve, it is also an elliptic function.

\begin{lemma} \label{lm:ell}
One has
$$
p(t)=\frac{1}{2}aR(t)+\frac{1}{4}b.
$$
\end{lemma}

\begin{proof}
Differentiating equation  \eqref{eq:param} twice, we find that
$$
p=-\frac{1}{4}R^{-2} (R'^2+4) + \frac{1}{2} R^{-1}R''.
$$
Differentiating equation  \eqref{eq:W}, we obtain
$$
R''=\frac{3}{2}aR^2+bR+\frac{1}{2}c.
$$
Substitute this and equation  \eqref{eq:W} in the above formula for $p$ to obtain the result.  
\end{proof}

Renaming the constants again, we obtain from equation  \eqref{eq:W}
$$
p'^2=2p^3+ap^2+bp+c,\ 
$$
which coincides with equation  \eqref{eq:wave}.

Let us also calculate the (Euclidean) curvature $k$ of a curve satisfying equation  \eqref{eq:W}.

\begin{lemma} \label{lm:curv}
One has
$$
k=-\frac{4aR+2b}{(aR^2+bR+c)^{\frac{3}{2}}}.
$$
\end{lemma}

\begin{proof}
Since $t$ is the centroaffine  parameter, we have for the curvature
$$
k=\frac{[\gamma',\gamma'']}{|\gamma'|^3}=\frac{-p(t)}{|\gamma'|^3}.
$$
We have
$$ |\gamma'|=\sqrt{{r'}^2+r^2\alpha'^2}=\sqrt{\frac{{R'}^2}{4R}+\frac{1}{R}}=\sqrt{\frac{{R'}^2+4}{4R}}=\frac{\sqrt{aR^2+bR+c}}{2}.
$$
Hence 
$$
k=\frac{-8p(t)}{\sqrt{aR^2+bR+c}^3}=-\frac{4aR+2b}{(aR^2+bR+c)^{\frac{3}{2}}}.
$$
\end{proof}

Thus the curvature is a function of the distance from the origin. This is a special class of curves, studied in \cite{Cas,Sin}. One can think of these curves as the trajectories of a charge in a rotationally symmetric magnetic field whose strength is a function of the distance from the origin. Note that Wegner's curves also have this property: their curvature satisfies $k=ar^2+b$, where $a,b$ are constants.

Likewise one can interpret equation $\g''=p\g$ as Newton's Second  Law, that is, $\g(t)$ is the trajectory of a point-mass in a central force field whose potential $V$ is rotationally symmetric. By Lemma \ref{lm:ell}, and renaming the constants, one has $V(r)= ar^4+br^2+c$. Using conservation of energy and momentum, one can solve the equation of motion in quadratures.

\begin{remark}
{\rm Consider a particular case when $V$ is a pure 4th power of the distance, that is, the force is proportional to $r^3$. 
According to a corollary of the Bohlin theorem, see Theorem 5, Appendix 1 in \cite{Ar}, some trajectories in this field are the images of straight lines under the conformal transformation $w=z^{1/3}$. These are cubic curves, see Figure \ref{cubic}.
}
\end{remark}

\begin{figure}[ht]
\centering
\includegraphics[width=.4\textwidth]{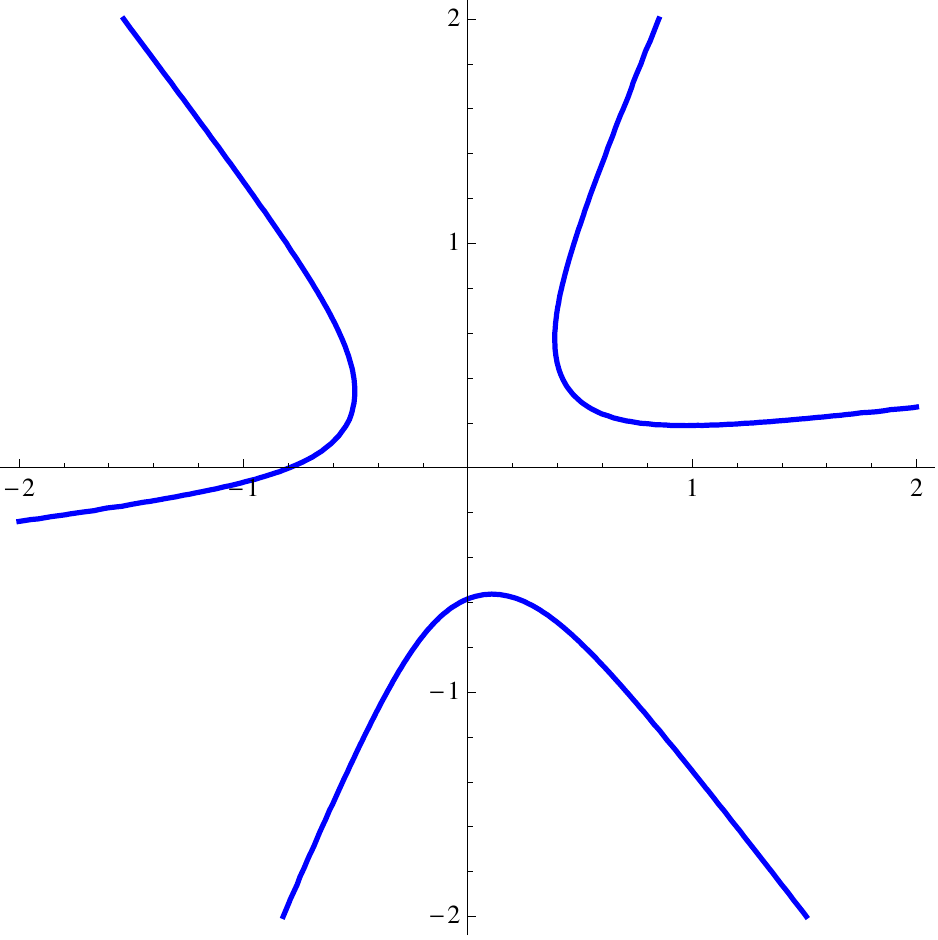}
\caption{The curve $2(x^3-3xy^2)-5(3x^2y-y^3)+1=0$, the image of the line $2a-5b+1=0$ under the conformal transformation $w=z^{1/3}$.}
\label{cubic}
\end{figure}

\subsection{Self-B\"acklund curves as solutions of the Lam\'e equation} \label{subsect:solLam}

In this section we give an explicit construction of a large family of self-B\"acklund curves, given by the  Wegner ansatz   of Section \ref{subsect:anzatz}. We shall make frequent use of standard facts about the Weierstrass elliptic functions $\wp,\zeta, \sigma,$ 
such as: the addition formulas \cite[pages 40-41]{Ak},  quasi-periodicity properties \cite[pages 35-37]{Ak}, reality conditions \cite[pages 29-32]{Pastras}, degenerate cases of Weierstrass functions \cite[pages 201]{Ak}.  We shall also use applications of elliptic functions to the Lam\'e equation which can be found in \cite[pages 48-54]{Pastras}.  We collected  most of the formulas and results that we are using in Appendix \ref{app:elliptic}.

\subsubsection{Constructing the curves} \label{subsect:constr}
Our starting point is equation \eqref{eq:wave}, 
$$(p')^2 = 2p^3 + a p^2 + 2bp + c,$$
 for the curvature $p(t)$ of the self-B\"acklund curves suggested by the Wegner's ansatz.
Comparing this equation to the equation satisfied by the Weierstrass $\wp$ function, 
\be\label{eq:wp}
(\wp')^2=4\wp^3-g_2\wp-g_3,
\ee
we conclude that $p(t)$ is given, in terms of $\wp$, by 
\begin{equation}\label{eq:pt}
p(t)=2\wp (t+\omega') +C.
\end{equation}
Here $\wp$ is the Weierstrass function with half periods $\omega, \omega',$ where the first one is real and the second one is pure imaginary, see Figure \ref{fig:periods}. Since $p(t)$ needs to be periodic, we are in the case of three real roots $e_1> e_2> e_3$ of  the right hand side of  
equation \eqref{eq:wp}.
In formula \eqref{eq:pt} the shift of the argument by $\omega'$ is performed in order to get a real, smooth, $2\omega$-periodic potential $p(t)$.

\begin{figure}[ht]
\centering
\includegraphics[width=\textwidth]{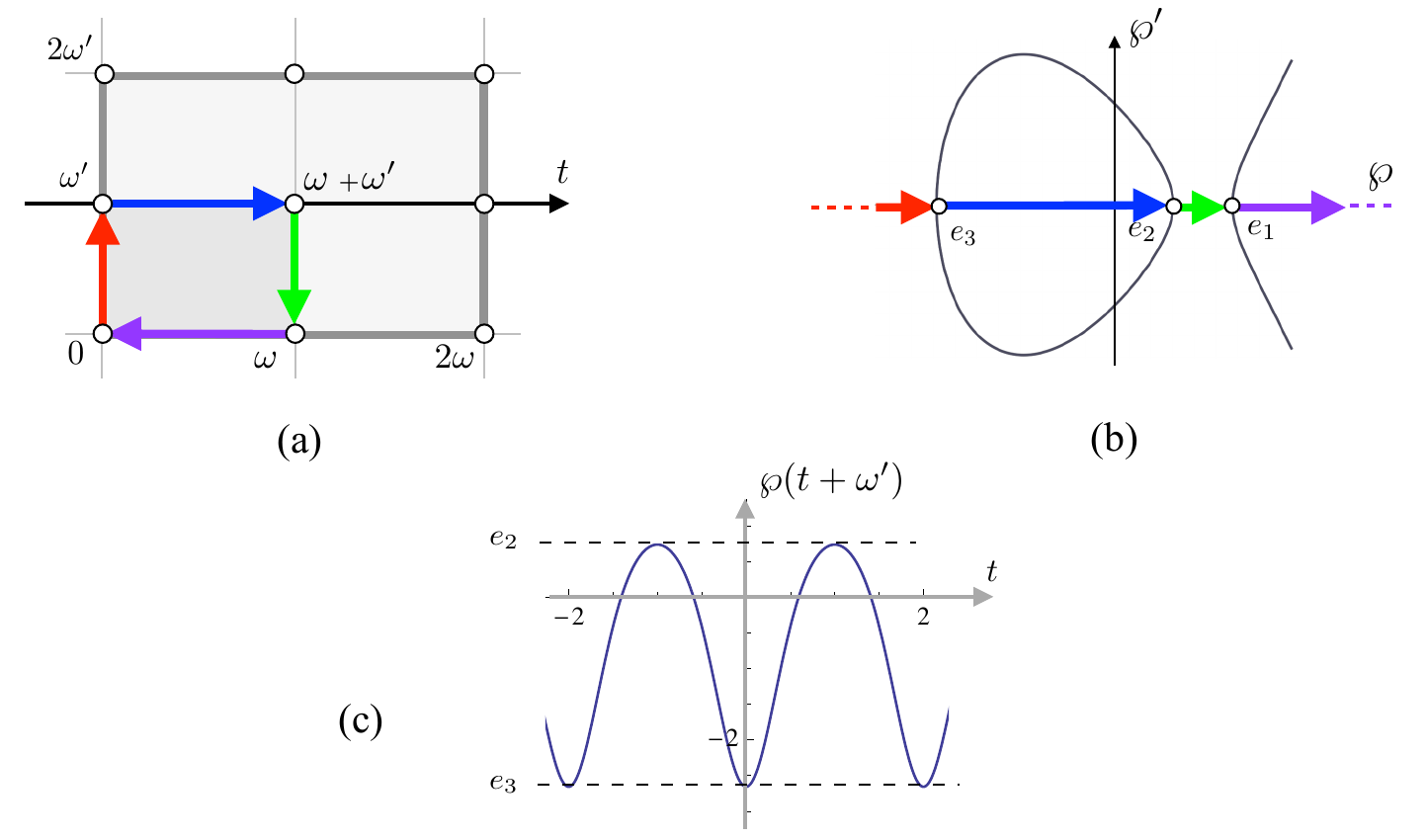}
\caption{
The Weierstrass  function $\wp(z)$ with  real invariants and fundamental half periods  $\omega\in\R, \omega'\in i\R.$ (a)  The fundamental rectangle  in the $z$ plane. The boundary of the rectangle $(0, \omega',\omega+\omega', \omega)$ is mapped by $\wp$ onto the extended real axis $\R\cup\{\infty\}$. (b)   The phase plane of  $(\wp')^2=4(\wp-e_1)(\wp-e_2)(\wp-e_3)$. 
 (c) The  line $\{t+\omega'| t\in\R\}$ is mapped, $2\omega$-periodically, onto the segment $[e_3,e_2]$. }
	\label{fig:periods}
\end{figure} 

The constant $C$ can be written as $C=\wp(a)$ for some $a\in \C$.
Thus 
\begin{equation}\label{p}
p(t)=2\wp (t+\omega')+\wp(a).
\end{equation}
 We write our curve in complex  form $X(t) =x(t)+iy(t),$ satisfying
\begin{equation}\label{Lame}
\ X''+(-\wp(a)-2\wp (t+\omega'))X=0,
\end{equation}
 which is precisely the Lam\'e equation (equation (6) of \cite[page 186]{Ak}).

In order to construct a centroaffine $\pi$-anti-periodic  curve, we shall require the following:

\begin{enumerate}
	\item
%
The Wronskian $[X,X']=1$.  This  can be achieved by rescaling  of any solution of  equation  \eqref{Lame} satisfying $[X,X']=const>0$ (see item \ref{item3} of Proposition \ref{pr:X} below).

	\item  $\omega=\pi/2k$ for some integer $k\geq 2$, so that $p$ is $\pi/k$-periodic. 

\item The solution $X$ is rotated over the period $2\omega$ by $\pi n/k$,
where $0<n<k$ is odd and co-prime to $k$, so that after $k$ periods we have $X(t+\pi)=-X(t)$. In other words, we  require  $X(t)$ to be a complex $2\omega$-quasi-periodic solution of equation  \eqref{Lame}, with Floquet multiplier   
 $\mu= e^{i\pi n/k}$:
	$$ \ X(t+2\omega)= X(t) e^{i\pi n/k}.$$
\end{enumerate}
A basis $X_+, X_-$ for the  solutions of the Lam\'e equation \eqref{Lame} can be written in the following form (see \cite[page 37]{Ak}):
\begin{equation}\label{Xpm}
X_{\pm}(t)=e^{-t\zeta(\pm a)}\frac{\sigma(\pm a+t+\omega')\sigma(\omega')}{\sigma(\pm a+\omega')\sigma(t+\omega')},
\end{equation}
where   $\zeta, \sigma$ are the Weierstrass zeta and sigma functions, respectively. 

The construction of the self-B\"acklund curves is this section boils down to a careful choice of the parameter $a$ in equation \eqref{Lame}.


\begin{proposition}\label {pr:X}
\note{}
	For every $a\in (0,\omega')\cup(\omega,\omega+\omega')$,

\se
	\item $\wp(a)$ is real, hence the potential $2\wp (t+\omega')+\wp(a)$ in the Lam\'e equation \eqref{Lame} is real as well. 
		\item $X_+(t)$ is a regular curve, that is, $X'_+(t)\neq 0$ for all $t$. 

		\item \label{item:b}$X_+(0)=1$ and $ X_+'(0)=ib$  for some $  b\in \mathbb R,\ b>0.$
%
%
		\item\label{item3}$X_+(t)$ is locally star-shaped and positively oriented: $$ [X_+(t),X_+'(t)]=const >0.$$
		\item \label{item4}
		$X_+(t+2\omega)=X_+(t)e^{2f(a)},$ where 
		\be\ f(a):=a\zeta(\omega)-\omega\zeta(a).\ee
That is, $X_+(t)$ is a $2\omega$-quasi-periodic solution of   equation  \eqref{Lame}  with a Floquet multiplier   $
\mu=e^{ 2 f(a)}.
$
\item The function $f$ of the previous item satisfies the identities 
	$$
f(-a)=-f(a),\ 	f(a+2\omega)=f(a),\  f(a+2\omega')=f(a)+i\pi.
	$$
	
	\end{enumerate}
	\end{proposition}

\begin{proof}
\begin{enumerate}[wide = 0pt, leftmargin = 0em]

\item See pages 31-32 of \cite{Pastras}.

\item Differentiating equation  \eqref{Xpm}, and using $\zeta=\sigma'/\sigma$ and the addition formula for $\zeta$, we compute:
$$X_+'(t)=X_+(t)\left[ \zeta(a+t+\omega')-\zeta(a)-\zeta(t+\omega')\right]=X_+(t)\frac{\wp'(a)-\wp'(t+\omega')}{2[\wp(a)-\wp(t+\omega')]}.
$$
Notice that the numerator in the last fraction cannot vanish, since $\wp'(t+\omega')$ is real and $\wp'(a)$ is purely imaginary, both non-vanishing  ($\wp'$ vanishes in the fundamental rectangle only at $0,\omega, \omega', \omega+\omega'$). It follows that $X_+'(t)$ does not vanish. 
	
\item  Substituting $t=0$ into equation  \eqref{Xpm} gives $X_+(0)=1.$ From the previous item we have 
$$X_+'(0)=\frac{\wp'(a)}{2(\wp(a)-e_3)}.
$$
For $a\in (0,\omega')\cup(\omega,\omega+\omega')$ the numerator $\wp'(a)$ is purely imaginary and the denominator is real, both non-vanishing. 
Hence we can write $X_+'(0)=ib, \ b\in \mathbb R,\ b\neq 0.$ Moreover, $\wp(a)<e_3$  and $\Im[\wp'(a)]<0$ for $a\in (0,\omega')$.  	
When $a\in (\omega,\omega+\omega')$ we have that $\wp(a)>e_3$ is positive and $\Im[\wp'(a)]>0$. (All this is   evident in Figure \ref{fig:periods}.) Hence, in both cases, $b>0$.
	 
\item  Since $X_+$ is a solution of Lam\'e equation \eqref{Lame}, which has no $X'$ term, one has
$${\rm Wronskian}=[X_+(t),X_+'(t)]=const.$$
The constant must be positive, due to item 2.
	
\item See \cite{Pastras} page 52. 
\item See \cite{Pastras} page 86. 
\end{enumerate}
\end{proof}
\begin{remark}\label{normalization}
Following   Proposition \ref{pr:X} (item \ref{item3}) and the proof of item \ref{item:b}, we  can normalize the solutions of the Lam\'e equation  \eqref{Lame} given by formula \eqref{Xpm} by the constant factor 
$$N:=\sqrt{|X'_\pm(0)|}=\sqrt{\frac{\wp'(a)}{2i(\wp(a)-e_3)}},$$  so that the normalized solutions $Y_\pm(t):=\frac{1}{N}X_\pm (t)$ satisfy the centro-affine condition $[Y(t), Y'(t)]=1$.
		\end{remark}
Next, due to requirement 3 and Proposition \ref{pr:X} (item \ref{item4}), we need to solve $2 f(a)\equiv i\pi n/k$ (mod $2\pi i$), or 
\begin{equation}\label{f}
 f(a)=\frac{i\pi n}{2k}+i\pi m,
\end{equation}
for some  integers $m,n\in\Z$, where $n$ is odd,  relatively prime to $k$, and $0<n<k.$


 To solve equation \eqref{f}, it is enough to restrict $a$ to the fundamental rectangle.
	Indeed, if $a_1$ and $a_2$ are two congruent solutions of equation  \eqref{f}, then the corresponding potentials \eqref{p} of Lam\'e equation are equal, and  the curves constructed by formula \eqref{Xpm} are equivalent under the action of $\SLt$.
	
 One may further restrict to solutions of equation \eqref{f} where $a$ belongs to one of the segments $(0, \omega')$ or $(\omega,\omega+\omega')$, and  $m\geq 0$. This follows from the properties of $f$ listed in Proposition \ref{pr:X}  and the monotonicity property of $f$ on the segments $[0, 2\omega']$ and $[\omega,\omega+2\omega']$. On the segment $[0, 2\omega']$ the function  $f$ varies monotonically from $+i \infty $ to $-i\infty $. On the segment $[\omega,\omega+2\omega']$ it varies from $0$ to $i\pi$.


%
%
\begin{theorem}\label{thm:am}
 Consider equation \eqref{f} for  fixed  integers $ k,n$, where   $k\geq 2$  and $n$ is odd, relative prime to $k$, and $0<n<k$. Then 
\begin{enumerate}
	\item 	For each  integer $m\geq 0$  there is a unique  solution $a_m\in  (0, \omega')\cup (\omega,\omega+\omega')$.
\item 
For $m>0$,  $a_m\in(0, \omega')$.
	\item For $m=0$, $a_0\in(\omega,\omega+\omega')$.
	\item The sequence $\lambda_m(\mu):=-\wp(a_m) $ is strictly monotone increasing and, in particular, the value $\lambda_0(\mu)=-\wp(a_0)$ is the smallest one.
\end{enumerate}	

\end{theorem}
\begin{proof}
	The proof of items 1--3 uses the behavior of the function $f$. Since  $\frac{\pi n}{2k}<\frac{\pi}{2}$, 
for $m=0$  there is a unique solution $a_0$ in the segment  $[\omega,\omega+\omega']$, because $f$ is pure imaginary on $[\omega,\omega+\omega']$ and varies monotonically from $0$ at $\omega$ to $i\pi/2$  at $\omega+\omega'$.

 For $m>0$, one can find a unique $a_m$ in the segment $[0, \omega']$ since there $f$ is pure imaginary, varying monotonically from  $+i\infty$ at $0$ to $i\pi/2$ at $\omega'$. Moreover, the  sequence $ a_m$ is monotone decreasing on $[0, \omega']$.
	
	In order to prove 4, notice that on the segment $[0, \omega']$ the function $\wp$ is real-valued  and monotone increasing from $-\infty$ to $e_3$. Hence $-\wp (a_m)$ is monotone increasing for $m\geq1$. Moreover, $-\wp (a_m)>-e_3$ for every $m\geq 1$. As for $m=0$, 
	$$-\wp(a_0)\in (-e_1, -e_2),$$
	because on the interval $[\omega,\omega+\omega']$ the function $\wp $
	is monotonically decreasing and takes the values $e_1, e_2$ at the end points, respectively.	
	Since $e_3<e_2<e_1$, this proves  item  4 (see Fig. \ref{fig:periods}).
	\end{proof}

Moreover we have the following result.

\begin{theorem}\label{embedded}For each $k,m,n$ as in Theorem \ref{thm:am}, consider the curve $X_+$ determined by the value $a_m$. 
	\begin{enumerate} 
		\item  $X_+$ is locally star-shaped $\pi$-anti-periodic curve, with the winding number $${\rm w}=2k\left\lceil \frac{m}{2}\right\rceil +n.$$	
	\item $X_+$ is embedded (simple)  if and only if $m=0, n=1.$
	\end{enumerate}
\end{theorem}
\begin{proof}

 It follows  from   Theorem \ref{thm:am} that the sequence $\lambda_m(\mu):=-\wp(a_m) $ is the sequence of Floquet eigenvalues 
for the problem  $$X'' +(\lambda-2\wp (t+\omega'))X=0, \ X(t+2\omega)=\mu X(t),\ \mu: =
e^{i\pi n/k},$$ 
and that $\lambda_m(\mu)$ is monotone increasing. 

 It follows from Proposition \ref{pr:X} that the curve is locally star-shaped and positively oriented.

In order to compute the winding number of the curve, we need first to see what happens over one period $[0,2\omega]$. Denote by $y_m(t)$ the imaginary part of the solution $X_+$ corresponding to $a_m$. We know by Proposition \ref{pr:X} (claim 2) that at the end points  of the period one has
$$
y_m(0)=0, \ y_m'(0)>0, \  y_m(2\omega)=\sin\left(\frac{\pi n}{k}\right)>0.
$$
This implies that the number of zeroes of $y_m$ on $(0, 2\omega]$ is even for every $m$.

In order to find the number of zeroes of $y_m$ on the interval $(0,2\omega)$ we use Sturm theory, comparing $y_m$ with the Dirichlet eigenfunctions of the Lam\'e equation, as follows. 

Let us denote by $\Lambda_m$, $\Psi_m, m\geq 0$, the eigenvalues and eigenfunctions corresponding to Dirichlet boundary conditions of the  equation 
\be\label{eq:Psi}
\Psi'' +(\lambda-2\wp (t+\omega'))\Psi=0.
\ee
Thus the eigenfunctions $\Psi_m$ vanish at the end points of the interval $[0,2\omega]$ and have  exactly $m$ zeros  in $(0,2\omega)$.

We claim that the number of zeroes of $y_m$ in $(0,2\omega)$ is given by the formula:
\begin{equation}\label{zeros}
\#\{t\in (0, 2\omega) :y_m(t)=0\}=2\left\lceil \frac{m}{2}\right\rceil.
\end{equation}
To prove this, we shall consider two cases (see Figure \ref{graph}):

\mn 1.  If $m=2l$ then $\Lambda_{2l-1}<\lambda_{2l}(\mu)<\Lambda_{2l}$. In this case, the zeroes of $\Psi_{2l-1}$ divide the interval into $2l$ subsegments. In each of them, $y_{2l}$ must vanish somewhere  (by Sturm theory). Hence there are at least $2l$ zeroes. In fact, this number must be exactly $2l$, because otherwise it would be at least $2l+2$ zeros ($y_m$ has an even number of zeroes). But then $\Psi_{2l}$ would have more than $2l$ zeroes.
\begin{figure}[ht]
	\centering
	\includegraphics[width=\textwidth]{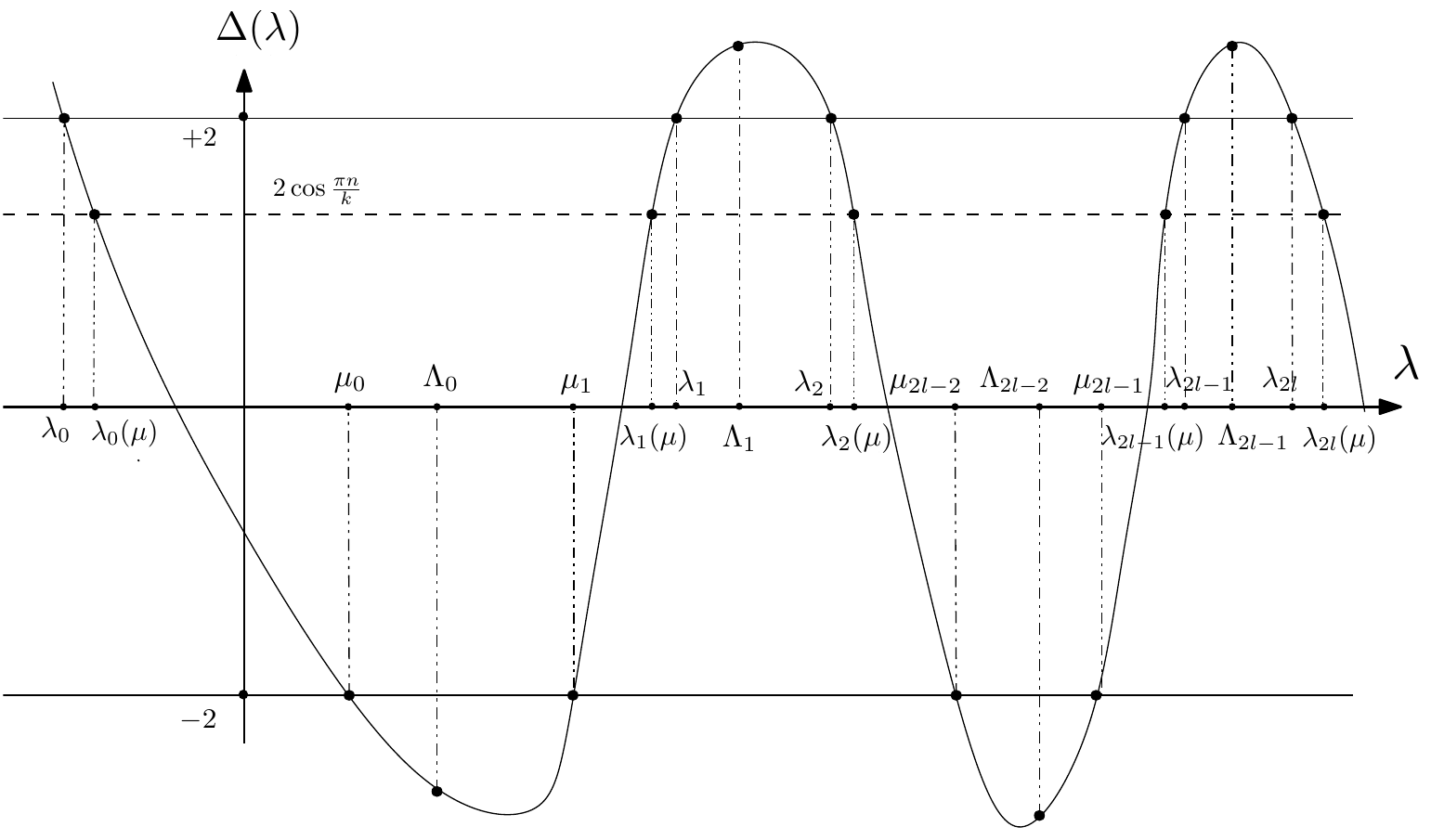}
	\caption{Graph of the function $\Delta(\lambda):=y_1({\lambda,2\omega})+y_2'(\lambda,2\omega)$, where $
		y_1(\lambda,t), y_2(\lambda,t)$ are the basic solutions of equation \eqref{eq:Psi} with $y_1(\lambda,0)=y_2'(\lambda,0)=1, \ y_1'(\lambda,0)=y_2(\lambda,0)=0;$
		the positions of the  periodic ($\lambda_n$), anti-periodic ($\mu_n$), Dirichlet ($\Lambda_n$), and Floquet ($\lambda_n(\mu)$) eigenvalues are indicated.}
	\label{graph}
\end{figure}

\mn 2.  If $m=2l+1$ then  $\Lambda_{2l}<\lambda_{2l+1}(\mu)<\Lambda_{2l+1}$. The zeroes of $\Psi_{2l}$ divide the interval into $2l+1$ subintervals, in each of which   $y_{2l+1}$ must vanish somewhere (by Sturm theory),
implying that $y_{2l+1}$ has at least $2l+1$ zeroes. But then this number is at least $2l+2$, because it is even.
Hence, the number of zeroes of $y_{2l+1}$ is exactly $2l+2$,  because otherwise $\Psi_{2l+1}$ would have more than $2l+1$ zeroes. This completes the proof of the claim.

As a consequence of  formula \eqref{zeros}, we see that for $a=a_m$ the solution $X_+$ makes $\lceil \frac{m}{2}\rceil$ full turns over the period $[0, 2\omega]$, plus an angle of $\frac{\pi n}{k}$, which is a $\frac{ n}{2k}$ fraction of a full turn. Altogether, after $2k$ periods, the  number of turns is
$$
{\rm w}=2k\left(\left\lceil \frac{m}{2}\right\rceil+\frac{ n}{2k}\right)=2k\left\lceil \frac{m}{2}\right\rceil +n.
$$
This proves the first claim of Theorem \ref{embedded}. 

The last formula implies that the curve is simple, that is, ${\rm w}=1$, if and only if 
$m=0, n=1$, proving the
second claim. This completes the proof.
\end{proof}

\subsubsection{Establishing the self-B\"acklund property} \label{subsect:est}


\begin{proposition}The curve $X_+$ of equation \eqref{Xpm} satisfies the self-B\"acklund property $[X_+(t),X_+(t+\alpha)]=const$ for a  value of the parameter $\alpha\in(0,\pi)$ if and only if 
\begin{equation}\label{sigma}
\sigma(a+\alpha)=e^{2\alpha\zeta(a)}\sigma(a-\alpha).
\end{equation}
\end{proposition}
\begin{proof}
Set  $\beta= \alpha/2.$
Then equation  \eqref{eq:selfbacklund} can be rewritten as
$$
\Im\left(X_+(t+\beta)\overline{X_+(t-\beta)}\right)=c,
$$
where overline denotes the  complex conjugation.
We can rewrite this equation as
$$
X_+(t+\beta)X_-(t-\beta)-X_-(t+\beta)X_+(t-\beta)=2c.
$$
Next we substitute in the last equation the expressions for $X_\pm$ from equation  \eqref{Xpm}:
\begin{align*}
2c=&e^{-(t+\beta)\zeta(a)}\frac{\sigma( a+t+\beta+\omega')\sigma(\omega')}{\sigma( a+\omega')\sigma(t+\beta+\omega')}
e^{(t-\beta)\zeta(a)}\frac{\sigma( -a+t-\beta+\omega')\sigma(\omega')}{\sigma( -a+\omega')\sigma(t-\beta+\omega')}-\\
&-e^{(t+\beta)\zeta(a)}\frac{\sigma( -a+t+\beta+\omega')\sigma(\omega')}{\sigma(- a+\omega')\sigma(t+\beta+\omega')}
e^{-(t-\beta)\zeta(a)}\frac{\sigma( a+t-\beta+\omega')\sigma(\omega')}{\sigma( a+\omega')\sigma(t-\beta+\omega')}.
\end{align*}
This can be simplified, using the identity \begin{equation}\label{addition}
\wp(z)-\wp(w)=-{\sigma(z-w)\sigma(z+w) \over \sigma^2(z)\sigma^2(w)}
\end{equation}
 (see \cite[page 25]{Pastras}).
We get
\begin{align*}
2c=&
e^{-2\beta\zeta(a)}\frac{\left[\wp(t+\omega')-\wp(a+\beta)\right]\sigma^2(a+\beta)
	\sigma^2(\omega')}{\left[\wp(t+\omega')-\wp(\beta)\right]\sigma^2(\beta)\sigma( a+\omega')\sigma(- a+\omega')}-\\
&
-e^{2\beta\zeta(a)}\frac{(\wp(t+\omega')-\wp(a-\beta))\sigma^2(a-\beta)
	\sigma^2(\omega')}{\left[\wp(t+\omega')-\wp(\beta)\right]\sigma^2(\beta)\sigma( a+\omega')\sigma(- a+\omega')}.
\end{align*}
Multiplying by the common denominator and renaming the constant,
$$\tilde c:=2c\sigma^2(\beta)\sigma( a+\omega')\sigma(- a+\omega')/\sigma^2(\omega'),
$$
we get 
\begin{align*}
\tilde c
\left[\wp(t+\omega')-\wp(\beta)\right]=&e^{-2\beta\zeta(a)}\left[\wp(t+\omega')-\wp(a+\beta)\right]\sigma^2(a+\beta)-\\
&-e^{2\beta\zeta(a)}\left[\wp(t+\omega')-\wp(a-\beta)\right]\sigma^2(a-\beta).
\end{align*}
Thus we must have
\begin{align*}
&\tilde c=e^{-2\beta\zeta(a)}\sigma^2(a+\beta)-e^{2\beta\zeta(a)}\sigma^2(a-\beta)\\
&\wp(\beta)\tilde c=e^{-2\beta\zeta(a)}\wp(a+\beta)\sigma^2(a+\beta)-e^{2\beta\zeta(a)}\wp(a-\beta)\sigma^2(a-\beta).
\end{align*}
Substituting $\tilde c$ from the first identity into the second and simplifying, we get 
$$
\sigma^2(a+\beta)\left[\wp(a+\beta)-\wp(\beta)\right]=e^{4\beta\zeta(a)}\sigma^2(a-\beta)\left[\wp(a-\beta)-\wp(\beta)\right].
$$
Now, using equation  \eqref{addition} again, we obtain 
$
\sigma(a+\alpha)=e^{2\alpha\zeta(a)}\sigma(a-\alpha),
$ as needed.
\end{proof}

The next theorem states the self-B\"acklund property of the curves $X_+$.

\begin{theorem}\label{th:selfbacklund}For each  $k,m,n$ as in Theorem \ref{thm:am}, the associated  curve $X_+$ satisfies the self-B\"acklund property
$[X_+(t),X_+(t+\alpha)]=const$  for $k-2$ values of $\alpha\in(0,\pi)$.
\end{theorem}
\begin{example}\label{example:weg}
 Let us look for solutions of  equation  \eqref{sigma}
of the form  $\alpha= l\omega,$ where $l$ is an integer. Using the quasi-periodicity property of $\sigma$ (see \cite[ page 37]{Ak}, \cite[ page 20]{Pastras}), we write 
\begin{align*}
\sigma(a+\alpha)&=\sigma(a+l\omega)=\sigma(a-\alpha+2l\omega)=(-1)^le^{2l\zeta(\omega)(a-\alpha+l\omega)}\sigma(a-\alpha)=\\
&=(-1)^le^{2la\zeta(\omega)}\sigma(a-\alpha).
\end{align*}
Comparing with equation  \eqref{sigma}, we require $(-1)^le^{2la\zeta(\omega)}=e^{2\alpha\zeta(a)}.$   We choose $l$ to be odd and require
$$2\alpha\zeta(a)=2l\omega\zeta(a)=2la\zeta(\omega)-i\pi.
$$
Hence  $f(a)=a\zeta(\omega)-\omega\zeta(a)=i\pi/2l.$ But, according to equation  \eqref{f}, $f(a)=i\pi n/2k + i\pi m.$	
Therefore, choosing $m=0, n=1, $ implies $l=k$, and so $\alpha=l\omega=k \pi/2k=\pi/2. $
In this way, we construct  an infinite family of self-B\"acklund simple closed curves with rotation number $\alpha=\pi/2$, as discussed in Section \ref{subsect:period2}, but now we have an analytical example. See Figure \ref{fig:weg2}.
\end{example}

\begin{figure}[ht]
	\centering
	\includegraphics[width=\textwidth]{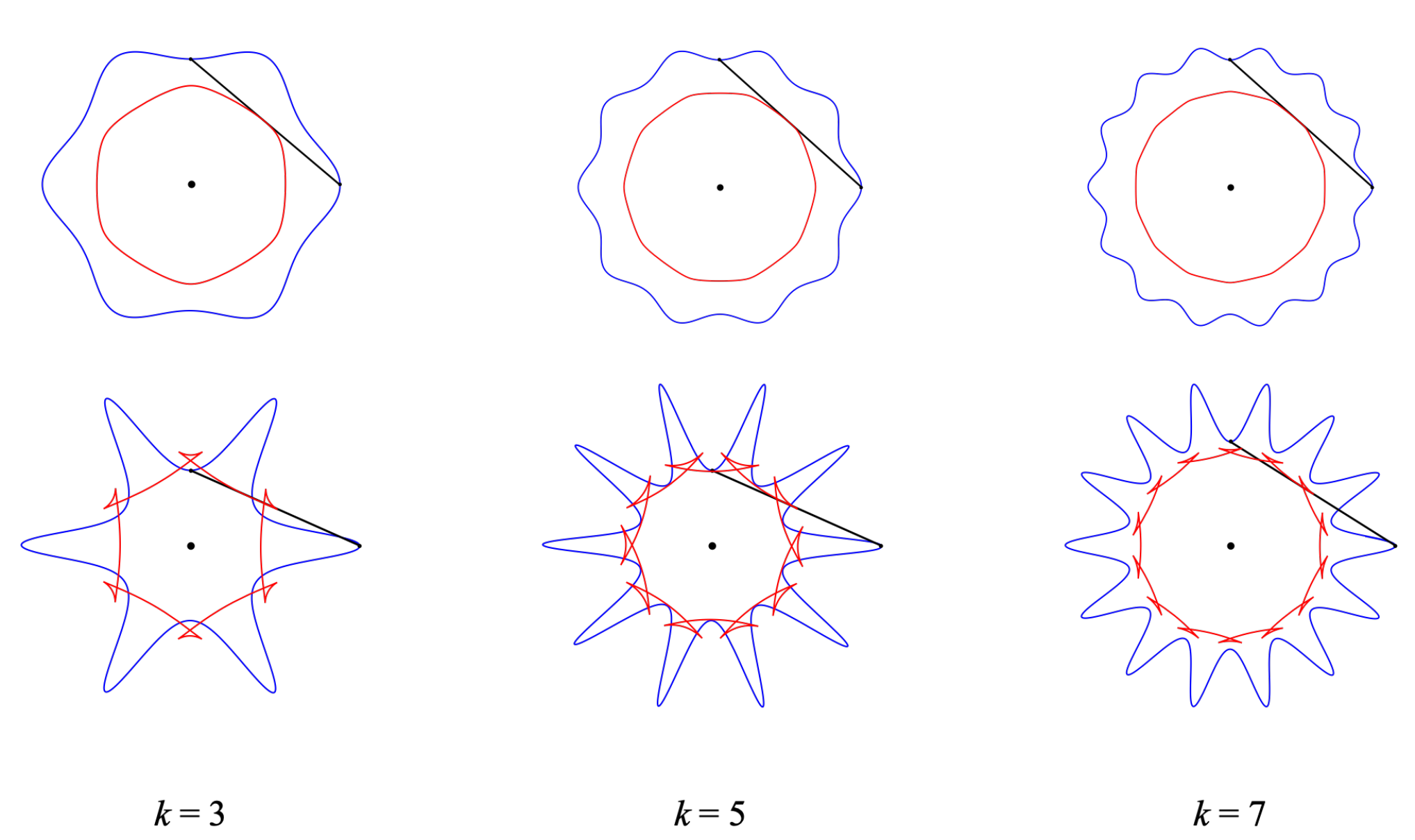}
	\caption{Example \ref{example:weg}.  Self-B\"acklund centroaffine simple curves $X_+(t)$ of Wegner type  (blue)  with $2k$-fold symmetry, $k=3,5,7$, with rotation number $\alpha=\pi/2$ (one quarter of a turn).  The red curve is traced by  the midpoint of the line segment $X_+(t)X_+(t+\pi/2)$ (black) and is tangent to it.  For  large enough $\omega'$, the midpoint curve  is smooth and convex (top);  as $\omega'$ becomes smaller, cusps appear (bottom).    }
	\label{fig:weg2}
\end{figure}

\subsubsection{Proof of the self-B\"acklund property (Theorem \ref{th:selfbacklund})} \label{subsect:pfsB}

We shall distinguish between two cases. In both cases we shall rewrite  equation \eqref{sigma} in a more tractable form. 

\mn {\bf Case 1}. Let us start with the most important case $m=0$ (the curve is simple if and only if $n=1$).
 For $m=0$ we have from 
 equation \eqref{f} that  $f(a)=\frac{i\pi n}{2k}$, where  $$a=\omega+ib \in  [\omega,\omega+\omega'], \ b\in\mathbb R.$$ 
We have from equation  \eqref{sigma} that
\begin{equation}\label{zeta1}
-\frac{\sigma(\alpha+\omega+ib)}{\sigma(\alpha-\omega-ib)}=e^{2\alpha
	\zeta(\omega+ib)}.
\end{equation}
Using the quasi-periodicity of $\sigma$, one has
$$
-\sigma(\alpha+\omega+ib)=\sigma(\alpha-\omega+ib)e^{2\zeta(\omega)(\alpha+ib)}.
$$
Substituting into equation  \eqref{zeta1}, we get
$$
\frac{\sigma(\alpha-\omega+ib)}{\sigma(\alpha-\omega-ib)}=
e^{2\alpha\zeta(\omega+ib)-2\zeta(\omega)(\alpha+ib)}=e^{{2\alpha[\zeta(\omega+ib)-\zeta(\omega)]-2i\zeta(\omega)b}},
$$
or, equivalently,
$$
-\frac{\sigma(\alpha-\omega+ib)}{\sigma(-\alpha+\omega+ib)}=e^{{2\alpha[\zeta(\omega+ib)-\zeta(\omega)]-2i\zeta(\omega)b}}.
$$
Taking $\log$,  we obtain
$$
i2\pi  l+\int_{-\alpha +\omega}^{\alpha-\omega}\zeta(ib+t)dt=i\pi+{2\alpha[\zeta(\omega+ib)-\zeta(\omega)]-2i\zeta(\omega)b}.
$$
Hence
\begin{equation}\label{zeta2}
\pi l +\Im\left(\int_{0}^{\alpha-\omega}\zeta(ib+t)dt\right)=\frac{\pi}{2}+{\frac{\alpha}{i}[\zeta(\omega+ib)-\zeta(\omega)]-\zeta(\omega)b}.
\end{equation}
Let us denote 
$$g(\alpha):= \Im\left(\int_{0}^{\alpha-\omega}\zeta(ib+t)dt\right).$$

\begin{lemma}\label{lemma:x}
For any $r\in \mathbb N\cup \{0\}$, we have 
	$$	\Im\left(\int_{0}^{2\omega r-\omega }\zeta(ib+t)dt\right)=(2r-1)b\zeta(\omega)-\pi r+\frac{\pi}{2}.
	$$
\end{lemma}
\begin{proof}
	Apply the Cauchy residue formula to the rectangular path
	$$
	-\omega(2r-1)+ib\rightarrow \omega(2r-1)+ib\rightarrow\omega(2r-1)-ib\rightarrow
	-\omega(2r-1)-ib\rightarrow -\omega(2r-1)+ib
	$$
to obtain the result.
\end{proof}

Using the quasi-periodicity of $\zeta$ and  Lemma \ref{lemma:x}, we have
\begin{equation*}
\begin{split}
g(\alpha +2\omega)&=\Im\left(\int_{0}^{2\omega r+\omega }\zeta(ib+t)dt\right) \\
&=\Im\left(\int_{0}^{2\omega r-\omega }\zeta(ib+t)dt\right)+\Im\left(\int_{-\omega}^{\omega }\zeta(ib+t)dt\right)\\
&=\Im\left(\int_{0}^{2\omega r-\omega }\zeta(ib+t)dt\right)+2\Im\left(\int_{0}^{\omega }\zeta(ib+t)dt\right)\\
&=\Im\left(\int_{0}^{2\omega r-\omega }\zeta(ib+t)dt\right)+
2b\zeta(\omega)-{\pi} =g(\alpha)+2b\zeta(\omega)-{\pi}.
\end{split}
\end{equation*}
Therefore we can write $g$ in the form 
\begin{equation}\label{g}
g(\alpha)=\left(\frac{2b\zeta(\omega)-{\pi}}{2\omega}\right)\alpha+h(\alpha),
\end{equation}
where $h$ is a $2\omega$-periodic function. 
Moreover, by Lemma \ref{lemma:x} (with $r=0$),
$$
h(0)=g(0)=-b\zeta(\omega) +\frac{\pi}{2}.
$$ 

It is convenient to use $h_0$ instead of $h$:
$$
h_0(\alpha):=h(\alpha)-h(0)=h(\alpha)+b\zeta(\omega) -\frac{\pi}{2}
,$$
so that $h_0$ is $2\omega$-periodic with $h_0(0)=0$.
Thus
\begin{equation}\label{g0}
g(\alpha)=\left(\frac{2b\zeta(\omega)-{\pi}}{2\omega}\right)\alpha+
h_0(\alpha)-b\zeta(\omega) +\frac{\pi}{2}.
\end{equation} 
Substituting equation  \eqref{g0} into equation  \eqref{zeta2}, we obtain the equation:
\begin{equation*}
\begin{split}
\pi l +\left(\frac{2b\zeta(\omega)-{\pi}}{2\omega}\right)\alpha+h_0(\alpha)-b\zeta(\omega) +\frac{\pi}{2}\\
=\frac{\pi}{2}+{\frac{\alpha}{i}[\zeta(\omega+ib)-\zeta(\omega)]-\zeta(\omega)b}.
\end{split}
\end{equation*}
This is the same as
\begin{equation}\label{zeta4}
\begin{split}
\pi l +h_0(\alpha) &=
\alpha\left(\frac{-2b\zeta(\omega)+{\pi}}{2\omega}+	\frac{(\zeta(\omega+ib)-\zeta(\omega))}{i}\right)\\
&=\alpha\left(\frac{\pi}{2\omega}+\frac{2\omega\zeta(\omega+ib)
	-2\omega\zeta(\omega)-2ib\zeta(\omega)}{2i\omega}\right)\\
&=\alpha\left(\frac{\pi}{2\omega}-\frac{2f(\omega+ib)}{2i\omega}\right)=
\alpha\left(\frac{\pi}{2\omega}-\frac{2f(a)}{2i\omega}\right).
\end{split}
\end{equation}
Taking into account that $f(a)=\frac{i\pi n}{2k}$ and $2\omega k=\pi$, we come to the final form of the equation:
\begin{equation}\label{final}
\pi l +h_0(\alpha) =\alpha(k-n).
\end{equation}

We claim that
	equation \eqref{final}	has  at least $k-n-1$ solutions for $\alpha$ in the open interval $(0, \pi)$.

Indeed, 
	since $h_0(0)=h_0(\pi)=0$,  the end points $\alpha=0, \alpha=\pi$ of the open interval are solutions of equation \eqref{final} for $l=0$ and $l=k-n$, respectively. (These two solutions are geometrically trivial, corresponding to $\alpha=2\beta=0$ and $\alpha=2\beta=\pi$ for the initial geometric problem.)
	Therefore, for all intermediate levels of $l$, that is, for  $l\in [1,k-n-1]$, there exists a solution of equation \eqref{final}. 
This proves the claim.

We shall prove now that the number of solutions of  equation \eqref{final}  in the interval $(0,\pi)$ is exactly equal to  $(k-n-1)$. 
For equation \eqref{zeta2}, it suffices to show that  the function 
$$\Im\left(\int_{0}^{\alpha-\omega}\zeta(ib+t)dt\right)-\frac{\alpha}{i}[\zeta(\omega+ib)-\zeta(\omega)]$$
has non-vanishing derivative with respect to $\alpha$. 
Arguing by contradiction, suppose that
$$
\Im\left(\zeta(ib+{\alpha-\omega})-[\zeta(\omega+ib)-\zeta(\omega)]\right)=0.
$$
Notice that $\zeta(\omega)$ is real, and $\zeta(\omega+\alpha +ib)$ and $\zeta(-\omega+\alpha +ib)$ have the same imaginary part. Hence 
\begin{equation}\label{derivative2}
\Im\left(\zeta(ib+{\alpha+\omega})-\zeta(\omega+ib)\right)=0.
\end{equation}
Using the addition formula, we have 
$$
\zeta(ib+\omega +\alpha)=\zeta(ib+\omega)+\zeta(\alpha)+\frac{\wp'(ib+\omega)-\wp'(\alpha)}{2(\wp(ib+\omega)-\wp(\alpha))}.
$$
It then follows from equation  \eqref{derivative2} that
$$
\zeta(\alpha)+\frac{\wp'(ib+\omega)-\wp'(\alpha)}{2(\wp(ib+\omega)-\wp(\alpha))}\in \mathbb R.
$$
Moreover, the values $\zeta (\alpha),\ \wp(ib+\omega),\ \wp(\alpha),\ \wp'(\alpha)$ are all real. We conclude that
$
\wp'(ib +\omega)\in\mathbb R.
$

On the other hand, 
$$ib+\omega\in(\omega,\omega')\Rightarrow e_2<\wp(ib+\omega) <e_1.$$
Thus   the equation
$
(\wp')^2=4(\wp-e_1)(\wp-e_2)(\wp-e_3)
$
implies that
$\wp'(ib+\omega)\in i\mathbb R,$
a contradiction. This completes the proof of Theorem \ref{th:selfbacklund} in Case 1.

\mn {\bf Case  2}. In this case $m>0$, $a=ib\in[0, \omega'],  b\in\mathbb R$ . 
Using $\frac{\sigma'}{\sigma}=\zeta$, we write 
$$
\sigma(z)=\sigma(z_0)\exp\left(\int_{z_0}^{z}\zeta(t)dt\right).
$$
Taking $\log$, we  rewrite equation  \eqref{sigma} in the form 
$$
\int_{-\alpha}^{\alpha}\zeta(ib+t)dt+2\pi il=2\alpha \zeta(ib),\ l\in\mathbb Z.
$$
Using that $\zeta$ is odd,  rewrite this as
$$
2\pi il+\int_{0}^{\alpha}[\zeta(ib+t)-\zeta(-ib+t)]dt=2\alpha \zeta(ib).
$$
Notice that both sides of this equation are purely imaginary, and hence 
\begin{equation}\label{zeta}
\pi l+\Im\left(\int_{0}^{\alpha}\zeta(ib+t)dt\right)=\frac{1}{i}\alpha\zeta(ib).
\end{equation}
On the right hand side  we have a linear function of $\alpha$. 
Let us denote the integral on the left hand side of  equation  \eqref{zeta} by
$$g(\alpha):= \Im\left(\int_{0}^{\alpha}\zeta(ib+t)dt\right).$$
\begin{lemma}For any $r\in \mathbb N$, we have 
	$$	 \Im\left(\int_{0}^{2\omega r }[\zeta(ib+t)dt\right)=-\pi r+ 2r\zeta(\omega)b.
	$$
\end{lemma}
\begin{proof}
	This  follows from the residue formula for the rectangular path 
	$$ib\rightarrow 2\omega r+ib\rightarrow 2\omega r-ib\rightarrow -ib\rightarrow ib,$$
	 avoiding the singular points of $\zeta $ at $0$ and $2\omega r$ by small half circles.
\end{proof}

In particular, using this lemma for $r=1$ and the quasi-periodicity of $\zeta$, we compute $$
g(\alpha +2\omega)=g(\alpha)+\frac{1}{i}\int_{0}^{2\omega}\zeta(ib+t)dt=g(\alpha)-\pi+2\zeta(\omega)b.
$$
Using this, one can expresses $g$ as the sum of a linear and a $2\omega$-periodic function as follows:
$$
g(\alpha)=\left(\frac{-\pi+2\zeta(\omega)b}{2\omega}\right) \alpha+h(\alpha),
\ g(0)=h(0)=0,
$$ 
where $h$ is $2\omega$-periodic.
Therefore, equation \eqref{zeta} takes the form
$$
\pi l+h(\alpha)=-\left(\frac{-\pi+2\zeta(\omega)b}{2\omega}\right) \alpha
+\frac{1}{i}\alpha\zeta(ib),$$ hence
\begin{equation*}\label{l}
\pi l+h(\alpha)=\alpha \left(\frac{1}{i}\zeta(ib)-\frac{-\pi+2\zeta(\omega)b}{2\omega}\right). 
\end{equation*}
Thus we arrive at the following equation
\begin{equation*}\label{l1}
\pi l+h(\alpha)=\alpha \left(\frac{\pi}{2\omega}+\frac{2\omega\zeta(ib)-2\zeta(\omega)ib}{2\omega  i}\right)=\alpha\left(\frac{\pi}{2\omega}-\frac{2f(ib)}{2\omega i}\right). 
\end{equation*}

Next, taking into account that $f(ib)=f(a)=\frac{i\pi n}{2k}$ and $2\omega k=\pi $, we obtain the simplest possible form:
\begin{equation}\label{l2}
\pi l+h(\alpha)=\alpha (k-n).
\end{equation}
Also in this case we claim that
equation \eqref{l2}	has  at least $k-n-1$ solutions for $\alpha$ in the open interval $(0, \pi)$.
	
	Indeed, 
	since $h(0)=h(\pi)=0$,  the end points $\alpha=0, \alpha=\pi$ of the open interval are solutions of equation \eqref{l2} for $l=0$ and $l=k-n$, respectively. 
	Therefore, for all intermediate levels of $l$, that is, for  $l\in [1,k-n-1]$, there exists a solution of equation \eqref{l2}. 
	This proves the claim.
	
We shall prove now that
 the number of solutions of equations  \eqref{l2} in the interval $(0,\pi)$  equals exactly  $k-n-1$. 
Consider  equation \eqref{zeta}.
	We shall check that the function 
	$$\Im\left(\int_{0}^{\alpha}\zeta(ib+t)dt-\alpha\zeta(ib)\right)$$	
	has everywhere non-vanishing derivative with respect to $\alpha$ when $ib\in(0,\omega')$. 
	
	Suppose, on the contrary, that  the derivative vanishes for some $\alpha$:
	\begin{equation}\label{derivative}
	\Im\left(\zeta(ib+\alpha)-\zeta(ib)\right)=0.
	\end{equation}
	Using the addition formula for  $\zeta$, we have 
	$$
	\zeta(ib +\alpha)=\zeta(ib)+\zeta(\alpha)+\frac{\wp'(ib)-\wp'(\alpha)}{2(\wp(ib)-\wp(\alpha))}.
	$$
	Taking the imaginary part and using equation  \eqref{derivative}, we obtain
	$$
	\zeta(\alpha)+\frac{\wp'(ib)-\wp'(\alpha)}{2(\wp(ib)-\wp(\alpha))}\in \mathbb R.
	$$
	Also we know that $\zeta (\alpha),\ \wp(ib),\ \wp(\alpha),\ \wp'(\alpha)$
	are all real. Therefore we conclude that 
	$
	\wp'(ib)\in \mathbb R.
	$
	But, on the other hand, $\wp$ satisfies the equation
	$
	(\wp')^2=4(\wp-e_1)(\wp-e_2)(\wp-e_3).
	$
	Moreover, 
	$$ib\in(0,\omega')\Rightarrow\wp(ib) <e_3\Rightarrow\wp'(ib)\in i\mathbb R.$$ 
	This  contradiction completes the proof in Case 2.
\qed

\mn 

The preceding theorem has the next corollary.

\begin{corollary}\label{transversal}
	All the solutions of equation \eqref{sigma} are transversal and hence change smoothly as one varies the parameter $\omega'$
	of the elliptic functions involved.
\end{corollary}

\subsection{Self-B\"acklund curves as deformations of conics} \label{subsect:deform}


In this section we use  the \sB curves  of Section \ref{subsect:solLam} in order to construct genuine non-trivial \sB  deformations of a central conic, as was promised in  Section \ref{subsect:infinitesimal}, see Corollary  \ref{cor:deform} below. 

To state the result, we recall  briefly from  Section \ref{subsect:solLam} our construction of {\em simple} self-B\"acklund centroaffine $\pi$-anti-periodic curves. For every  integer $k\geq 3$ and $\omega'\in i\R_+$ one  considers the Weierstrass $\wp$-function with half periods $\omega=\pi/2k$, $\omega',$  the associated $\sigma$- and $\zeta$-functions  and   the (unique) solution $a\in(\omega,\omega')$   to   
\be\label{eq:a}
a\zeta(\omega)-\omega\zeta(a)=i\omega,
\ee
then set 
\be\label{eq:Y}
Y(t):=X(t)/N, 
\ee
where
\be\label{eq:XN}
X(t):=\frac{\sigma(a+t+\omega')\sigma(\omega')}{\sigma(a+\omega')\sigma(t+\omega')}e^{-t\zeta(a)}, \ N:=\sqrt{|X'(0)|}. 
\ee

\begin{remark}
 The normalization factor $N=\sqrt{|X'(0)|}$ in equations \eqref{eq:Y}-\eqref{eq:XN} is introduced  so as to render the normalized curve $Y$ centroaffine and $\pi$-anti-periodic (enclosing area $\pi$). See Remark \ref{normalization} for an explicit expression for $N$.  

\end{remark}

The deformations of the unit circle we are seeking are obtained by fixing $k$ and letting $\omega'\to \infty$ in the above construction.  To examine this limit we  let $\omega'=i/s,$ $s\in(0,1]$, and use henceforth the subscript $s$ to denote all   associated objects, such as $\wp_s,\sigma_s, \zeta_s, a_s, X_s, N_s$ and $Y_s$ (suppressing the dependence on $k$, which is fixed throughout the section). Our goal in this section    is  to prove the following theorem, illustrated in Figure \ref{fig:deform}. 
\begin{theorem}\label{thm:cdef} For each integer $k\geq 3$,
\begin{enumerate}
\item  The family of curves  $Y_s(t)$, $s\in(0,1]$,  given by equations \eqref{eq:a}-\eqref{eq:XN} with $\omega=\pi/2k,\ \omega'=i/s,$ extends  smoothly to $s\in[0,1]$ by setting $Y_0(t):=e^{it}.$

\item Each curve $Y_s(t)$ is a centroaffine $\pi$-anti-periodic simple curve with $2k$-fold symmetry, $Y_s(t+\pi/k)=Y_s(t)e^{i\pi/k},$ \sB for $s>0$  with respect to $k-2$ rotation numbers $\alpha\in(0,\pi)$, varying smoothly  in  $s\in[0,1]$ and converging as $s\to 0$ to the $k-2$ solutions of equation \eqref{eq:ubiq},  $\tan(k\alpha)=k\tan\alpha$. 

\item  The   deformation $Y_s$, $s\in[0,1]$,     is analytic away from $s=0$ but not at $s=0$. In fact, one has  $\left.{\left(\partial_s\right)^n}\right|_{s=0}Y_s(t)=0,$ $n\geq 1$, so the associated infinitesimal deformation of the unit circle vanishes to all orders, yet the deformation itself is non-trivial. 
\item The  change of parameter,
\be\label{eq:tau}
		\e:=\left\{\begin{array}{ll}e^{-2k/s}, & s>0,\\
		& \\
		0, &s=0,
		\end{array}
		\right.
\ee
gives a  deformation   $Y_\e$ of the unit circle $Y_0$, analytic in $\e\in[0, e^{-2k}].$  
\item The infinitesimal deformation associated with the analytic deformation $Y_\e$ is non-trivial. That is, $$Y_\e(t)=e^{it}+Y_1(t)\e+O(\e^2),$$ where $Y_1$ is non-vanishing.

\end{enumerate}
\end{theorem}

\begin{figure}[h]
 	\centering
 	\includegraphics[width=\textwidth]{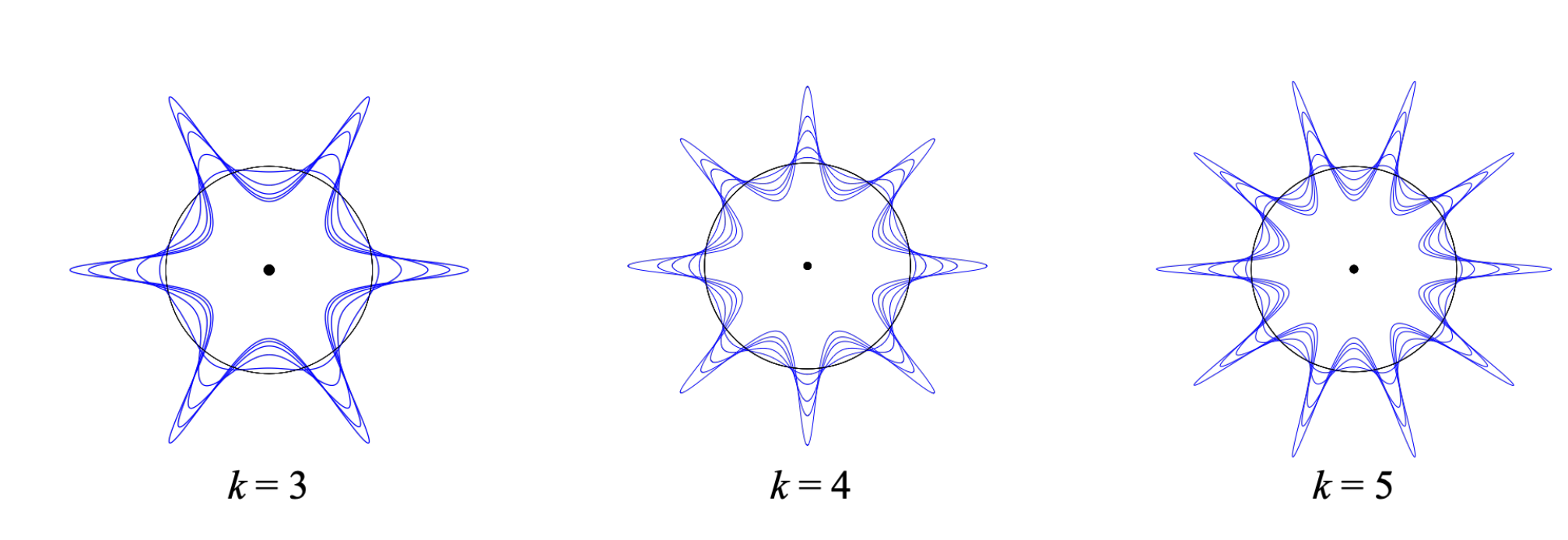}
 	\caption{Theorem \ref{thm:cdef}. 
	Three  families of deformations of the circle (black) through a 1-parameter family of centroaffine self-B\"acklund curves $Y_s$ (blue) with $2k$-fold symmetry, $k=3,4,5$. }
 	\label{fig:deform}
 \end{figure} 

\begin{proof}The main idea  of the proof of this theorem is to write the functions $X_s, s\in[0,1]$, as suitably normalized Floquet eigenfunctions of a Hill operator depending  smoothly on $s$, and use a general argument of smooth dependence of the eigenfunctions of a Hill operator  depending on the smooth parameter. Similarly,    when replacing $s$ with $\e$ the Hill operator depends analytically on $\e$ and so do its eigenfunctions. 

In more detail, we recall from Section \ref{subsect:solLam} that  $X_s$, $s\in(0,1]$, is precisely the  eigenfunction corresponding to the smallest eigenvalue $\lambda_{0,s}$ for the Floquet problem 
		\begin{equation}\label{eq:F}
			X''+ (\lambda-2q_s(t))X=0,\ X(t+\pi/k)=\mu X(t),\ \mu=e^{i{\pi/k}}, 
		\end{equation} 
where $q_s(t)=\wp(i/s+t)$  and $X_s$ satisfy the normalization condition $X_s(0)=1.$  Moreover, we showed that $\lambda_{0,s}=-\wp_s(a_s)$,  where $a_s\in (\omega, \omega')$  is the (unique) solution   to equation \eqref{eq:a}. 


\mn

Following this idea,  we begin by extending $q_s$ smoothly to $s=0$. 
\begin{lemma}\label{lemma:qs}
\begin{enumerate}[(A)]
	
	\item The function 
	$$q_s(t):=\left\{\begin{array}{ll} 
\wp_s(t+i/s), &s\neq 0,\\ 
&\\
-k^2/{3}, & s=0. \end{array}\right.
$$
 depends smoothly on $(s,t)\in[0,1]\times\R$.
	\item The  change of parameter $s\mapsto\e$ of equation \eqref{eq:tau} transforms the deformation $q_s$ to $q_\e$ which is real analytic in $\e\in [0,e^{-2k}]$,   with Taylor series
	\be\label{eq:qeps}
	q_\e=-\frac{k^2}{3}-8k^2\cos(2kt)\e+O(\e^2).
	\ee
	
	\end{enumerate}
\end{lemma}

We postpone the proof of this lemma, as well as the following three, to the end of this section. 

	\begin{lemma}\label{lemma:416}
		The eigenfunctions $X_s(t)$, $s\in[0,1]$,  corresponding to the  first eigenvalue $\lambda_{0,s}$ of the Floquet problem \eqref{eq:F}, are uniquely determined by the condition 
		$X_s(0)=1$ and are smooth (analytic) in $s$ if the potential $q_s$ is smooth (analytic) in $s$. 
	\end{lemma}

\begin{lemma}\label{lemma:wegdef}For every $s\in[0,1]$ the curves   $Y_s$  are self-B\"acklund for $k-2$ values of $\alpha_s\in(0,\pi)$, satisfying 
	\begin{equation}\label{sigmas}
		\frac{\sigma_s(a_s+\alpha_s)}{\sigma_s(a_s-\alpha_s)}=e^{2\alpha_s\zeta_s(a_s)}.
	\end{equation}
	All $k-2$ solutions $\alpha_s$ depend smoothly on $s\in [0,1]$.
	For $s=0$ this equation  reduces to equation \eqref{eq:ubiq} of  Theorem \ref{thm:infdef}, 
	$
	k\tan(\alpha)=\tan(k \alpha).
	$
	Moreover, with respect to the parameter $\e$ of equation \eqref{eq:tau} the $k-2$ families $\alpha_\e$ are analytic in $\e\in [0, e^{-2k}]$. 
	
	\begin{lemma}\label{lemma:419}
	$X_\e$ has  a Taylor series in $\e$, 
	$$X_\e(t)=e^{it}+X_1(t)\e+O(\e^2),$$
	where $X_1$ is non-vanishing. 
	\end{lemma}
\end{lemma}

With these 4 lemmas the proof of the 5 items of Theorem \ref{thm:cdef} is straightforward: by Lemma \ref{lemma:qs}, the Hill operator of equation \eqref{eq:F} is smooth in $s\in[0,1]$ and analytic in $\e\in[0,e^{-2k}].$ 
This implies, by Lemma   \ref{lemma:416}, that $X_s$ is smooth in  $s$ and $X_\e$ is analytic in $\e$, therefore the same holds for $Y_s$ and $Y_\e$. This proves items 1 and 4 of 
Theorem \ref{thm:cdef}. 
Lemma \ref{lemma:wegdef} proves item 2. Item 3 follows from the well known fact that $\e(s)$ of formula \eqref{eq:tau} is ``flat'' at $s=0$ (all  derivatives exist and vanish). Lemma \ref{lemma:419} gives item 5.
\end{proof}

\begin{corollary}\label{cor:deform}
For every value of $\alpha\in(0,\pi)$ for which the unit circle admits a non-trivial infinitesimal \sB deformation (solution of $\tan(k\alpha)=k\tan\alpha$ for some $k\geq 3$) there is a genuine analytic \sB deformation   realizing it. 
	\end{corollary}

We now proceed to the promised proofs of the four lemmas appearing in the above proof of Theorem \ref{thm:cdef}.

\subsubsection{Proof of Lemma  \ref{lemma:qs}.}

By the definition of $\wp$,  we have the following series representing $q_s$ for $s>0$:
\begin{align}\label{seriesq}
\begin{split}
	q_s(t)&=\wp_s(t+i/s)=\left(t+i/s\right)^{-2}+\\
	&
	+\sum_{(m,n)\neq (0,0)}\left[\left(t+{\pi n\over k} +i{2m+1\over s}\right)^{-2}-
	\left({\pi n\over k} +i{2m\over s}\right)^{-2}\right]. 
	\end{split}
	\end{align}

Let $z:=t+i/s, \Omega_{nm}:=\pi n/k+2m i /s$, $m,n\in\Z, \ s>0.$
We break the double sum in the series \eqref{seriesq} as a sum $\sum_m Q_m$, where each $Q_m$ is a series in  $n$:

$$
Q_m=\left\{\begin{array}{ll}
 \sum_{n\in\Z} \left[\left(z-\Omega_{nm}\right)^{-2}-\left(\Omega_{nm}\right)^{-2}\right], &   m\neq 0,\\
 &\\
\sum_{n\in\Z,n\neq 0}\left[\left(z-\Omega_{n0}\right)^{-2}-\left(\Omega_{n0}\right)^{-2}\right], &  m=0.
\end{array}\right.
$$
We have for $Q_m$ the exact expressions (see \cite{Ak}, page 197, Table I):
$$
Q_m=\left\{\begin{array}{ll}
k^2\Big[{\sin^{-2}\big(k(z-i{2m\over {s}})\big)}-{\sin^{-2}\left(i{2km\over s}\right)}\Big], &m\neq 0,\\
&\\
k^2\Big[-\frac{1}{3}+{\sin^{-2}\big(kz\big)}\Big], & m=0.
\end{array}\right.
$$
Substituting into these formulas $z=t+i/s,$ we get
$$
Q_m=\left\{\begin{array}{ll}
k^2\Big[{\sin^{-2}\big(k(t-i{2m-1\over s})\big)}-{\sin^{-2}\big(i{2km\over s}\big)}\Big], &m\neq 0,\\
&\\
k^2\left[-\frac{1}{3}+{\sin^{-2}\left(k(t+{i\over s})\right)}\right],& m=0.
\end{array}\right.
$$
Thus we have
$$
Q_m=\left\{\begin{array}{ll}
k^2
\left[
\left(\sin (k t)\cosh(k{2m-1\over s})-
	i\cos (k t) \sinh(k{2m-1\over s})\right)^{-2}\right. &\\
	&\\
	\left. \qquad -
{\sinh^{-2}({2km\over s })}
\right],
&m\neq 0,\\
&\\
k^2\left[-\frac{1}{3}+\left(\sin (k t)\cosh\left(\frac{k}{s}\right)+
	i\cos (k t) \sinh\left(\frac{k}{s}\right)\right)^{-2}\right],
& m=0.
\end{array}\right.
$$
Next introduce the change of parameter, $s\mapsto\tau=e^{-k/s}$, $0\leq\tau\leq\tau_0=e^{-k},$ ie $\e=\tau^2.$  In terms of $\tau$, we have 
\begin{align}\label{eq:Sm}
\begin{split}
Q_m&=\left\{\begin{array}{ll}
4 k^2\left[\left(\sin (k t)(\tau^{1-2m}+
 \tau^{2m-1}) \right.\right. \\
 &\\
 \qquad -\left. i\cos (k t) (\tau^{1-2m}- \tau^{2m-1})\right)^{-2} &\\
 &\\
\qquad \qquad \left. -(\tau^{-2m}-\tau^{2m})^{-2}\right], &  \tau>0, m\neq 0, \\
 &\\
0, &   \tau=0, m\neq 0,
\end{array}
\right.\\
&\\
 Q_0&=\left\{\begin{array}{ll}
k^2\left[-\frac{1}{3}+{4}\left(\sin (k t)(\tau^{-1}+\tau)\right.\right.\\
\qquad\qquad \left.\left. +
	i\cos (k t) \left(\tau^{-1}-\tau\right)\right)^{-2}\right],&  \ \tau>0\\
&\\
-\frac{1}{3}k^2,&\ \tau=0.\
\end{array}
\right.
\end{split}
\end{align}
From formulas \eqref{eq:Sm} one can conclude 
the following facts:
\begin{enumerate}[(1)]
	\item \label{item:q1}The series $$q=\sum_{m\in\mathbb{Z}} Q_m$$ 
	converges as $\tau\to0$, uniformly in $(\tau,t)\in[0,\tau_0]\times\R$, to the constant function $q_0=-\frac{1}{3}k^2.$ This follows from the estimate
	$$
	|Q_m|\leq {C(\tau_0)}(\tau_0^{4|m|}+\tau_0^{|4m-2|})
	$$
	for some constant $C(\tau_0)>0$. Since $\tau_0=e^{-k}<1$  this implies uniform convergence in $[0,\tau_0]\times\R$.  
	\item \label{item:q2}Every term $Q_m$ in equation \eqref{eq:Sm} is analytic in $\tau$ at $\tau=0$ with 
	radius of convergence $R_m=1>\tau_0.$
	To see this, one represents each term in the square brackets of 
	\eqref{eq:Sm} as a rational function of $\tau$ and finds that its poles  all lie on the unit circle in the complex $\tau$ plane. Hence $R_m=1$. 
	\item It follows from items  \ref{item:q1} and   \ref{item:q2}, by Weierstrass theorem, that the sum of the series $\sum Q_m$, which equals exactly $q_{\tau}(t)$, is analytic in $\tau\in[0,\tau_0]$.
	\item Each $Q_m$ in equation \eqref{eq:Sm} is clearly even in $\tau$, hence so is $q.$ Thus, with the change of variable $\e=\tau^2$,  $q_\e$ becomes analytic in $\e$. 
	
	\item The following 1st order Taylor expansions at $\tau=0$ hold: 

	\be
 Q_0=-\frac{k^2}{3}-4 k^2 e^{2 i k t}\tau^2+\ldots,\quad Q_1=-4  k^2 e^{-2 i k t}\tau^2+\ldots,
\ee
and  $Q_m$ is of order $\tau^{4m-2}$ for $m>0,$ which implies equation \eqref{eq:qeps}. 
\qed
\end{enumerate}

\subsubsection {Proof of Lemma \ref{lemma:416}}
	Notice that, for a given periodic potential $q(t)$, 
	the problem \eqref{eq:F} of Floquet eigenvalues has the following properties
	(see \cite {East}, page 32):
	\begin{enumerate}[(1)]
		\item The eigenvalues $\lambda_m(\mu)$ are  solutions of the equation 
		\begin{equation}\label{eq:trace}
			\Delta(\lambda)=2\cos\left(\frac{\pi}{k}\right). 
		\end{equation}
		Here and below,   $\Delta(\lambda)=\tr M(\lambda)$ is the trace of the monodromy matrix of equation  \eqref{eq:F}.
		It is defined as follows. Fix a basis of solutions $
		\{y_1(\lambda,t), y_2(\lambda,t)\}$ of the second order differential equation
		$$
		X''+ (\lambda-2q(t))X=0, 
		$$ 
			such that
			$$ y_1(\lambda,0)=y_2'(\lambda,0)=1, \ y_1'(\lambda ,0)=y_2(\lambda,0)=0.
			$$
		Then the monodromy matrix is 
		$$
		M(t,\lambda)=
		\begin{pmatrix}
			m_{11}&m_{12}\\
			m_{21}&m_{22}
		\end{pmatrix}=
		\begin{pmatrix}
			y_1(2\omega,\lambda)&y_2(2\omega,\lambda)\\
				y_1'(2\omega,\lambda)&y_2'(2\omega,\lambda) 
		\end{pmatrix},\quad
	\det(M)=1.
		$$ 
		\item The graph of the function $\Delta (\lambda)$ (see Figure \ref{graph}) is such that all the solutions of equation  \eqref{eq:trace}
		are transversal. Hence all eigenvalues $\lambda_{m,s}(\mu)$ of equation \eqref{eq:F}, and, in particular, $\lambda_{0,s}(\mu)$, depend smoothly on the parameter $s$.
		\item All Floquet eigenvalues $\lambda_{0,s}(\mu)$ of equation \eqref{eq:F} have multiplicity $1$, because if $X$ is an eigenfunction
		for some non-real Floquet exponent $\mu$, then $\overline X$ is not.
	\end{enumerate}
	In our situation, we have a smooth  family of potentials $q_s(t)$. So we have the standard basis $\{
	y_1(\lambda,s,t), y_2(\lambda,s,t)\},$ where 
	$$y_1(\lambda,s,0)=y_2'(\lambda,s,0)=1, \ y_1'(\lambda,s ,0)=y_2(\lambda,s,0)=0,
	$$ and the monodromy matrix $M(\lambda,s)$ which is smooth  in
	$\lambda,s$.
	We can write the eigenfunction corresponding to $\lambda_{m,s}(\mu)$ in the form
	$$X= A y_1+B y_2,$$ for some complex $A,B$.
	Then the Floquet boundary conditions in terms of $A,B$ reads
	\begin{equation}\label{eq:M}
	(M(\lambda,s)-\mu Id)\cdot\begin{pmatrix}
			A\\
			B
		\end{pmatrix} 
		=0.
	\end{equation} 
	
	Moreover,
	it follows from properties (2) and (3) above that,  for $\lambda=\lambda_{m,s}$, the matrix $(M-\mu Id)$ has rank  1 and that  $M(\lambda_{m,s},s)$  depends smoothly on $s$.
	The normalization $X(0)=1$ implies that $A=1$ and hence $B$ can be found  uniquely from \eqref{eq:M}, 
	$$
	B=-(m_{11}-\mu)/m_{12}.
	$$
	It is important that the denominator  $m_{12}$ in this formula cannot vanish, because otherwise the matrix $M$ would be triangular
	having real eigenvalues, which is not the case, since $\mu$ is not real.
	 Thus we conclude that the solution $\begin{pmatrix}
		A\\
		B
	\end{pmatrix} $ of equation  \eqref{eq:M} is smooth  in  $s\in[0,1]$. An analogous proof applies when the potential $q_\e$ depends for analyticaly on  $\e\in[0,e^{-2k}].$ This completes the proof of our Lemma. 
\qed


\subsubsection{Proof of Lemma \ref{lemma:wegdef}}  

The functions $\wp_s,\sigma_s, \zeta_s$ depend analytically on $s\in(0,1]$ and can be shown to converge, as $s\to 0$,  to the limiting functions   (see \cite{Ak}, page 201) 
\begin{equation} \label{eq:formulas}
\begin{split}
&\wp_0(z)=-\frac{k^2}{3}
+k^2{\sin^{-2}(kz)}, \quad
\ \zeta_0(z)=\frac{k^2}{3}z+k\cot (k z),\\
&\sigma_0(z)=\frac{1}{k}e^{k^2 z^2/6}\sin(k z).
\end{split}
\end{equation}
Using the above formula  for $\zeta_0$, we compute that equation  \eqref{eq:a} for $s=0$ is equivalent to 
\begin{equation}\label{eq:tanh}
a={\pi\over 2k}+ib, \quad \tanh \left(\frac{\pi b}{2}\right)=\frac{1}{k}.
\end{equation}

Consider equation \eqref{sigmas} on $\alpha$ for $s=0$ :
	$$
	\frac{\sigma_0(a+\alpha)}{\sigma_0(a-\alpha)}=e^{2\alpha\zeta_0(a)},
	$$where $a$ is the solution of equation  \eqref{eq:a} for $s=0.$
	Set 
	$$
	F(\alpha):=\frac{\sigma_0(a+\alpha)}{\sigma_0(a-\alpha)}e^{-2\alpha\zeta_0(a)}.
	$$
	Using the explicit formulas \eqref{eq:formulas}-\eqref{eq:tanh},
	we have:
	$$
	F(\alpha)=\frac{\sin \left(k(a+\alpha)\right)}{\sin \left(k(a-\alpha)\right)}e^{i2\alpha}=
	\frac{1-i\frac{1}{k}\tan (k\alpha) }{1+i\frac{1}{k}\tan (k\alpha)}e^{i2\alpha}.
	$$
	This immediately implies that the equation $F=1$ is equivalent to the familiar  equation \eqref{eq:ubiq}:
	$$
	k\tan(\alpha)=\tan(k \alpha).
	$$This means that, for $s=0$, equation \eqref{sigmas}  has precisely 
	$k-2$ solutions for $\alpha\in (0,\pi)$.
	
	Moreover, differentiating $F$ at a point $\alpha$ where $F(\alpha)=1$
	we have:
	\begin{equation*}
		\begin{split}
			F'(\alpha)&
			=2i\frac{(1-k^2)\tan^2 k\alpha}{k^2+\tan^2 k\alpha}\neq 0.
		\end{split}
	\end{equation*}
	Applying the implicit function theorem, we conclude that all $k-2$ solutions of equation  \eqref{sigmas} can be smoothly extended from $s=0$ to $s>0$. This, together with
	Theorem \ref{th:selfbacklund} and Corollary \ref{transversal}, implies the existence of $k-2$ solutions for every $s\in[0,1]$, smoothly depending on $s$. An analogous proof applies for analytic dependence on $\e\in[0,e^{-2k}].$
\qed

\subsubsection{Proof of Lemma \ref{lemma:419}}
We calculate  mod  $\e^2$. Use the Taylor expansion \eqref{eq:qeps}, 
$$ q_\e=-\frac{k^2}{3}-8 k^2\cos(2kt)\e+\ldots, $$
and let $ X_\e=e^{it}+X_1\e+\ldots,  \ \lambda_{0,\e}=\lambda_0+\lambda_1\e+\ldots$. Substitute these  into $X''+(\lambda-2q)X=0$ and solve  for successive powers of $\e$. The $\e^0$ term gives
$$\lambda_0=1-2k^2/3
$$
and 
the $\e^1$ term gives
$$
X_1''+X_1+ 8 k^2\left( e^{i (1+2 k) t}+ e^{i (1-2 k) t}\right)+\lambda_1 e^{i t}=0.
$$
The general solution is 
$$ X_1=A_+e^{i(1+2k)t}+A_-e^{i(1-2k)t}+B_+e^{it} +B_-e^{-it}  +{\lambda_1\over 2i}te^{it},
$$
where $A_\pm=2k/(k\pm 1)\neq 0$ and  $B_\pm\in\C$ are arbitrary. Since $X_1$ is periodic we must have 
$\lambda_1=0$
and what  remains is non-vanishing. \qed
	
\section{Self-B\"acklund polygons} \label{sect:selfBP}

\subsection{Centroaffine butterflies, Bianchi permutability} \label{subsect:permut}


The central projection $\R^2 \setminus\{0\} \to \RP^1$ takes a centroaffine curve to a curve in the projective line. Conversely, a projective curve admits a unique lift to a centroaffine curve. 
Bianchi permutability for $c$-relation was established for projective curves, in \cite{Tab18}. Here we do it for centroaffine curves. 

Let us say that a quadrilateral    $P_1P_2P_3P_4$ forms a {\it centroaffine butterfly} if 
\be\label{eq:cab}
[P_1,P_2]=[P_4,P_3] \ \mbox{ and }\ [P_2,P_3]=[P_1,P_4].
\ee
Note that a centroaffine butterfly is not necessarily a centroaffine polygon.

\begin{lemma} \label{lm:constr}
A generic quadrilateral  $P_1P_2P_3P_4$ is a {\it centroaffine butterfly} if and only if  any of the following equivalent conditions are satisfied:
\begin{enumerate}
\item There is a linear involution $\RR\in \GLt$ interchanging $P_1P_2$ and $P_3P_4$. 
That is, $\RR( P_1)=P_3$, $\RR (P_2)=P_4$, $\RR (P_3)=P_1$,  $\RR (P_4)=P_2$. 
\item The line segments $P_1P_3,P_2P_4$ are parallel and their midpoints are collinear. See Figure \ref{dbf}.
\item $P_aP_bP_cP_d$ is a centroaffine butterfly, where $abcd$ is any of the 8 permutations of $1234 $ generated by $ (1234), (24) , (12)(34). $
\end{enumerate}
\end{lemma}


\begin{figure}[ht]
\centering
\includegraphics[width=.45\textwidth]{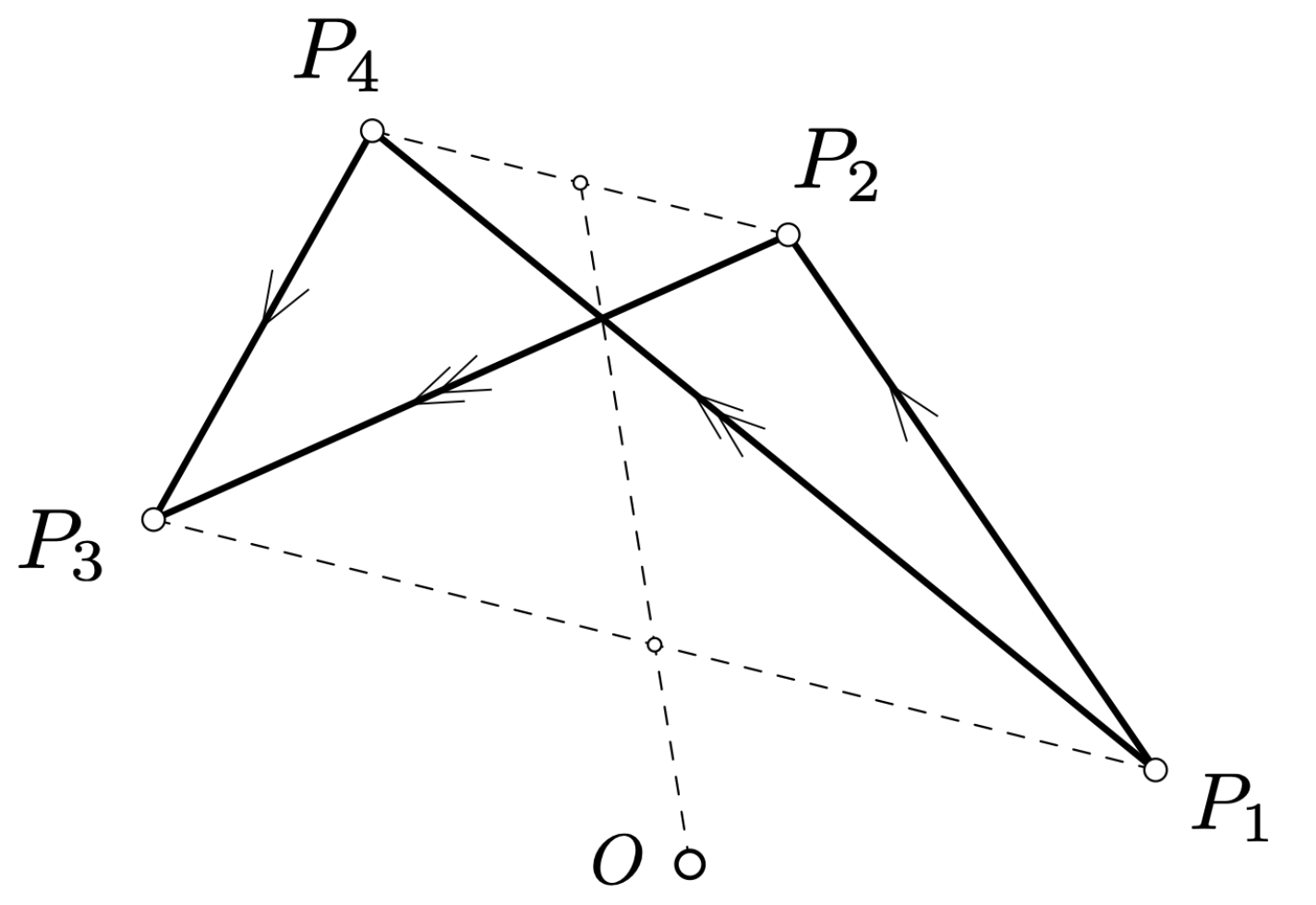}
\caption{A centroaffine butterfly}
\label{dbf}
\end{figure}

\begin{proof} 1. By applying a linear transformation, we can assume the $P_1=(1,0), \ P_3=(0,1).$ 
Let $P_3=(c,d)$. Then equations \ref{eq:cab} imply $P_4=(d,c).$ Thus  $\RR:(x,y)\mapsto (y,x)$ is the required symmetry. 

\mn 2. Note that the said  segments  are parallel and their midpoints  are collinear if and only if $[P_1\pm P_3,P_2\pm P_4]=0$ (`$-$' for the 1st statement, `+' for the 2nd). By expanding these expressions  we see that they are equivalent to $[P_1,P_2]=[P_4,P_3], \ [P_2,P_3]=[P_1,P_4]$. 

\mn 3. This is  a simple  verification (omitted). 
\end{proof}

It follows from this lemma that,  given a generic triple of points $P_1, P_2,P_3$, there is a unique fourth point $P_4$ such that $P_1P_2P_3P_4$ form a  centroaffine butterfly. Namely, by property 1, $P_4=\RR P_2$ where $\RR$ is defined by $\RR P_1=P_3$, $\RR P_3=P_1$. More geometrically, by property 2, one constructs the line $\ell$ through $P_2$ and parallel to $P_1P_3$, intersect $\ell$ with the line through the origin $O$ and  the midpoint of $P_1P_3$, then finds the unique point  $P_4$ on $\ell$ such that this intersection point is the midpoint of $P_2P_4$.

\begin{theorem}[Bianchi permutability] \label{thm:Bianchi}
Consider three centroaffine curves $\g, \delta$, and $\G$ such that $\G$ and $\delta$ are $b$-  and $c$-related to $\g$ (respectively).  Then there exists a fourth centroaffine curve $\Delta$ that is $b$-related to $\delta$ and $c$-related to $\G$. In fact, $\Delta(t)$ is the unique point  such that $\delta(t)\g(t)\G(t)\Delta(t)$ form a centroaffine butterfly.

\end{theorem} 

\begin{proof}
The idea of the proof that if $\g(t), \delta(t)$ and $\G(t)$ are considered as three vertices of time-evolving centroaffine butterfly, then $\Delta(t)$ is its forth vertex. 

Specifically, we have 
$$
[\g,\delta]=[\G,\Delta]=c,\ [\g,\G]=[\delta,\Delta]=b,
$$
and  need to check that $\Delta(t)$ is a centroaffine curve, that is, $[\Delta,\Delta']=1$. 

Using the above relations, 
one can write $\Delta$ as a linear combination of $\delta$ and $\G$, 
$$
\Delta = \frac{[\g,\delta]}{[\G,\delta]} \delta - \frac{[\g,\G]}{[\G,\delta]}\G = \frac{c\delta-b\G}{[\G,\delta]} .
$$
Then
$$
[\Delta,\Delta']=\frac{[c\delta-b\G,c\delta'-b\G']}{[\G,\delta]^2}=
\frac{b^2+c^2-bc([\delta,\G']+[\G,\delta'])}{[\G,\delta]^2}.
$$
Thus we want to show that 
\begin{equation} \label{eq:lrhs}
b^2+c^2-bc([\delta,\G']+[\G,\delta'])=[\G,\delta]^2.
\end{equation} 

We have
$$
\delta=f\g+c\g',\ \G=g\g+b\g',
$$
hence
$$
\delta'=(f'+cp)\g+f\g',\ \G'=(g'+bp)\g+g\g'.
$$
It follows that 
$$
[\G,\delta]=cg-bf, \ [\delta,\G']=fg-cg'-bcp, \ [\G,\delta'] =fg - bf' - bcp.
$$
In addition, one has equations \eqref{eq:fp}:
$$
cf'=f^2-c^2p-1,\ bg'=g^2-b^2p-1.
$$
Substitute these formulas into equation  \eqref{eq:lrhs} to obtain a true identity.
\end{proof}

\subsection{Rigidity results and flexible examples of self-B\"acklund polygons} \label{subsect:rigid}

B\"acklund transformation can be defined on centroaffine polygons. Similarly to its continuous version, it is a completely integrable dynamical system. We refer to \cite{AFT} for a detailed study; see also \cite{Mat}.

For the purpose of this paper, we recall, from Introduction, 
that an origin-symmetric $2n$-gon ${\bf P}$ in $\R^2$ with vertices $P_i, i=1,\ldots,2n$, is called a {\it self-B\"acklund $(n,k)$-gon} if
$$
[P_i,P_{i+1}]=1,\ [P_i,P_{i+k}]=c
$$ 
for all $i$ and $2\le k \le n-2$. Such polygons are acted upon by $\SLt$.  Since $P_{i+n}=-P_i$, we can assume, without loss of generality, that $k \le n/2$. 

A regular $2n$-gon is a self-B\"acklund $(n,k)$-gon for all $2\le k\le n/2$. We call these self-B\"acklund $(n,k)$-gons and their $\SLt$ images trivial. The problem is to find non-trivial self-B\"acklund $(n,k)$-gons.

The next result is analogous to Theorem 9 of \cite{Tab06}.

\begin{theorem} \label{triv}
In the following cases every self-B\"acklund $(n,k)$-gon is trivial:
\begin{enumerate}
\item  $n$ is arbitrary, $k=2$;
\item   $n$ is odd, $k=3$;
\item   $k$ is arbitrary, $n=2k+1$.
\item   $n=3k$.
\end{enumerate}
On the other hand, there exist non-trivial self-B\"acklund $(n,k)$-gons in the following cases:
\begin{enumerate}
\item $n$ is even and $k$ is odd;
\item  $n=2k$.
\end{enumerate}
\end{theorem}

\begin{proof} 
Each next vertex is a linear combination of the preceding two:
$
P_{i+2} =  a_i P_{i+1}-P_i.
$

Let $k=2$. Then $[P_i,P_{i+2}]=c$, hence $a_i=c$ for all $i$. 
Let $A$ be the linear map defined by
$$
A(P_1)=P_2, \ A(P_2)=P_3.
$$
We claim that $A$ is area preserving and $A(P_i)=P_{i+1}$ for all $i$. This would imply that the polygon ${\bf P}$ is centroaffine regular, that is, trivial.

That $A$ is area preserving follows from $[P_1,P_2]=[P_2,P_3]$. Next,
$$
P_3=-P_1+cP_2,\ \ {\rm hence}\ \ A(P_3)=-P_2+cP_3=P_4.
$$
Repeating this argument, we obtain $A(P_i)=P_{i+1}$ for all $i$.
\medskip 

Now let $n$ be odd and $k=3$. Consider four consecutive vertices of ${\bf P}$; they satisfy the Ptolemy-Pl\"ucker relation
$$
[P_i,P_{i+1}] [P_{i+2},P_{i+3}] + [P_{i+1},P_{i+2}] [P_{i},P_{i+3}] = [P_{i},P_{i+2}] [P_{i+1},P_{i+3}].
$$
Therefore 
$$
1 + c = [P_{i},P_{i+2}] [P_{i+1},P_{i+3}].
$$
It follows that $[P_{i},P_{i+2}] = [P_{i+2},P_{i+4}]$ for all $i$. 

Recall that $n$ is odd and that $P_{i+n}=-P_i$ for all $i$. This implies that 
$$
[P_i,P_{i+2}] = [P_{n+i},P_{n+i+2}] = [P_{i+1},P_{i+3}], 
$$
and hence $[P_i,P_{i+2}]$ has the same value for all $i$. Thus ${\bf P}$ is a self-B\"acklund $(n,2)$-gon, the  already considered case. 
\medskip

Next, let $n=2k+1$. First we notice that $[P_i,P_{i+k+1}]=c$. Indeed, 
$$
[P_i,P_{i+k}]=[P_{i+k+1},P_{i+n}]=[P_i,P_{i+k+1}].
$$
Now consider the quadruple of vertices $P_i,P_{i+1},P_{i+k}, P_{i+k+1}$. The Ptolemy-Pl\"ucker relation implies that
$$
[P_{i+1},P_{i+k}]= \frac{c^2-1}{c}
$$
for all $i$. That is, $[P_i,P_{i+k-1}]$ is independent of $i$.

Continuing in the same way, we reduce $k$ until we get to the case $k=2$, considered above, and we conclude that ${\bf P}$ is centroaffine regular.

Now let $n=3k$. Let us scale the polygon so that $[P_i, P_{i+k}]=\sqrt{3}/2$ for all $i$ (as for a regular $6k$-gon inscribed in a unit circle).  Then $[P_i,P_{i+1}]=t$, a constant. 

Each hexagon ${\bf P}_i:=(P_i,P_{i+k},P_{i+2k}, P_{i+3k},P_{i+4k},P_{i+5k})$ is affine-regular, and they are all equivalent under $\SLt$. Hence we assume, without loss of generality, that the vertices of ${\bf P}_0$ are the sixth roots of unity. Let $A \in  \SLt$ take ${\bf P}_0$ to ${\bf P}_1$. A quick calculation, using the  equations
$$
[P_0,P_1]=[P_{k},P_{k+1}]=[P_{2k},P_{2k+1}]=[P_{3k},P_{3k+1}]=[P_{4k},P_{4k+1}]=[P_{5k},P_{5k+1}]=t,
$$
reveals that $A$ is a rotation
$$
A=
\begin{bmatrix}
\cos\alpha&-\sin\alpha\\
\sin\alpha&\cos\alpha
\end{bmatrix},
\ t=\sin\alpha.
$$

The same argument, applied to the linear map that takes ${\bf P}_1$ to ${\bf P}_2$, shows that this map is the same rotation, $A$. And so on, showing that the polygon is regular.
\medskip

Let us construct non-trivial self-B\"acklund $(n,k)$-gons for even $n$ and odd $k$. Start with a regular $2n$-gon, and consider the midpoints of its sides. These points are the vertices of another regular $2n$-gon. Dilate the latter $2n$-gon with the center of dilation at its center. We obtain a centrally symmetric $4n$-gon having a dihedral symmetry, and this symmetry implies $[P_i,P_{i+k}] = [P_{i+1},P_{i+k+1}]$. See Figure \ref{polys} on the left. (The projection of this polygon to $\RP^1$ is a regular $n$-gon therein).

\begin{figure}[ht]
\centering
\includegraphics[width=.4\textwidth]{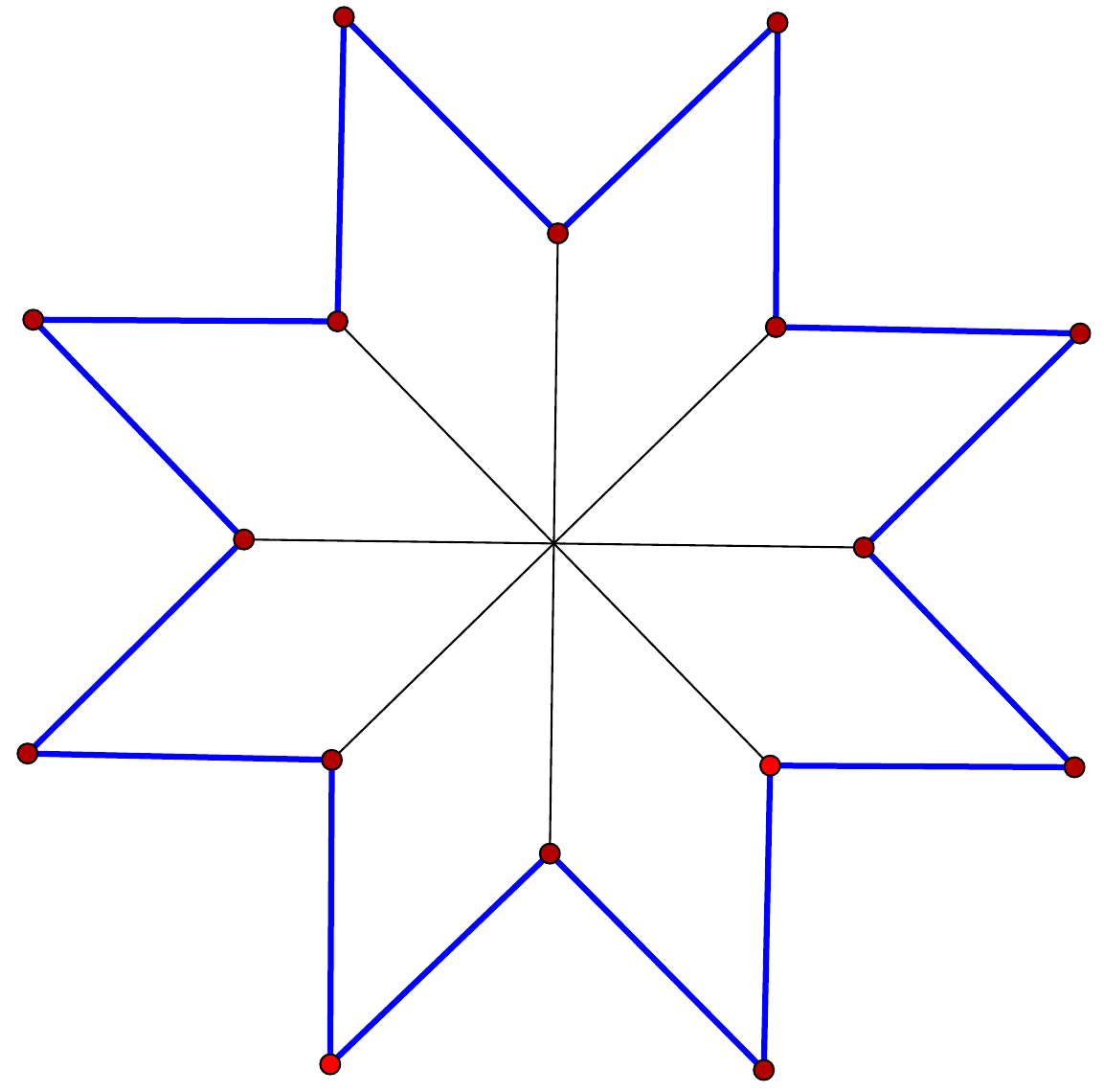}\qquad 
\includegraphics[width=.4\textwidth]{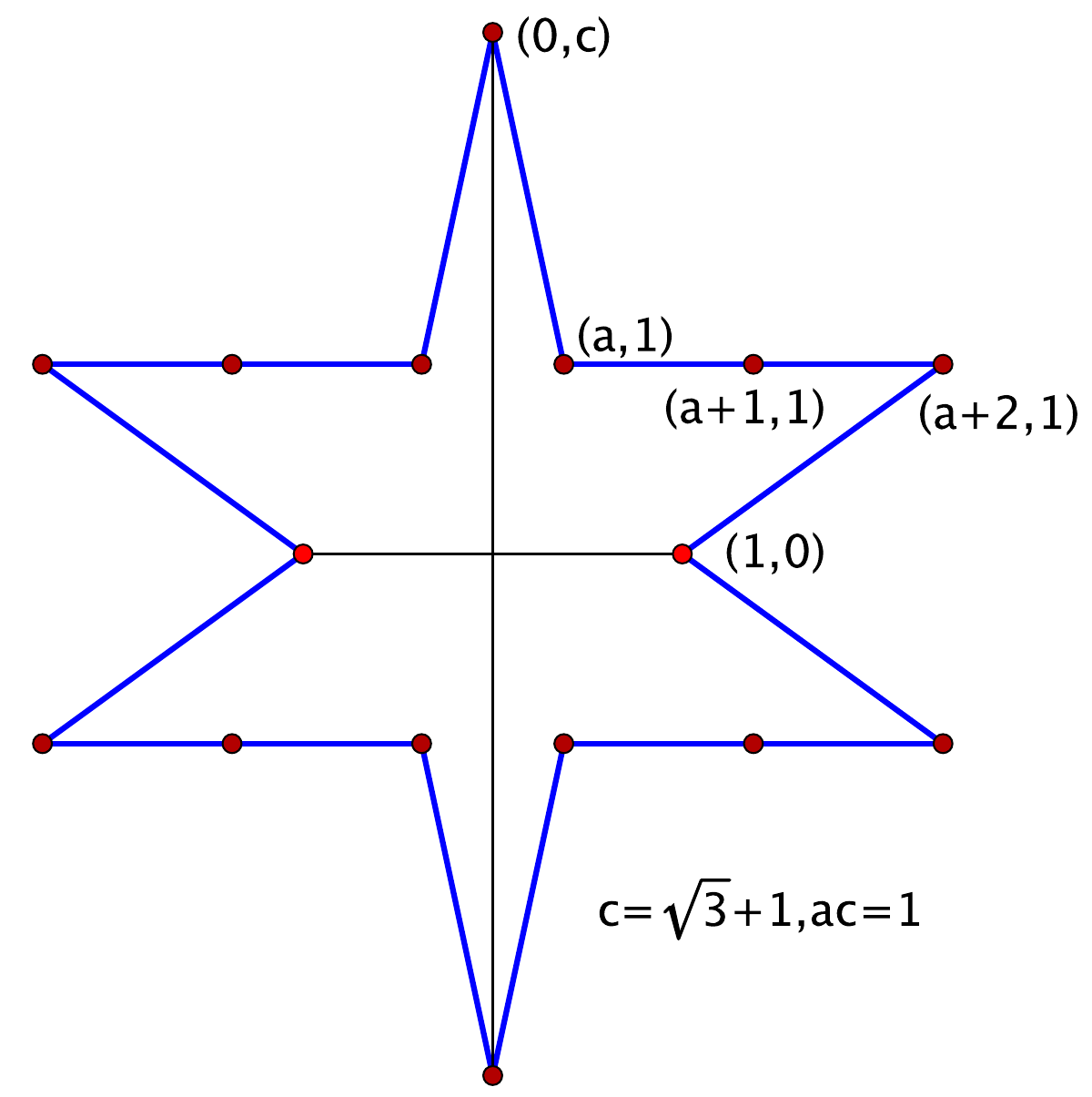}
\caption{Left: a self-B\"acklund $(8,3)$-gon. Right: a self-B\"acklund $(8,4)$-gon.}
\label{polys}
\end{figure}

The construction of a non-trivial self-B\"acklund $(2k+4,k+2)$-gon is presented in Figure \ref{polys} on the right (where $k=2$).\footnote{We are grateful to Michael Cuntz for suggesting this construction.} This polygon has two axes of symmetry. In the general case, one has points $(a,1),(a+1,1),\ldots,(a+k,1)$ on a horizontal line with
$$
a=\frac{\sqrt{k^2+8}-k}{4},\ c= \frac{\sqrt{k^2+8}+k}{2}.
$$
One checks that $[P_i,P_{i+1}]=1$ and $[P_i,P_{i+k+2}]=c$ for all $i$.
\end{proof}

\subsection{Infinitesimal deformations of regular polygons} \label{subsect:polyinf}

Here we consider the linearized problem, that is, infinitesimal deformations of regular polygons as self-B\"acklund $(n,k)$-gons;
this is a discrete analog of the material in Section \ref{subsect:infinitesimal}.

Call a regular polygon {\it infinitesimally rigid} as a self-B\"acklund $(n,k)$-gon if each of  its infinitesimal deformations in the class of  
self-B\"acklund $(n,k)$-gons is induced by the action of $\mathfrak{sl}(2,\R)$.

\begin{theorem} \label{thm:infinites}
A regular $2n$-gon is infinitesimally rigid as a self-B\"acklund $(n,k)$-gon unless one of the following holds:
\begin{enumerate}
\item $n$ is even and $k$ is odd;
\item $n=2k$ with even $k>2$;
\item there exists an integer $j$ with $2\le j\le n-2$ such that 
$n=2(k+j)$ and $n$ divides $(k-1)(j-1)$.
\end{enumerate}
\end{theorem}

\begin{corollary} \label{cor:rigid}
A regular $2n$-gon is infinitesimally rigid as a self-B\"acklund $(n,k)$-gon  if $n$ is odd,
or if  both $n$ and $k$ are even, $k< n/2$, and gcd $(n,k)>2$. 
\end{corollary}

\begin{proof}
The first statement of the corollary follows immediately from the theorem. 

For the second statement, assume that a non-trivial infinitesimal deformation exists. We claim that $k$ and $j$ are coprime. Indeed, if $(j,k)=p$, then $n=2(j+k) \equiv 0$ mod $p$, but $(j-1)(k-1)\equiv 1$ mod $p$. This contradicts to the fact that $n$ divides $(j-1)(k-1)$. 
It follows that 
$$
(n,k) = (2(j+k),k) = 2(j,k)=2,
$$
 proving the second statement.
\end{proof}

Now we prove Theorem \ref{thm:infinites}.

\begin{proof}  Let
$$
P_j = \left( \cos\left( \frac{\pi j}{n} \right), \sin\left( \frac{\pi j}{n} \right)\right),\ j=1,\ldots,2n,
$$
be the vertices of a regular $2n$-gon. We have 
$$
[P_j,P_{j+1}]=\sin\left(\frac{\pi}{n}\right)=a,\ [P_j,P_{j+k}] = \sin\left(\frac{\pi k}{n}\right) =b.
$$
(One can rescale to have $a=1$, but it is not really needed for the argument.)

We also have  the respective second-order linear recurrence
\begin{equation} \label{rec}
P_{j+1}=2\cos\left( \frac{\pi}{n} \right) P_j - P_{j-1}.
\end{equation}

Consider an infinitesimal deformation $P_j + \eps V_j$, where $V_j$ is an $n$-anti-periodic sequence of vectors, 
that is, $V_{j+n}=-V_j$ for all $j$,  and assume that the resulting polygon is a self-B\"acklund $(n,k)$-gon. By applying a dilation, we may assume that the constant $a$ does not change. Then, calculating modulo $\eps^2$, we obtain two systems of equations
\begin{equation} \label{norm}
[P_j,V_{j+1}] + [V_j,P_{j+1}]=0,\ j=1,\ldots,n,
\end{equation} 
and
\begin{equation} \label{lox}
[P_j,V_{j+k}] + [V_j, P_{j+k}] = C,\ j=1,\ldots,n,
\end{equation}
where  $C$ is a constant.

Consider  the system \eqref{norm}. Let
$$
V_j=a_jP_j+b_jP_{j+1}=c_jP_j+d_jP_{j-1}.
$$
Then  the recurrence \eqref{rec} implies that
$$
\frac{c_j-a_j}{b_j} = 2\cos\left( \frac{\pi}{n} \right), \ \frac{d_j}{b_j} =-1.
$$
Substitute vectors  $V_j$ into equation \eqref{norm} to obtain
\begin{equation} \label{abd}
a_j=-c_{j+1},\ b_j = \frac{c_j+c_{j+1}}{2\cos(\pi/n)},\ d_j = - \frac{c_j+c_{j+1}}{2\cos(\pi/n)},
\end{equation}
where $c_j$ is an $n$-periodic sequence to be determined.

Now consider  the system \eqref{lox}. Substituting vectors $V_j$, using equation  \eqref{abd}, and collecting terms yields the linear system
\begin{equation} \label{linc}
\mu_{k-1} c_j - \mu_{k+1} c_{j+1} + \mu_{k+1} c_{j+k} - \mu_{k-1} c_{j+k+1} = C,\ j=1,\ldots,n,
\end{equation}
where $\mu_k = \sin(\pi k/n)$.

First, we note that $C$ must be zero. Indeed, add  equations \eqref{linc}: the left hand side vanishes, and so must the right hand side.

Second,  system \eqref{linc} has a 3-dimensional space of trivial solutions that correspond to the action of the Lie algebra $\slt$. These solutions are given by the formulas
$$
c_j = 1;\ c_j = \cos\left( \frac{\pi (2j-1)}{n} \right);\ c_j = \sin\left( \frac{\pi (2j-1)}{n} \right).
$$
We need to find out when there are no other solutions.

To this end, consider the eigenvalues of the matrix defining  the system \eqref{linc}. This is a circulant matrix, and its eigenvalues are given by the formula
$$
\lambda_j = \mu_{k-1}  - \mu_{k+1} \omega_j + \mu_{k+1} \omega_j^k - \mu_{k-1} \omega_j^{k+1},\ j=0,\ldots,n-1,
$$
where 
$\omega_j = e^{i\frac{2\pi j}{n}}$, 
see  \cite{Dav}. 

We are interested in zero eigenvalues. One has $\lambda_j=0$ if and only if
$$
\omega_j^{k+1} = \frac{\mu_{k-1} - \mu_{k+1} \omega_j}{\mu_{k-1} - \mu_{k+1} \overline \omega_j}.
$$
Let $2\alpha$ be the argument of the unit complex number on the right. A direct calculation yields
$$
\tan \alpha = - \frac{\sin\left( \frac{\pi (k+1)}{n} \right) \sin\left( \frac{2\pi j}{n} \right)}{\sin \left( \frac{\pi (k-1)}{n}\right) - \sin\left( \frac{\pi (k+1)}{n} \right) \cos\left( \frac{2\pi j}{n} \right)}.
$$
The argument of $\omega_j^{k+1}$ is $2\pi j(k+1)/n$, hence (after cleaning up the formulas)
$$
\sin\left( \frac{\pi j(k+1)}{n} \right) \sin\left( \frac{\pi (k-1)}{n} \right) = \sin\left( \frac{\pi j(k-1)}{n} \right) \sin\left( \frac{\pi (k+1)}{n} \right),
$$
or, equivalently,
\begin{equation}  \label{tangs}
\tan \left( \frac{\pi j}{n} \right) \tan \left( \frac{\pi k}{n} \right) = \tan \left( \frac{\pi jk}{n} \right) \tan \left( \frac{\pi}{n} \right).
\end{equation}
Note the trivial solutions $j=0,1,n-1$, corresponding to the action of $\slt$. Let us assume that $2\le j \le n-2$.

One also has other trivial solutions, when both sides of equation  \eqref{tangs} are infinite: $n=2j$ and $k$ odd, and $n=2k$ and $j$ odd.  Note that, in the latter case, $k>2$. Indeed, if $k=2$, then $n=4$, and since $2\le j \le n-2$, we have $j=2$, contradicting that $j$ is odd. 

Equation \eqref{tangs} appeared in \cite{Tab06}  and in \cite{ASTW}, and it was solved in \cite{CC}. This equation has non-trivial solutions if and only if $n=2(j+k)$ and $n$ divides $(j-1)(k-1)$. This completes the proof.
\end{proof}

\begin{remark}
{\rm As we know from Theorem \ref{triv}, if $n$ is even and $k$ is odd, or if $n=2k$, non-trivial self-B\"acklund $(n,k)$-gons indeed exist.  The smallest values in case 3) of Theorem \ref{thm:infinites} are $k=4, n=30$. Does there exist a non-trivial self-B\"acklund $(30,4)$-gon?
}
\end{remark}

\begin{remark}
{\rm 
One wonders whether the symmetry between $k$ and $j$ in the formulation of Theorem \ref{thm:infinites} corresponds to some kind of duality between self-B\"acklund $(n,k)$- and $(n,j)$-gons.
}
\end{remark}


\section{Appendix A: From the centroaffine plane to the hyperbolic plane} \label{sect:hyp}

In this appendix we connect two geometries associated with the group $\SLt$, the centroaffine and the hyperbolic ones.

Consider the 3-dimensional space of quadratic forms $ax^2+2bxy+cy^2$ with the pseudo-Euclidean metric given by quadratic form $b^2-ac$,  the negative of the determinant of the quadratic form.  The projectivization of the subspace of the positive-definite forms is the hyperbolic plane $H^2$; the degenerate forms comprise the circle at infinity. In the modern literature, this approach to hyperbolic geometry was developed in \cite{Ar2}. 

In the coordinates $(u,v,w)$, such that
$$
a=u+v,\ b=w,\ c=u-v,
$$
one has the standard Minkowski metric $v^2+w^2-u^2$. The unit-determinant quadratic forms comprise the hyperboloid of two sheets, and the condition $a+c>0$ describes its upper half, the pseudo-sphere.

A ``unit" central ellipse of area $\pi$ is an $\SLt$ image of the unit circle,  given by an equation of the form $ax^2+2bxy+cy^2=1$ with $ac-b^2=1$ and $a+c>0$. This defines a point of the hyperbolic plane $H^2$ in the pseudo-sphere model. 

Likewise, a central hyperbola, which is an $\SLt$ image of the ``unit" hyperbola $xy=1$, is given by an equation of the form $ax^2+2bxy+cy^2=1$ with $ac-b^2=-1$. It defines a point of  the hyperboloid of one sheet.

\begin{lemma} \label{lm:duality}
Let a unit central ellipse $ax^2+2bxy+cy^2=1$ and a unit central hyperbola $a'x^2+2b'xy+c'y^2=1$ be tangent at point $(x,y)$. Then the vectors $(a,b,c)$ and $(a',b',c')$ are orthogonal.
\end{lemma}

\begin{proof}
The group $\SLt$ acts transitively on the space of contact elements of the punctured plane whose line does not pass through the origin. And it acts by isometries on the space of quadratic forms.
Therefore it suffices to consider the point $(1,0)$ and the vertical direction. In this case the two conics are $x^2+y^2=1$ and $x^2-y^2=1$, and the vectors $(1,0,1)$ and $(1,0,-1)$ are indeed orthogonal.
\end{proof}

To a point $(x,y)$ of the punctured plane there corresponds the affine plane $ax^2+2bxy+cy^2=1$ in the 3-dimensional space of quadratic forms. The normal vector of this plane is isotropic, and this plane lies above the origin. Hence 
its intersection with the pseudo-sphere is a horocycle in $H^2$. The symmetric point $(-x,-y)$ yields the same horocycle.

To summarize, a point of the centroaffine plane is a horocycle in $H^2$, and a unit central ellipse is a point of $H^2$. 

Let $\g(t)$ be a centoraffine curve. The {\em osculating ellipse} at a point $(x,y)= \g(t)$ is a unit central ellipse tangent to $\g$ at this point. As $t$ varies, one obtains a curve $\g^*(t) \subset H^2$, the dual curve of $\g$. Due to the central symmetry of $\g$, this curve closes up after $t$ is increased by $\pi$.
Equivalently, the curve $\g^*$ is the envelope of the horocycles corresponding to the points of the curve $\g$.

\begin{lemma} \label{lm:param}
If $[\g(t),\g'(t)]=1$, then $|\g^*(t)'|=|1+p(t)|$. 
\end{lemma}

\begin{proof}
Let $\g(t)=(x(t),y(t))$. Then $xy'-x'y=1$. 

The osculating ellipse at a point $(x,y)$ satisfies the equations
$$
ax^2+2bxy+cy^2=1,\  (ax+by,bx+cy)\cdot (x',y')=0.
$$ 
Taking $ac-b^2=1$ into account, one solves these equations to obtain
$$
a=y^2+y'^2,\ b=-(xy+x'y'),\ c=x^2+x'^2.
$$ 
This is the equation of $\g^*$.

Next, $x''=px, y''=py$. Then 
$$
(\g^*)'= (1+p) (2yy',-(x'y+xy'),2xx'),
$$
and $|(\g^*)'|=|1+p|$, as claimed.
\end{proof}

Let $k$ be curvature of the curve $\g^*$.

\begin{lemma} \label{lm:hcurv}
One has
$$
k=\frac{1-p}{1+p}\ \ {\rm or}\ \ (1+p)(1+k)=2.
$$
\end{lemma}

For example, when $\g$ is a unit central ellipse with $p=-1$, the dual curve is a point, and the formula accordingly gives $k=\infty$. If $\g$ is a unit central hyperbola with $p=1$, then the formula gives $k=0$. Indeed, Lemma \ref{lm:duality} implies that $\g^*$ is a straight line, the intersection of the pseudo-sphere with the 2-dimensional subspace orthogonal to the vector corresponding to this hyperbola. 

\begin{proof}
Let $\tau$ be the arc length parameter on $\g^*$. Then $dt/d\tau=1/(1+p)$.

The curvature is the magnitude of the projection of the vector $d^2\g^*/d\tau^2$ on the pseudosphere. 
If $u$ is a position vector of a point of the pseudo-sphere and $v$ is a vector with foot point $u$, then the projection of $u$ is given by $u+(u\cdot v)v$.

From the previous lemma, we know that
$$
\frac{d\g^*}{d\tau} = (2yy',-(x'y+xy'),2xx'),
$$
hence
$$
\frac{d^2\g^*}{d\tau^2} = \frac{1}{1+p} (2yy',-(x'y+xy'),2xx')' = \frac{2}{1+p} (py^2+y'^2,-pxy-x'y',px^2+x'^2).
$$
Next, 
$$
\frac{d\g^*}{d\tau}\cdot \g^*=0 \Rightarrow \frac{d^2\g^*}{d\tau^2}\cdot \g^* + \frac{d\g^*}{d\tau}\cdot \frac{d\g^*}{d\tau}=0 \Rightarrow \frac{d^2\g^*}{d\tau^2}\cdot \g^*=-1,
$$
therefore the projection of $d^2\g^*/d\tau^2$ on the pseudosphere is
\begin{equation*}
\begin{split}
&\frac{d^2\g^*}{d\tau^2} - \g^* = \frac{2}{1+p} (py^2+y'^2,-pxy-x'y',px^2+x'^2) -\\ 
&(y^2+y'^2,-(xy+x'y'),x^2+x'^2) = \frac{1-p}{1+p} (y'^2-y^2,xy-x'y',x'^2-x^2),
\end{split}
\end{equation*}
and it remains to notice that the vector in the parentheses is unit.
\end{proof}

\begin{remark}
{\rm
According to a theorem of E. Ghys, see \cite{OT97}, the potential $p(t)$ of the curve $\g$ assumes the value -1 at least four times on the period $[0,\pi)$. It follows that the curve $\g^*$ has at least four cusps; in particular, it cannot be smooth.
}
\end{remark}

\section{Appendix B: Weierstrass  elliptic functions }\label{app:elliptic}

These are meromorphic functions $\wp, \zeta, \sigma:\C\to\CP^1$, defined for each rank 2  lattice $\Lambda=\Z2\omega+\Z2\omega', $
where $\omega, \omega'\in\C^*$, $\omega'/\omega\not\in\R$.  Define also $\Lambda'=\Lambda\setminus 0$.

\mn {\bf Alternative (useful) notation}  :  $\omega_1:=\omega, \ \omega_2:=-(\omega+\omega'), \ \omega_3=\omega', $ so $\sum\omega_i=0$.

\subsection{ The  $\wp$-function}
{\bf Definition:}

\begin{itemize}
\item Infinite sum 
$$ \wp(z):={1\over z^2}+\sum _{\lambda\in\Lambda'}\left[{1\over  (z+\lambda)^2}-{1\over \lambda^2}\right].
$$

\item ODE: 
$$(\wp')^2=4(\wp-e_1)(\wp-e_2)(\wp-e_3)=4\wp^3-g_2\wp-g_3,
$$
 so $e_1+e_2+e_3=0.$
 
 \item Integral formula. Let $\Sigma\subset \CP^2$ be the Riemann surface given in affine coordinates $(x:y:1)$ by $y^2=4(x-e_1)(x-e_2)(x-e_3)$. 
Then $z\mapsto (\wp(z), \wp'(z))$ defines a biholomorphism $\C/\Lambda\simeq \Sigma$. The inverse $\Sigma\to \C/\Lambda$ is given by 
$$ (x,y)\mapsto \int_x^\infty{dx\over y}\ \mod \Lambda.
$$
The integral does not depend, mod $\Lambda$, on the integration path.

\end{itemize}
{\bf  Properties:} meromorphic, even, $\Lambda$-periodic,  defining a double cover $\C/\G\to\CP^1$, branched over 4  pts, 
$$
\wp(0)= \infty,\ \wp(\omega)= e_1,\ \wp(\omega+\omega')=e_2,\ \wp(\omega')= e_3. 
$$
Alternatively, $\wp(\omega_i)=e_i,$ $ i=1,2,3,$ $\sum e_i=0.$

At these 4 branch points, 
$\wp'=0.$ In particular, the poles of $\wp$ occur at $\Lambda$ and are of order 2.


\subsection{The  $\zeta$ function}
\mn{\bf Definition:}
\begin{itemize}
\item Infinite sum: 
$$ \zeta(z):={1\over z}+\sum _{\lambda\in\Lambda'}\left[{1\over z+\lambda}-{1\over \lambda}+{z\over \lambda^2}\right].
$$

\item  ODE: 
$$\zeta'(z)=-\wp(z),\quad \zeta= {1\over z} + \mbox{ holomorphic function,    near } z=0.
$$

\item Integral formula: 
$$\zeta(z)={1\over z}-\int_0^z\left(\wp(u)-{1\over u^2}\right)du.
$$

\end{itemize}

\mn{\bf Properties:} odd, meromorphic,  simple poles at $\Lambda$,  $\Lambda$-quasi-periodic (\cite[p. 35]{Ak}):
$$\zeta(z+2\omega_i)=\zeta(z)+2\eta_i, \mbox{ where }\eta_i:=\zeta(\omega_i), \ i=1,2,3. 
$$

\mn{\bf Important relation:}
$$\eta\omega'-\eta'\omega={i\pi\over 2}\quad \mbox{if }\  \ 
\Im\left(\omega'/\omega\right)>0.
$$

\subsection  {The  $\sigma$ function}

\mn{\bf  Definition:} 
\begin{itemize}
\item Infinite product
$$\sigma(z)=z\prod_{\lambda\in\Lambda'}\left(1-{z\over \lambda}\right)\exp\left({{z\over \lambda}+{z^2\over 2\lambda^2}}\right).
$$

 \item ODE
 $${\sigma'\over \sigma}=\zeta.
 $$
 \end{itemize}
\mn{\bf Properties:} entire, quasi periodic \cite[p. 37]{Ak}:
$$\sigma(z+2\omega_i)=-e^{2\eta_i(z+\omega_i)}\sigma(z),
$$
where $\omega_1=\omega,\ \omega_2=-(\omega+\omega'), \ \omega_3=\omega', \eta_i=\zeta(\omega_i),$  so that $\sum \omega_i=\sum\eta_i=0.$

\subsection{ Addition formulas}
Express the relations between $\wp,\zeta,\sigma$ at $u\pm v,u,v$ \cite[p. 271]{Ak}:
\begin{align*}
	\wp(u)-\wp(v)&=-{\sigma(u-v)\sigma(u+v) \over \sigma^2(u)\sigma^2(v)},\\
\wp(u+v)+\wp(u)+\wp(v)&={1\over 4}\left[{\wp'(u)-\wp'(v)\over \wp(u)-\wp(v)}\right]^2,\\
\zeta(u+v)-\zeta(u)-\zeta(v)&={1\over 2}{\wp'(u)-\wp'(v)\over \wp(u)-\wp(v)}.
\end{align*}

\subsection{Reality condition} See \cite[page 104]{Ak}. If $g_2, g_3\in\R$ then $4x^3-g_2x-g_3=0$ has at least 1 real root. We want $\wp$ to be oscillating, that is, bounded, so we better have 3 real roots (in case of multiple roots $\wp$ is not doubly periodic, that is, not elliptic). In this case $e_1>e_2>e_3$, $\omega\in\R,$ $\omega'\in i\R$. This is proved by showing 
$$
\omega=\int_{e_1}^\infty {dx\over y}, \qquad \omega'=i\int_{-e_3}^\infty {dx\over y}, \qquad y^2=4x^3-g_2x-g_3=4\prod(x-e_i). 
$$

Also,  $\wp$ maps 
$$
(\infty,\omega']\mapsto(\infty,e_3],\ [\omega', \omega+\omega']\mapsto [e_3,e_2], \ [ \omega+\omega', \omega]\mapsto[e_2,e_1].
$$ 
See Figure \ref{fig:wp}. So $t\mapsto \wp(\omega'+t)$ describes a particle bouncing back and fourth along $[e_3,e_2].$

\begin{figure}[ht]
\centering
\includegraphics[width=.9\textwidth]{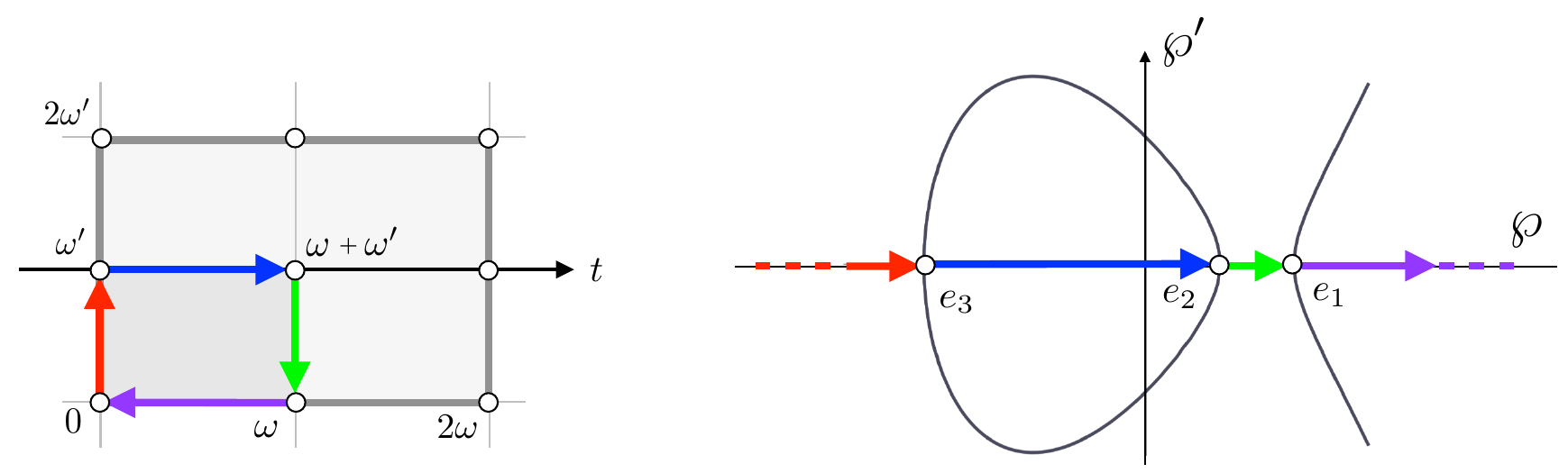}
\caption{
The Weierstrass $\wp$ function with real invariants $g_2, g_3$ and 3 real roots $e_i$: its fundamental rectangle (left) and the phase plane of  $(\wp')^2=4\wp^3-g_2\wp-g_3=4(\wp-e_1)(\wp-e_2)(\wp-e_3)$ (right). It maps $(0,\omega', \omega+\omega',\omega)\mapsto (\infty,  e_3,e_2,e_1),$ and the horizontal axis $\{\omega'+t \ | \ t\in\R\}$, $2\omega$-periodically,  onto the segment $[e_3,e_2]$. }
\label{fig:wp}
\end{figure} 

\subsection{  The Lam\'e equation} This has the form $X''=(A\wp(z)+B)X$ for some constants $A,B$. When $A=n(n+1)$ all solutions are meromorphic \cite[p.\,184]{Ak}. By a theorem of Picard  \cite[Equation (6), p.\,182-3]{Ak}, there is then a basis of solutions which are $\Lambda$-quasi-periodic (classically, ``doubly periodic of the 2nd kind"). That is, $X(z+2\omega)=\mu X(z),\   X(z+2\omega')=\mu'  X(z) .$ In our case, $n=1$:
$$X''=(2\wp(z)+B)X,$$ and such a basis is 
$$X_\pm(z)=e^{-z\zeta(\pm a)}{\sigma(z\pm a)\over \sigma(z)}, \qquad \wp(a)=B.$$
These two solutions are linearly independent if $B\neq e_i$. The associated multipliers are 
$$\mu_\pm=e^{\pm2[a\eta-\omega\zeta(a)]}, \quad \mu'_\pm=e^{\pm2[a\eta'-\omega'\zeta(a)]}, 
$$
where $\wp(a)=B,\ 
\eta=\zeta(\omega), \  \eta'=\zeta(\omega).$

\end{document}